\documentclass[reqno]{amsart}
\usepackage{amssymb}
\usepackage{amsfonts}
\usepackage{amsmath}
\usepackage{color}
\usepackage{amsaddr}

\newtheorem{theorem}{Theorem}[section]

\newtheorem{corollary}[theorem]{Corollary}

\newtheorem{definition}[theorem]{Definition}

\newtheorem{lemma}[theorem]{Lemma}

\newtheorem{proposition}[theorem]{Proposition}

\theoremstyle{remark}
\newtheorem{remark}[theorem]{Remark}

\numberwithin{equation}{section}

\begin{document}
\title[Boundary conditions with memory]{Coleman-Gurtin type equations with
dynamic boundary conditions}
\author[C. G. Gal and J. L. Shomberg]{Ciprian G. Gal$^1$ and Joseph L.
Shomberg$^2$}
\address{$^1$Department of Mathematics, Florida International University,
Miami, FL 33199, USA, \\
\texttt{cgal@fiu.edu}}
\address{$^2$Department of Mathematics and Computer Science, Providence
College, Providence, RI 02918, USA, \\
\texttt{{jshomber@providence.edu} }}
\date{\today }

\begin{abstract}
We present a new formulation and generalization of the classical theory of
heat conduction with or without fading memory which includes the usual heat
equation subject to a dynamic boundary condition as a special case. We
investigate the well-posedness of systems which consist of Coleman-Gurtin
type equations subject to dynamic boundary conditions, also with memory.
Nonlinear terms are defined on the interior of the domain and on the
boundary and subject to either classical dissipation assumptions, or to a
nonlinear balance condition in the sense of \cite{Gal12-2}. Additionally, we
do not assume that the interior and the boundary share the same memory
kernel.
\end{abstract}

\maketitle
\tableofcontents



\section{Introduction}

In recent years there has been an explosive growth in theoretical results
concerning dissipative infinite-dimensional systems with memory including
models arising in the theory of heat conduction in special materials and the
theory of phase-transitions. The mathematical and physical literature,
concerned primarily with qualitative/quantitative properties of solutions to
these models, is quite extensive and much of the work before 2002 is largely
referenced in the survey paper by Grasselli and Pata \cite{Grasselli}. More
recent results and updates can be found in \cite{CPS05, CPS06, CPGP, FGP10}
(cf. also \cite{GPM98, GPM00}). A basic evolution equation considered in
these references is that for an homogeneous and isotropic heat conductor
occupying a $d$-dimensional (bounded) domain $\Omega $ with sufficiently
smooth boundary $\Gamma =\partial \Omega $ and reads%
\begin{equation}
\partial _{t}u-\omega \Delta u-\left( 1-\omega \right) \int_{0}^{\infty
}k_{\Omega }\left( s\right) \Delta u\left( x,t-s\right) ds+f\left( u\right)
=0,  \label{heat-m1}
\end{equation}%
in $\Omega \times \left( 0,\infty \right) .$ Here $u=u\left( t\right) $ is
the (absolute) temperature distribution, $\omega >0,$ $r=-f\left( u\left(
t\right) \right) $ is a temperature dependent heat supply, and $k_{\Omega
}:[0,\infty )\rightarrow \mathbb{R}$ is a continuous nonnegative function,
smooth on $(0,\infty )$ and vanishing at infinity, and summable. As usual, (%
\ref{heat-m1}) is derived by assuming the following energy balance equation%
\begin{equation*}
\partial _{t}e+\text{div}\left( q\right) =r
\end{equation*}%
by considering the following relationships:%
\begin{equation}
e=e_{\infty }+c_{0}u,\text{ }q=-\omega \nabla u-\left( 1-\omega \right)
\int_{0}^{\infty }k_{\Omega }\left( s\right) \nabla u\left( x,t-s\right) ds,
\label{heat-m2}
\end{equation}%
for some constants $e_{\infty },c_{0}>0$. Equation (\ref{heat-m1}) is always
subject to either homogeneous Dirichlet ($u=0$) or Neumann boundary
conditions ($\partial _{n}u=0$) on $\Gamma \times \left( 0,\infty \right) $.
The first one asserts that the temperature is kept constant and close to a
given reference temperature at $\Gamma $ for all time $t>0$, while the
second \textquotedblleft roughly\textquotedblright\ states that the system
is thermally isolated from outside interference. This equation is also
usually supplemented by the \textquotedblleft initial\textquotedblright\
condition $\widetilde{u}:(-\infty ,0]\rightarrow \mathbb{R}$ such that%
\begin{equation}
u_{\mid t\in (-\infty ,0]}=\widetilde{u}\text{ in }\Omega .  \label{ic-m2}
\end{equation}%
These choices of boundary conditions, although help simplify substantially
the mathematical analysis of (\ref{heat-m1})-(\ref{ic-m2}), are actually
debatable in practice since in many such systems it is usually difficult, if
not impossible, to keep the temperature constant at $\Gamma $ for all
positive times without exerting some additional kind of control at $\Gamma $
for $t>0$. A matter of principle also arises for thermally isolated systems
in which, in fact, the correct physical boundary condition for (\ref{heat-m1}%
) turns out to be the following%
\begin{equation}
q\cdot n=\omega \partial _{n}u+\left( 1-\omega \right) \int_{0}^{\infty
}k_{\Omega }\left( s\right) \partial _{n}u\left( x,t-s\right) ds=0\text{ on }%
\Gamma \times \left( 0,\infty \right) \text{,}  \label{heat-m3}
\end{equation}%
see, for instance, \cite[Section 6]{CM1963}. Indeed, the condition $\partial
_{n}u=0$ on $\Gamma \times \left( 0,\infty \right) $ implies (\ref{heat-m3}%
), say when $u$ is a sufficiently smooth solution of (\ref{heat-m1})-(\ref%
{ic-m2}), but clearly the converse cannot hold in general.

In the classical theory of heat conduction, it is common to model a wide
range of diffusive phenomena including heat propagation in homogeneous
isotropic conductors, but generally it is assumed, as above, that surface
(i.e., boundary) conditions are completely static or stationary. In some
important cases this perspective neglects the contribution of boundary
sources to the total heat content of the conductor. A first step to remedy
this situation was done in Goldstein \cite{Gold06} for heat equations. The
approach presented there introduces dynamic boundary conditions into an
\emph{ad hoc} fashion and lacks some rigor in the case of reaction-diffusion
equations. In the next section of the paper we will make use of the usual
physical principles and present a new formulation and generalization of the
classical theory. Our general approach follows that of Coleman and Mizel
\cite{CM1963} which regards the second law of thermodynamics as included
among the laws of physics and which is compatible with the principle of
equipresence in the sense of Truesdell and Toupin (see Section \ref{dme}).
Thus, this new formulation is expected to give a solid foundation to the
arguments employed in derivations of the heat equation with
\textquotedblleft dynamic\textquotedblright\ boundary conditions developed
in Goldstein \cite{Gold06}, or in models for phase transitions developed in
Gal and Grasselli \cite{Gal&Grasselli08, GGM08}. Accounting for the presence
of boundary sources, the new formulation naturally leads to dynamic boundary
conditions for the temperature function $u$ and that contain the above
static conditions (especially, (\ref{heat-m3})) as special cases (see
Section \ref{dme}). In particular, we derive on $\Gamma \times \left(
0,\infty \right) ,$ the following boundary condition for (\ref{heat-m1}):%
\begin{align}
& \partial _{t}u-\nu \Delta _{\Gamma }u+\omega \partial _{n}u+g\left(
u\right)   \notag \\
& +\left( 1-\omega \right) \int_{0}^{\infty }k_{\Omega }\left( s\right)
\partial _{n}u\left( x,t-s\right) ds+\left( 1-\nu \right) \int_{0}^{\infty
}k_{\Gamma }\left( s\right) \left( -\Delta _{\Gamma }+\beta \right) u\left(
x,t-s\right) ds  \label{heat-m4} \\
& =0,  \notag
\end{align}%
for some $\nu \in \left( 0,1\right) $ and $\beta >0$. Here $k_{\Gamma
}:[0,\infty )\rightarrow \mathbb{R}$ is also a smooth nonnegative, summable
function over $(0,\infty )$ such that $k_{\Gamma }$ is vanishing at
infinity. The last two boundary terms on the left-hand side of equation (\ref%
{heat-m4}) are due to contributions coming from a (linear) heat exchange
rate between the bulk $\Omega $ and the boundary $\Gamma $, and boundary
fluxes, respectively (cf. Section \ref{dme}).

Our goal in this paper is to extend the previous well-posedness results of
\cite{CPS05, CPS06, CPGP, FGP10, Grasselli, GPM98, GPM00} and \cite{Gal12-2,
Gal12-3, Gal&Warma10} in the following directions:

\begin{itemize}
\item by allowing general boundary processes take place also on $\Gamma $,
equation (\ref{heat-m1}) is now subject to boundary conditions of the form (%
\ref{heat-m4});

\item we consider more general functions $f,g\in C^{1}\left( \mathbb{R}%
\right) $ satisfying either classical dissipation assumptions, or more
generally, nonlinear balance conditions allowing for bad behavior of $f,g$
at infinity;

\item we develop a general framework allowing for both weak and smooth
initial data for (\ref{heat-m1}), (\ref{heat-m4}), and possibly \emph{%
different} memory functions $k_{\Omega },k_{\Gamma }.$

\item we extend a Galerkin approximation scheme whose explicit construction
is crucial for the existence of strong solutions.
\end{itemize}

The paper is organized as follows. In Section \ref{fs}, we provide the
functional setup. In Section \ref{vf}, we prove theorems concerning the
well-posedness of the system, based on (\ref{heat-m1}), (\ref{heat-m4}),
generated by the new formulation. In the subsequent section, we present a
rigorous formulation and examples in which (\ref{heat-m4}) naturally occurs
for (\ref{heat-m1}).

\section{Derivation of the model equations}

\label{dme}

To begin let us consider a bounded domain $\Omega \subset \mathbb{R}^{d}$
which is occupied by a rigid body. The region $\Omega $ is assumed to be
bounded by a smooth boundary $\Gamma :=\partial \Omega $ which is assumed to
be at least Lipschitz continuous. As usual, a thermodynamic process taking
place in $\Omega $ is defined by five basic functions, that is, the specific
internal energy $e_{\Omega }\left( x,t\right) $, the specific entropy $\eta
_{\Omega }=\eta _{\Omega }\left( x,t\right) $, the heat flux $q=q\left(
x,t\right) $, the absolute temperature $u=u\left( x,t\right) >0$ and the
heat supply $h_{\Omega }\left( x,t\right) \,$, absorbed by the material at $%
x\in \Omega ,$ and possibly furnished by the external world (i.e.,
thermodynamic processes that occur outside of $\Omega $). All these
quantities, defined per unit volume and unit time, are scalars except for $%
q\in \mathbb{R}^{d}$ which is a vector. The classical theory \cite{CG1967,
CM1963} of heat conduction in the body $\Omega $ ignores any heat
contribution which may be supplied from processes taking place on $\Gamma $
and, hence, this situation is never modelled by the theory. This is the case
in many applications, in calorimetry, which go back to problems that occur
as early as the mid 1950's, see \cite[Chapter I, Section 1.9, pg. 22-24]{CJ}%
. A typical example arises when a given body $\Omega $ is in perfect thermal
contact with a thin metal sheet, possibly of different material $\Gamma
=\partial \Omega $ completely insulating the body $\Omega $\ from contact
with, say, a well-stirred hot or cold fluid. The assumption made is that the
metal sheet $\Gamma $\ is sufficiently thin such that the temperature $%
v\left( t\right) $ at any point on $\Gamma $ is constant across its
thickness. Since the sheet $\Gamma $ is in contact with a fluid it will
either heat or cool the body $\Omega $ in which case the heat supplied to $%
\Omega $ is due to both $\Gamma $ and the body of fluid, not to mention the
fact that the temperature distribution in the sheet is also affected by heat
transfer between $\Gamma $ and the interior $\Omega $. Since the outershell $%
\Gamma $ is in perfect contact with the body $\Omega $, it is reasonable to
assume by continuity that the temperature distribution $u\left( t\right) $
in $\Omega ,$ in an infinitesimal layer near $\Gamma $ is equal to $v\left(
t\right) $, for all times $t>\delta $, that is, $u\left( t\right) _{\mid
\Gamma }=v\left( t\right) $ for all $t>\delta $; they need not, of course,
be equal at $t=\delta $, where $\delta $ is the (initial) starting time.
When $\rho _{1},$ $\rho _{2}$ correspond to the densities of $\Omega $ and $%
\Gamma $, respectively, and $c_{1},c_{2}$ denote the heat capacities of $%
\Omega $ and $\Gamma ,$ respectively, this example can be modelled by the
balance equation%
\begin{equation}
\rho _{1}c_{1}\partial _{t}u=-\text{div}\left( q\right) +h_{\Omega }\text{
in }\Omega \times \left( \delta ,\infty \right) ,  \label{toy1}
\end{equation}%
suitably coupled with an equation for $\Gamma $, by considering the heat
balance of an element of area of the sheet $\Gamma $, which is%
\begin{equation}
\rho _{2}c_{2}\partial _{t}u=q\cdot n-\text{div}_{\Gamma }\left( q_{\Gamma
}\right) +l_{\Gamma }\text{ in }\Gamma \times \left( \delta ,\infty \right) .
\label{toy2}
\end{equation}%
Here $n\in \mathbb{R}^{d}$ denotes the exterior unit normal vector to $%
\Gamma $, $l_{\Gamma }\left( x,t\right) $ is an external heat supply and $%
q_{\Gamma }$ is a tangential heat flux on $\Gamma $ while div$_{\Gamma }$ is
the surface divergence whose definition is given below. Note that the
correct term to couple the balance equations for $\Omega $ and $\Gamma $ is
given by $q\cdot n$, since this is used to quantify a (linear) heat exchange
rate across $\Gamma $ from $\Omega $ in all directions normal to the
boundary $\Gamma $. The system (\ref{toy1})-(\ref{toy2}) is also important
in control problems for the heat equation, say when a specific temperature
distribution at the boundary $\Gamma $ is desired (see \cite{HKR1}).

As mentioned earlier, in the classical theory on heat conduction one usually
ignores boundary contributions by either prescribing the temperature on $%
\Gamma $ or assuming that the flux across the surface $\Gamma $ from $\Omega
$ is null, or simply, by invoking Newton's law of cooling which states that
the flux across the surface is directly proportional to temperature
differences between the surface and the surrounding medium. In the sequel,
it is our goal to include general boundary processes into the classical
theory of heat conduction. To this end, in order to define a complete
thermodynamic process in $\overline{\Omega }=\Omega \cup \Gamma $, as in the
previous example, we need to add four more response functions, that is, the
specific surface energy $e_{\Gamma }\left( x,t\right) ,$ the specific
surface entropy density $\eta _{\Gamma }\left( x,t\right) $, the tangential
heat flux $q_{\Gamma }=q_{\Gamma }\left( x,t\right) \in \mathbb{R}^{d-1}$,
and the external heat supply $h_{\Gamma }\left( x,t\right) ,$ all defined
for $x\in \Gamma ,$ per unit area and unit time. It is assumed that the
absolute (local) temperature $u\left( \cdot ,t\right) $\ is sufficiently
smooth up to $\overline{\Omega }$ as a function of the spatial coordinate.
We now introduce the following definition.

\begin{itemize}
\item We say that the set of nine time-dependent variables constitutes a
\emph{complete} \emph{thermodynamic process }in $\overline{\Omega }$ if the
following conservation law holds, not only for $\overline{\Omega }$, but
also for any subdomain $\Omega _{0}\subset \Omega $ and any part $\Gamma
_{0}\subset \Gamma $:%
\begin{equation}
\int_{\Omega }\overset{\centerdot }{e}_{\Omega }dx+\int_{\Gamma }\overset{%
\centerdot }{e}_{\Gamma }d\sigma =-\int_{\Omega }\text{div}\left( q\right)
dx-\int_{\Gamma }\text{div}_{\Gamma }\left( q_{\Gamma }\right) d\sigma
+\int_{\Omega }h_{\Omega }dx+\int_{\Gamma }h_{\Gamma }d\sigma .
\label{1stlaw}
\end{equation}
\end{itemize}

\noindent In (\ref{1stlaw}), $dx$ denotes the volume element, $d\sigma $ is
the element of surface area and the superimposed dot denotes the
time-derivative. Note that in general, the external heat supply $h_{\Gamma }$
on $\Gamma $ must also depend, possibly in a nonlinear fashion, on the heat
content exchanged across $\Gamma $ from $\Omega $, i.e., $h_{\Gamma
}=f\left( q\cdot n\right) +l_{\Gamma }$, for some function $f$, where $%
l_{\Gamma }$ accounts either for the heat supply coming solely from $\Gamma $
or some other source outside of $\Gamma $, see the above example (\ref{toy1}%
)-(\ref{toy2}). In order to give a rigorous definition to div$_{\Gamma
}\left( q_{\Gamma }\right) ,$ we regard $\Gamma $ as a compact Riemanian
manifold without boundary, endowed with the natural metric inherited from $%
\mathbb{R}^{d}$, given in local coordinates by $\mathbf{\tau }$ and with
fundamental form $\left( \mathbf{\tau }_{ij}\right) _{i,j=1,...,d-1}$. A
scalar-valued function $w\in C^{\infty }\left( \Gamma \right) $ induces an
element of the dual space of $T_{x}\Gamma $ via the directional derivative
of tangential vectors at $x\in \Gamma $. Clearly, $T_{x}\Gamma $ is a
Hilbert space when endowed with scalar product induced from $\mathbb{R}^{d}$%
. For a tangential vector field $q_{\Gamma }\in C^{\infty }\left( \Gamma
\right) ,$ that is, $q_{\Gamma }\left( x\right) \in T_{x}\Gamma ,$ for $x\in
\Gamma ,$ the surface divergence, div$_{\Gamma }\left( q_{\Gamma }\right) ,$
is in the local coordinates $\mathbf{\tau }$ for $\Gamma ,$%
\begin{equation*}
\text{div}_{\Gamma }q_{\Gamma }\left( \mathbf{\tau }\right) =\frac{1}{\sqrt{%
\left\vert \mathbf{\tau }\right\vert }}\sum_{i=1}^{d-1}\partial _{i}(\sqrt{%
\left\vert \mathbf{\tau }\right\vert }q_{i}\left( \mathbf{\tau }\right) ),
\end{equation*}%
where $q_{i}$ are the components of $q_{\Gamma }$ with respect to the basis $%
\left\{ \partial _{1}\mathbf{\tau ,...,\partial }_{d-1}\mathbf{\tau }%
\right\} $ of $T_{x}\Gamma $ and $\left\vert \mathbf{\tau }\right\vert =\det
\left( \mathbf{\tau }_{ij}\right) $. Moreover, we can define the surface
gradient $\nabla _{\Gamma }u$ as a unique element of $T_{x}\Gamma $
corresponding to this dual space element via a natural isomorphism, that is,
\begin{equation*}
\nabla _{\Gamma }u\left( \mathbf{\tau }\right) =\sum_{i,j=1}^{d-1}\mathbf{%
\tau }_{ij}\partial _{j}u\left( \mathbf{\tau }\right) \partial _{i}\mathbf{%
\tau ,}
\end{equation*}%
with respect to the canonical basis $\left\{ \partial _{1}\mathbf{\tau
,...,\partial }_{d-1}\mathbf{\tau }\right\} $ of $T_{x}\Gamma $. For a
multi-index $\alpha \in \mathbb{N}_{0}^{m}$, the operator $\nabla _{\Gamma
}^{\alpha }u$ is defined by taking iteratively the components of $\nabla
_{\Gamma }u.$ It is worth emphasizing that our form of the first law (\ref%
{1stlaw}) is equivalent to%
\begin{equation}
\overset{\centerdot }{e}_{\Omega }=-\text{div}\left( q\right) +h_{\Omega }%
\text{ in \ }\Omega \text{, and }\overset{\centerdot }{e}_{\Gamma }=-\text{%
div}_{\Gamma }\left( q_{\Gamma }\right) +h_{\Gamma }\text{ on }\Gamma ,
\label{1stlaw-strong}
\end{equation}%
under suitable smoothness assumptions on the response functions involved in (%
\ref{1stlaw-strong}). Equation (\ref{1stlaw}) may be called the law of
conservation of total energy or the \emph{extended} First Law of
Thermodynamics. For each such complete thermodynamic process, let us define
the total rate of production of entropy in $\overline{\Omega }=\Omega \cup
\Gamma $ to be%
\begin{equation}
\Upsilon :=\int_{\Omega }\overset{\centerdot }{\eta }_{\Omega
}dx+\int_{\Gamma }\overset{\centerdot }{\eta }_{\Gamma }d\sigma
-\int_{\Omega }\frac{h_{\Omega }}{u}dx+\int_{\Omega }\text{div}\left( \frac{q%
}{u}\right) dx+\int_{\Gamma }\text{div}_{\Gamma }\left( \frac{q_{\Gamma }}{u}%
\right) d\sigma -\int_{\Gamma }\frac{h_{\Gamma }}{u}d\sigma ,  \label{2ndlaw}
\end{equation}%
where we regard $q/u$ as a vectorial flux of entropy in $\Omega $, $%
h_{\Omega }/u$ as a scalar supply of entropy produced by radiation from
inside the body $\Omega $, $h_{\Gamma }/u$ is viewed as a scalar supply of
entropy produced by radiation from $\Gamma $ and $q_{\Gamma }/u$ is a
tangential flux of entropy on $\Gamma $. More precisely, we define $\Upsilon
$ to be the difference between the total rate of change in entropy of $%
\overline{\Omega }$ and that rate of change which comes from the heat
supplies in both $\Omega $ and $\Gamma $, and both the inward and tangential
fluxes. We postulate the following extended version of the Second Law of
Thermodynamics as follows.

\begin{itemize}
\item For every complete thermodynamic process in $\overline{\Omega }$ the
inequality%
\begin{equation}
\Upsilon \geq 0  \label{2ndlawbis}
\end{equation}%
must hold for all $t$, not only in $\overline{\Omega }$, but also on all
subdomains $\Omega _{0}\subset \Omega $ and all parts $\Gamma _{0}\subset
\Gamma ,$ respectively\footnote{%
When (\ref{2ndlawbis}) holds on all parts $\Omega _{0}\subset \Omega ,$ it
is understood that all the boundary integrals in (\ref{2ndlaw}) drop out; in
the same fashion, when (\ref{2ndlawbis}) is satisfied for all parts $\Gamma
_{0}\subset \Gamma ,$ the bulk integrals are also omitted from the
definition of $\Upsilon .$}. For obvious reasons, we will refer to the
inequality $\Upsilon \geq 0$\ as the \emph{extended} Clausius-Duhem
inequality. Finally, a complete thermodynamic process is said to be \emph{%
admissible} in $\overline{\Omega }$ if it is compatible with a set of
constitutive conditions given on the response functions introduced above, at
each point of $\overline{\Omega }$\ and at all times $t$.
\end{itemize}

\noindent Of course, for the postulate to hold, the various response
functions must obey some restrictions, including the usual ones which are
consequences of the \emph{classical} Clausius-Duhem inequality. In
particular, the entropy $\eta _{\Omega }$ at each point $x\in \Omega $ must
be determined only by a function of the specific internal energy $e_{\Omega
},$ and the temperature $u$ at $x\in \Omega $ is determined only by a
relation involving $e_{\Omega }$ and $\eta _{\Omega }$. More precisely, it
turns out that for the postulate to hold on any $\Omega _{0}\subset \Omega $%
, both the internal energy $e_{\Omega }$ and the entropy function $\eta
_{\Omega }$ must be constitutively independent of any higher-order stress
tensors $\nabla ^{\gamma }u$ for any $\gamma \geq 1$, such that they are
only functions of the local temperature, i.e., it follows that%
\begin{equation}
e_{\Omega }=e_{\Omega }\left( u\right) \text{ and }\eta _{\Omega }=\eta
_{\Omega }\left( u\right) ,  \label{c-relation1}
\end{equation}%
respectively, cf. \cite[Theorem 1, pg. 251]{CM1963}. Indeed, our postulate
implies that the local form of the second law must hold also on any
subdomain $\Omega _{0}$ of $\Omega $; this implies that%
\begin{equation}
\gamma _{\Omega }:=\left( \overset{\centerdot }{\eta }_{\Omega }-\frac{%
h_{\Omega }}{u}+\text{div}\left( \frac{q}{u}\right) \right) \geq 0\text{ in }%
\Omega  \label{2ndlaw-conseq}
\end{equation}%
and%
\begin{equation}
\gamma _{\Gamma }:=\left( \overset{\centerdot }{\eta }_{\Gamma }-\frac{%
h_{\Gamma }}{u}+\text{div}_{\Gamma }\left( \frac{q_{\Gamma }}{u}\right)
\right) \geq 0\text{ on }\Gamma .  \label{2ndlaw-conseq2}
\end{equation}%
From \cite{CM1963}, we know that $\gamma _{\Omega }\geq 0$ in the body $%
\Omega $ if and only if%
\begin{equation}
q\cdot \nabla u\leq 0,  \label{cond-ineq-body}
\end{equation}%
for all values $u,$ $\nabla u,$...., $\nabla ^{\gamma }u$, with $q=q\left(
u,\nabla u,\nabla ^{2}u,...,\nabla ^{\gamma }u\right) $. This inequality is
called the heat conduction inequality in $\Omega $. In fact, this inequality
was established in \cite{G1965} under more general constitutive assumptions
on $\eta _{\Omega },q$ and $e_{\Omega }$, excluding memory effects, as
functionals of the entropy field over the entire body $\Omega $ at the same
time.

We now find necessary and sufficient set of restrictions on the remaining
functions $\eta _{\Gamma },$ $e_{\Gamma }$, $q_{\Gamma }$. As in \cite%
{CM1963}, we assume a formulation of constitutive equations to be compatible
with the Principle of Equipresence in the sense of Truesdell and Toupin \cite%
[pg. 293]{TT1960}, which basically states that \textquotedblleft \textit{a
variable present as an independent variable in one constitutive equation
should be so present in all}\textquotedblright . In the present formulation,
the material at $x\in \Gamma $ is characterized by the response functions $%
\widehat{\eta }_{\Gamma },$ $\widehat{e}_{\Gamma }$ and $\widehat{q}_{\Gamma
},$ which give the functions $\eta _{\Gamma }\left( x,t\right) ,$ $e_{\Gamma
}\left( x,t\right) $ and $q_{\Gamma }\left( x,t\right) $, respectively, when
the values $\nabla _{\Gamma }^{j}u\left( x,t\right) $ are known for $%
j=0,1,2,...,\alpha .$ Dropping the hats for the sake of convenience and by
force of this principle, we assume that%
\begin{align}
e_{\Gamma }& =e_{\Gamma }\left( u,\nabla _{\Gamma }u,\nabla _{\Gamma
}^{2}u,...,\nabla _{\Gamma }^{\alpha }u\right) ,  \label{c1} \\
\eta _{\Gamma }& =\eta _{\Gamma }\left( u,\nabla _{\Gamma }u,\nabla _{\Gamma
}^{2}u,...,\nabla _{\Gamma }^{\alpha }u\right) ,  \label{c2} \\
q_{\Gamma }& =q_{\Gamma }\left( u,\nabla _{\Gamma }u,\nabla _{\Gamma
}^{2}u,...,\nabla _{\Gamma }^{\alpha }u\right) .  \label{c3}
\end{align}%
Furthermore, we assume that for any fixed values of $\nabla _{\Gamma }^{j}u,$
the response function $e_{\Gamma }$ is smooth in the first variable $u$,
i.e., we suppose $\frac{\partial e_{\Gamma }}{\partial u}\left( u,\nabla
_{\Gamma }u,\nabla _{\Gamma }^{2}u,...,\nabla _{\Gamma }^{\alpha }u\right)
\neq 0.$ This implies that there exist new response functions, say $%
\widetilde{\eta }_{\Gamma },$ $\widetilde{e}_{\Gamma }$ and $\widetilde{q}%
_{\Gamma },$ which can be used to write (\ref{c1})-(\ref{c3}) in the
following form:%
\begin{align}
u& =\widetilde{u}\left( e_{\Gamma },\nabla _{\Gamma }u,\nabla _{\Gamma
}^{2}u,...,\nabla _{\Gamma }^{\alpha }u\right) ,  \label{c1bis} \\
\eta _{\Gamma }& =\widetilde{\eta }_{\Gamma }\left( e_{\Gamma },\nabla
_{\Gamma }u,\nabla _{\Gamma }^{2}u,...,\nabla _{\Gamma }^{\alpha }u\right) ,
\label{c2bis} \\
q_{\Gamma }& =\widetilde{q}_{\Gamma }\left( e_{\Gamma },\nabla _{\Gamma
}u,\nabla _{\Gamma }^{2}u,...,\nabla _{\Gamma }^{\alpha }u\right) .
\label{c3bis}
\end{align}%
For each fixed values of the tensors $\nabla _{\Gamma }^{j}u,$ $%
j=0,1,2,...,\alpha $, the variable $\widetilde{u}\left( \cdot ,\nabla
_{\Gamma }u,\nabla _{\Gamma }^{2}u,...,\nabla _{\Gamma }^{\alpha }u\right) $
is determined through the inverse function of $e_{\Gamma },$ given by (\ref%
{c1}), such that $\widetilde{\eta }_{\Gamma }$ and $\widetilde{q}_{\Gamma }$
are defined by%
\begin{align*}
\widetilde{\eta }_{\Gamma }\left( e_{\Gamma },\nabla _{\Gamma }u,\nabla
_{\Gamma }^{2}u,...,\nabla _{\Gamma }^{\alpha }u\right) & =\eta _{\Gamma
}\left( \widetilde{u}\left( e_{\Gamma },\nabla _{\Gamma }u,\nabla _{\Gamma
}^{2}u,...,\nabla _{\Gamma }^{\alpha }u\right) ,\nabla _{\Gamma }u,\nabla
_{\Gamma }^{2}u,...,\nabla _{\Gamma }^{\alpha }u\right) , \\
\widetilde{q}_{\Gamma }\left( e_{\Gamma },\nabla _{\Gamma }u,\nabla _{\Gamma
}^{2}u,...,\nabla _{\Gamma }^{\alpha }u\right) & =q_{\Gamma }\left(
\widetilde{u}\left( e_{\Gamma },\nabla _{\Gamma }u,\nabla _{\Gamma
}^{2}u,...,\nabla _{\Gamma }^{\alpha }u\right) ,\nabla _{\Gamma }u,\nabla
_{\Gamma }^{2}u,...,\nabla _{\Gamma }^{\alpha }u\right) .
\end{align*}%
Note that with $u\left( x,t\right) $ specified for all $x$ and $t$,
equations (\ref{c1})-(\ref{c3}) give $\eta _{\Gamma }\left( x,t\right) ,$ $%
e_{\Gamma }\left( x,t\right) $ and $q_{\Gamma }\left( x,t\right) ,$ for all $%
x$ and $t,$ in which case the local form of the First Law (see also (\ref%
{1stlaw-strong})) determines also $h_{\Gamma }$. In particular, every
temperature distribution $u\left( x,t\right) >0$ with $x$ varying over $%
\Gamma $, determines a unique complete thermodynamic process in $\Gamma $.
By a standard argument in \cite[pg. 249]{CM1963}, in (\ref{c1})-(\ref{c3})
we may regard not only $e_{\Gamma },\nabla _{\Gamma }u,\nabla _{\Gamma
}^{2}u,...,\nabla _{\Gamma }^{\alpha }u$ as independent variables, but also
their time-derivatives $\overset{\centerdot }{e}_{\Gamma }$, $\nabla
_{\Gamma }\overset{\centerdot }{u},$ $\nabla _{\Gamma }^{2}\overset{%
\centerdot }{u},...,$ $\nabla _{\Gamma }^{\alpha }\overset{\centerdot }{u},$
to form a set of quantities which can be chosen independently at one fixed
point $x\in \Gamma $ and time.

For each complete thermodynamic process in $\overline{\Omega }$, the second
energy balance equation (\ref{1stlaw-strong}) allows us to write (\ref%
{2ndlaw-conseq2}) as%
\begin{equation}
\gamma _{\Gamma }=\overset{\centerdot }{\eta }_{\Gamma }-\frac{h_{\Gamma }}{u%
}+\text{div}_{\Gamma }\left( \frac{q_{\Gamma }}{u}\right) =\overset{%
\centerdot }{\eta }_{\Gamma }-\frac{e_{\Gamma }}{u}+q_{\Gamma }\cdot \nabla
_{\Gamma }\left( \frac{1}{u}\right) .  \label{2ndlaw-conseq3}
\end{equation}%
Since $q_{\Gamma }$ and $\eta _{\Gamma }$ must be given by (\ref{c3bis}) and
(\ref{c2bis}), at any point $\left( x,t\right) ,$ we have%
\begin{equation*}
\overset{\centerdot }{\eta }_{\Gamma }=\frac{\partial \widetilde{\eta }%
_{\Gamma }}{\partial e_{\Gamma }}\overset{\centerdot }{e}_{\Gamma
}+\sum\nolimits_{j=1}^{\alpha }\left( \frac{\partial \widetilde{\eta }%
_{\Gamma }}{\partial u_{,l_{1}l_{2}...l_{j}}}\right) \overset{\centerdot }{u}%
_{,l_{1}l_{2}...l_{j}},
\end{equation*}%
where the summation convention is used and where in local coordinates of $%
\Gamma $, $u_{,l_{1}l_{2}...l_{j}}=(\nabla _{\Gamma
}^{j}u)_{l_{1}l_{2}...l_{j}}$. It follows that%
\begin{equation}
\gamma _{\Gamma }=\left( \frac{\partial \widetilde{\eta }_{\Gamma }}{%
\partial e_{\Gamma }}-\frac{1}{\widetilde{u}}\right) \overset{\centerdot }{e}%
_{\Gamma }+\sum\nolimits_{j=1}^{\alpha }\left( \frac{\partial \widetilde{%
\eta }_{\Gamma }}{\partial u_{,l_{1}l_{2}...l_{j}}}\right) \overset{%
\centerdot }{u}_{,l_{1}l_{2}...l_{j}}-\frac{1}{u^{2}}\widetilde{q}_{\Gamma
}\cdot \nabla _{\Gamma }u.  \label{c4}
\end{equation}%
In order for $\gamma _{\Gamma }\geq 0$ to hold on $\Gamma $ (but also on all
parts $\Gamma _{0}\subset \Gamma $), according to (\ref{2ndlaw-conseq2}) and
our postulate, it is necessary and sufficient that%
\begin{equation}
\frac{\partial \widetilde{\eta }_{\Gamma }}{\partial e_{\Gamma }}=\frac{1}{%
\widetilde{u}},\text{ and }\frac{\partial \widetilde{\eta }_{\Gamma }}{%
\partial u_{,l_{1}l_{2}...l_{j}}}=0,\text{ }j=1,2,...,\alpha .  \label{c5}
\end{equation}%
It follows from (\ref{c5}) that the functions $\widetilde{\eta }_{\Gamma }$
and $\widetilde{u}_{\Gamma }$ from (\ref{c1bis})-(\ref{c2bis}) cannot depend
on $\nabla _{\Gamma }u$, $\nabla _{\Gamma }^{2}u$, $...$, $\nabla _{\Gamma
}^{\alpha }u$, and they must reduce to functions of the scalar variable $%
e_{\Gamma }$ only, i.e., $\eta _{\Gamma }=\widetilde{\eta }_{\Gamma }\left(
e_{\Gamma }\right) ,$ $u=\widetilde{u}_{\Gamma }\left( e_{\Gamma }\right) $.
These function must also obey the first equation of (\ref{c5}); hence, the
variables $\nabla _{\Gamma }u$, $\nabla _{\Gamma }^{2}u$, $...$, $\nabla
_{\Gamma }^{\alpha }u$ must also be dropped out of equations (\ref{c1bis})
and (\ref{c2bis}) to get%
\begin{equation}
e_{\Gamma }=e_{\Gamma }\left( u\right) \text{ and }\eta _{\Gamma }=\eta
_{\Gamma }\left( u\right) .  \label{c-relation2}
\end{equation}%
Consequently, with this reduction we observe that (\ref{c4}) becomes%
\begin{equation*}
\gamma _{\Gamma }=-\frac{1}{u^{2}}\widetilde{q}_{\Gamma }\cdot \nabla
_{\Gamma }u,
\end{equation*}%
for all temperature fields $u>0$ and $q_{\Gamma }$ given by (\ref{c3bis}).
Thus, in order to have $\gamma _{\Gamma }\geq 0$ on $\Gamma $, it is
necessary and sufficient that $\widetilde{q}_{\Gamma }\cdot \nabla _{\Gamma
}u\leq 0,$ or equivalently,%
\begin{equation}
q_{\Gamma }\left( u,\nabla _{\Gamma }u,\nabla _{\Gamma }^{2}u,...,\nabla
_{\Gamma }^{\alpha }u\right) \cdot \nabla _{\Gamma }u\leq 0,  \label{hcb}
\end{equation}%
for all values $u$, $\nabla _{\Gamma }u$, $\nabla _{\Gamma }^{2}u$, $...,$ $%
\nabla _{\Gamma }^{\alpha }u.$ We call (\ref{hcb}) the heat conduction
inequality on $\Gamma $. Therefore, we have established that a necessary and
sufficient condition for the \emph{extended} Clausius-Duhem inequality to
hold for all complete thermodynamic processes on $\overline{\Omega }$ is
that both the conduction inequalities (\ref{cond-ineq-body})-(\ref{hcb}) in $%
\Omega $ and $\Gamma $, respectively, hold. An interesting consequence is
that the following choices $q=-\kappa _{\Omega }\left( u\right) \nabla u$
and $q_{\Gamma }=-\kappa _{\Gamma }\left( u\right) \nabla _{\Gamma }u,$
where $\kappa _{\Omega },\kappa _{\Gamma }>0$ are the thermal conductivity
of $\Omega $ and $\Gamma $, respectively, are covered by this theory. Such
choices were assumed by the theories developed in \cite{Gal&Grasselli08},
\cite{GGM08}, \cite{Gold06} for the system (\ref{toy1})-(\ref{toy2}).

Motivated by the above result, we now wish to investigate more general
constitutive conditions for the response functions involved in (\ref{2ndlaw}%
), by allowing them to depend also explicitly on histories up to time $t$ of
the temperature and/or the temperature gradients at $x$. Following the
approach of \cite{CG1967}, using the abbreviations $g_{\Omega }:=\nabla u$, $%
g_{\Gamma }:=\nabla _{\Gamma }u$, we consider a fixed point $x\in \overline{%
\Omega }$, and define the functions $u^{t}$, $g_{\Omega }^{t},$ $g_{\Gamma
}^{t}$ as the histories up to time $t$ of the temperature and the
temperature gradients at $x$. More precisely, we let%
\begin{equation*}
\left\{
\begin{array}{ll}
u^{t}\left( x,s\right) =u\left( x,t-s\right) , &  \\
\text{ }g_{\Omega }^{t}\left( x,s\right) =g_{\Omega }\left( x,t-s\right) &
\\
\text{ }g_{\Gamma }^{t}\left( x,s\right) =g_{\Gamma }\left( x,t-s\right) , &
\end{array}%
\right.
\end{equation*}%
for all $\,s\in \lbrack 0,\infty )$, on which these functions are
well-defined. For a complete thermodynamic process in $\overline{\Omega }$,
we define the following energy densities on $\Omega $ and $\Gamma $,
respectively, by%
\begin{equation}
\psi _{\Omega }:=e_{\Omega }-u\eta _{\Omega },\text{ }\psi _{\Gamma
}:=e_{\Gamma }-u\eta _{\Gamma }.  \label{energy2}
\end{equation}%
Of course, knowledge of $e_{\Omega },e_{\Gamma }$ and $\eta _{\Omega },\eta
_{\Gamma }$ obviously determine $\psi _{\Omega }$ and $\psi _{\Gamma }$ by
these relations. We now consider a new generalization of the constitutive
equations for (\ref{c-relation1}), (\ref{c-relation2}) and the bulk and
surface fluxes $q,$ $q_{\Gamma },$ respectively. We shall investigate the
implications that the second law (\ref{2ndlawbis}) has on these functions.
We assume that the material at $x\in \overline{\Omega }$ is characterized by
three constitutive functionals $P_{\Omega }$, $H_{\Omega }$ and $q$, in the
bulk $\Omega $, and three more constitutive functionals $P_{\Gamma },$ $%
H_{\Gamma }$ and $q_{\Gamma }$, on the surface $\Gamma $, which give the
present values of $\psi _{\Omega },\psi _{\Gamma },\eta _{\Omega },\eta
_{\Gamma },q$ and $q_{\Gamma }$ at any $x$, whenever the histories are
specified at $x$. Note that the restrictions of the functions $u^{t}$, $%
g_{\Omega }^{t},$ $g_{\Gamma }^{t}$ on the open interval $\left( 0,\infty
\right) $, denoted here by $u_{r}^{t}$, $g_{\Omega ,r}^{t},$ $g_{\Gamma
,r}^{t}$, are called past histories. Since a knowledge of the histories $%
\left( u^{t},g_{\Omega }^{t},g_{\Gamma }^{t}\right) $ is equivalent to a
knowledge of the past histories $\left( u_{r}^{t},g_{\Omega
,r}^{t},g_{\Gamma ,r}^{t}\right) ,$ and the present values $u^{t}\left(
0\right) =u,$ $g_{\Omega }^{t}\left( 0\right) =g_{\Omega }\left( t\right) ,$
$g_{\Gamma }^{t}\left( 0\right) =g_{\Gamma }\left( t\right) ,$ it suffices
to consider%
\begin{equation}
\left\{
\begin{array}{ll}
\psi _{\Omega }=P_{\Omega }\left( u^{t},g_{\Omega }^{t}\right) , & \psi
_{\Gamma }=P_{\Gamma }\left( u^{t},g_{\Gamma }^{t}\right) , \\
\eta _{\Omega }=H_{\Omega }\left( u^{t},g_{\Omega }^{t}\right) , & \eta
_{\Gamma }=H_{\Gamma }\left( u^{t},g_{\Gamma }^{t}\right) , \\
q=q\left( u^{t},g_{\Omega }^{t}\right) , & q_{\Gamma }=q_{\Gamma }\left(
u^{t},g_{\Gamma }^{t}\right) ,%
\end{array}%
\right.  \label{c-relation3}
\end{equation}%
where the Principle of Equipresence is assumed in (\ref{c-relation3}). We
further suppose that all the functionals in (\ref{c-relation3}) obey the
principle of fading memory as formulated in \cite{Co1964} (cf. also \cite[%
Section 5]{G1965}). In particular, this assumption means that
\textquotedblleft \textit{deformations and temperatures experienced in the
distant past should have less effect on the present values of the entropies,
energies, stresses, and heat fluxes than deformations and temperatures which
occurred in the recent past}\textquotedblright . Such assumptions can be
made precise through the so-called \textquotedblleft
memory\textquotedblright\ functions $m_{\Omega },$ $m_{\Gamma },$ which
characterize the rate a\noindent t which the memory fades both in the body $%
\Omega $ and on the surface $\Gamma $, respectively. In particular, we may
assume that both functions $m_{S}\left( \cdot \right) ,$ $S\in \left\{
\Omega ,\Gamma \right\} $, are positive, continuous functions on $\left(
0,\infty \right) $ decaying sufficiently fast to zero as $s\rightarrow
\infty $. In this case, we let $D_{S}$ denote the common domain for the
functionals $P_{S},H_{S}$ and $q_{S}$ ($q_{\Omega }=q$)$,$ as the set of all
pairs $\left( u^{t},g_{S}^{t}\right) $ for which $u^{t}>0$ and $\left\Vert
\left( u^{t},g_{S}^{t}\right) \right\Vert <\infty $, where%
\begin{equation}
\left\Vert \left( u^{t},g_{S}^{t}\right) \right\Vert ^{2}:=\left\vert
u^{t}\left( 0\right) \right\vert ^{2}+\left\vert g_{S}^{t}\left( 0\right)
\right\vert ^{2}+\int_{0}^{\infty }\left\vert u^{t}\left( s\right)
\right\vert ^{2}m_{S}\left( s\right) ds+\int_{0}^{\infty }\left(
g_{S}^{t}\left( s\right) \cdot g_{S}^{t}\left( s\right) \right) m_{S}\left(
s\right) ds,  \label{norm}
\end{equation}%
and where $S\in \left\{ \Omega ,\Gamma \right\} $. Furthermore, for each $%
S\in \left\{ \Omega ,\Gamma \right\} $ we assume as in \cite{CG1967} that $%
P_{S}$, $H_{S}$, and $q_{S}$ ($q_{\Omega }=q$) are continuous over $D_{S}$
with respect to the norm (\ref{norm}), but also that $P_{S}$ is continuously
differentiable over $D_{S}$ in the sense of Fr\'{e}chet, and that the
corresponding functional derivatives are jointly continuous in their
arguments.

In order to observe the set of restrictions that the postulate (\ref%
{2ndlawbis}) puts on the response functions, we recall (\ref{1stlaw-strong})
and substitute (\ref{energy2}) into the local forms (\ref{2ndlaw-conseq}), (%
\ref{2ndlaw-conseq2}) to derive the following (local) forms of the extended
Clasius-Duhem inequality on $\overline{\Omega }$:%
\begin{equation}
\left\{
\begin{array}{ll}
\overset{\centerdot }{\psi }_{\Omega }+\overset{\centerdot }{u}\eta _{\Omega
}+\frac{1}{u}q_{\Omega }\cdot \nabla u\leq 0 & \text{in }\Omega , \\
\overset{\centerdot }{\psi }_{\Gamma }+\overset{\centerdot }{u}\eta _{\Gamma
}+\frac{1}{u}q_{\Gamma }\cdot \nabla _{\Gamma }u\leq 0 & \text{on }\Gamma .%
\end{array}%
\right.  \label{2ndlaw-conseq4}
\end{equation}%
We recall that a complete thermodynamic process is \emph{admissible} in $%
\overline{\Omega }$ if it is compatible with the set of constitutive
conditions given in (\ref{c-relation3}) at each point $x$\ and at all times $%
t$. Since we believe that our postulate (\ref{2ndlawbis}) \emph{should} hold
for all time-dependent variables compatible with the extended law of balance
of energy in (\ref{1stlaw}), it follows from \cite[Theorem 6]{CG1967} (cf.
also \cite[Section 6, Theorem 1]{Co1964}) that the Clausius-Duhem
inequalities (\ref{2ndlaw-conseq4}) imply for each $S\in \left\{ \Omega
,\Gamma \right\} $ that

\begin{itemize}
\item The instantaneous derivatives of $P_{S}$ and $H_{S}$ with respect to $%
g_{S}$ are zero; more precisely,%
\begin{equation*}
D_{g_{S}}P_{S}=D_{g_{S}}H_{S}=0.
\end{equation*}

\item The functional $H_{S}$ is determined by the functional $P_{S}$ through
the entropy relation:%
\begin{equation*}
H_{S}=-D_{u}P_{S}.
\end{equation*}

\item The modified heat conduction inequalities%
\begin{equation*}
\frac{1}{u^{2}}\left( q_{S}\cdot g_{S}\right) \leq \sigma _{S},\text{ }S\in
\left\{ \Omega ,\Gamma \right\} ,
\end{equation*}%
(with $q_{\Omega }=q$) hold for all smooth processes in $\overline{\Omega }$
and for all $t$.
\end{itemize}

\noindent Above, $\sigma _{S}$ denotes the internal/boundary dissipation%
\begin{equation*}
\sigma _{S}\left( t\right) :=-\frac{1}{u\left( t\right) }\left[ \delta
_{u}P_{S}\left( u^{t},g_{S}^{t}\mid \overset{\centerdot }{u}_{r}^{t}\right)
+\delta _{g_{S}}P_{S}\left( u^{t},g_{S}^{t}\mid \overset{\centerdot }{g}%
_{S,r}^{t}\right) \right] ,
\end{equation*}%
at time $t$, corresponding to the histories $\left( u^{t},g_{S}^{t}\right) $%
, where $\overset{\centerdot }{u}$ is the present rate of change of $u$ at
\thinspace $x$, $\overset{\centerdot }{u}_{r}^{t}$ is the past history of
the rate of change of $u$ at $x,$ and so on. Moreover, $D_{g_{S}}P_{S}$, $%
\delta _{u}P_{S}$ and $\delta _{g_{S}}P_{S}$ denote the following linear
differential operators%
\begin{align*}
D_{g_{S}}P_{S}\left( u^{t},g_{S}^{t}\right) \cdot l& =\left( \frac{\partial
}{\partial y}P_{S}\left( u_{r}^{t},g_{S,r}^{t},u,g_{S}+yl\right) \right)
_{y=0}, \\
\delta _{u}P_{S}\left( u^{t},g_{S}^{t}\mid k\right) & =\left( \frac{\partial
}{\partial y}P_{S}\left( u_{r}^{t}+yk,g_{S,r}^{t},u,g_{S}\right) \right)
_{y=0}, \\
\delta _{g_{S}}P_{S}\left( u^{t},g_{S}^{t}\mid \kappa \right) & =\left(
\frac{\partial }{\partial y}P_{S}\left( u_{r}^{t},g_{S,r}^{t}+y\kappa
,u,g_{S}\right) \right) _{y=0},
\end{align*}%
with identities which hold clearly for $\left( u^{t},g_{S}^{t}\right) \in
D_{S},$ $S\in \left\{ \Omega ,\Gamma \right\} $, $l\in \mathbb{R}^{\zeta
_{S}}$ ($\zeta _{\Omega }=d$, $\zeta _{\Gamma }=d-1$), and all $\left(
k,\kappa \right) $ such that%
\begin{equation*}
\int_{0}^{\infty }\left\vert k\left( s\right) \right\vert ^{2}m_{S}\left(
s\right) ds<\infty ,\int_{0}^{\infty }\left\vert \kappa \left( s\right)
\right\vert ^{2}m_{S}\left( s\right) ds<\infty .
\end{equation*}

To derive a simple model which is sufficiently general (see (\ref{eq1m})-(%
\ref{eq2m}) below), we need to consider a set of constitutive equations for $%
e_{S},q_{S}$, $S\in \left\{ \Omega ,\Gamma \right\} $, which comply with the
above implications that the second law has on the response functions
associated with a given complete thermodynamic process in $\overline{\Omega }
$. A fairly general assumption is to consider small variations in the
absolute temperature and temperature gradients on both $\Omega $ and $\Gamma
$, respectively, from equilibrium reference values (cf. (\ref{toy1})-(\ref%
{toy2})). We take%
\begin{equation*}
e_{\Omega }\left( u\right) =e_{\Omega ,\infty }+\rho _{\Omega }c_{\Omega }u,%
\text{ }e_{\Gamma }\left( u\right) =e_{\Gamma ,\infty }+\rho _{\Gamma
}c_{\Gamma }u,
\end{equation*}%
where the involved positive constants $e_{S,\infty },$ $c_{S},$ $\rho _{S}$
denote the internal energies at equilibrium, the specific heat capacities
and material densities of $S\in \left\{ \Omega ,\Gamma \right\} $,
respectively. In addition, we assume that the internal and boundary fluxes
satisfy the following constitutive equations:%
\begin{equation}
\begin{array}{ll}
q\left( t\right) =-\omega \nabla u-\left( 1-\omega \right) \int_{0}^{\infty
}m_{\Omega }\left( s\right) \nabla u^{t}\left( s\right) ds, &  \\
&  \\
q_{\Gamma }\left( t\right) =-\nu \nabla _{\Gamma }u-\left( 1-\nu \right)
\int_{0}^{\infty }m_{\Gamma }\left( s\right) \nabla _{\Gamma }u^{t}\left(
s\right) ds, &
\end{array}
\label{heatfluxes}
\end{equation}%
for some constants $\omega ,\nu \in \left( 0,1\right) $. Of course, when $%
m_{S}=0$, $S\in \left\{ \Omega ,\Gamma \right\} $, we recover in (\ref%
{heatfluxes}) the usual Fourier laws. Thus, in this context the constants $%
\omega ,\nu $ correspond to the instantaneous conductivities of $\Omega $
and $\Gamma $, respectively. Furthermore, we assume in (\ref{1stlaw-strong})
nonlinear temperature dependent heat sources $h_{S},$ $S\in \left\{ \Omega
,\Gamma \right\} $, namely, we take%
\begin{equation}
\begin{array}{ll}
h_{\Omega }\left( t\right) :=-f\left( u\left( t\right) \right) -\alpha
\left( 1-\omega \right) \int_{0}^{\infty }m_{\Omega }\left( s\right) u\left(
x,t-s\right) ds, &  \\
&  \\
h_{\Gamma }\left( t\right) :=-g\left( u\left( t\right) \right) -q\cdot
n-\beta \left( 1-\nu \right) \int_{0}^{\infty }m_{\Gamma }\left( s\right)
u\left( x,t-s\right) ds, &
\end{array}
\label{sources}
\end{equation}%
for some $\beta >0,\alpha >0$, where the source on $\Gamma ,$ $h_{\Gamma }$
is also assumed to depend linearly on heat transport from inside of $\Omega $
in directions normal to the boundary $\Gamma $. With these assumptions, (\ref%
{1stlaw-strong}) yields the following system with memory%
\begin{align}
& \partial _{t}u-\omega \Delta u-\left( 1-\omega \right) \int_{0}^{\infty
}m_{\Omega }\left( s\right) \Delta u\left( x,t-s\right) ds+f\left( u\right)
\label{eq1m} \\
& +\alpha \left( 1-\omega \right) \int_{0}^{\infty }m_{\Omega }\left(
s\right) u\left( x,t-s\right) ds  \notag \\
& =0,  \notag
\end{align}%
in $\Omega \times \left( 0,\infty \right) ,$ subject to the boundary
condition%
\begin{align}
& \partial _{t}u-\nu \Delta _{\Gamma }u+\omega \partial _{n}u+\left(
1-\omega \right) \int_{0}^{\infty }m_{\Omega }\left( s\right) \partial
_{n}u\left( x,t-s\right) ds  \label{eq2m} \\
& +\left( 1-\nu \right) \int_{0}^{\infty }m_{\Gamma }\left( s\right) \left(
-\Delta _{\Gamma }+\beta \right) u\left( x,t-s\right) ds+g\left( u\right)
\notag \\
& =0,  \notag
\end{align}%
on $\Gamma \times \left( 0,\infty \right) .$

It is worth emphasizing that a different choice $e_{\Gamma }\left( u\right)
=e_{\Gamma ,\infty }$ in (\ref{1stlaw-strong}) leads to a formulation in
which the boundary condition (\ref{eq2m}) is not dynamic any longer in the
sense that it does not contain the term $\partial _{t}u$ anymore. This
stationary boundary condition can be also reduced to (\ref{heat-m3}) by a
suitable choice of the parameters $\beta ,$ $\nu $ and the history $%
m_{\Gamma }$ involved in (\ref{heatfluxes}) and (\ref{sources}). On the
other hand, it is clear that if we (formally) choose $m_{S}=\delta _{0}$
(the Dirac mass at zero), for each $S\in \left\{ \Omega ,\Gamma \right\} $,
equations (\ref{eq1m})-(\ref{eq2m}) reduce into the following system%
\begin{equation}
\left\{
\begin{array}{ll}
\partial _{t}u-\Delta u+\overline{f}\left( u\right) =0, & \text{in }\Omega
\times \left( 0,\infty \right) , \\
\partial _{t}u-\Delta _{\Gamma }u+\partial _{n}u+\overline{g}\left( u\right)
=0, & \text{on }\Gamma \times \left( 0,\infty \right) ,%
\end{array}%
\right.   \label{eq-d}
\end{equation}%
where $\overline{g}\left( x\right) :=g\left( x\right) +\left( 1-\nu \right)
\beta x,$ $\overline{f}\left( x\right) :=f\left( x\right) +\left( 1-\omega
\right) \alpha x$, $x\in \mathbb{R}$. The latter has been investigated quite
extensively recently in many contexts (i.e., phase-field systems, heat
conduction phenomena with both a dissipative and non-dissipative source $%
\overline{g}$, Stefan problems, and many more). We refer the reader to
recent investigations pertaining the system (\ref{eq-d}) in \cite{CGGM10,
Gal12-2, Gal12-3, GGM08, Gal&Grasselli08, Gal&Warma10}, and the references
therein.

\section{Past history formulation and functional setup}

\label{fs}

As in \cite{CPS06} (cf. also \cite{Grasselli}), we can introduce the
so-called integrated past history of $u$, i.e., the auxiliary variable%
\begin{equation*}
\eta ^{t}\left( x,s\right) =\int_{0}^{s}u\left( x,t-y\right) dy,
\end{equation*}%
for $s,t>0$. Setting
\begin{equation}
\mu _{\Omega }\left( s\right) =-\omega ^{-1}\left( 1-\omega \right)
m_{\Omega }^{^{\prime }}\left( s\right) ,\text{ }\mu _{\Gamma }\left(
s\right) =-\nu ^{-1}\left( 1-\nu \right) m_{\Gamma }^{^{\prime }}\left(
s\right) ,  \label{eq:mu-1}
\end{equation}%
assuming that $m_{S},$ $S\in \left\{ \Omega ,\Gamma \right\} $, is
sufficiently smooth and vanishing at $\infty $, formal integration by parts
into (\ref{eq1m})-(\ref{eq2m}) yields%
\begin{align*}
\left( 1-\omega \right) \int_{0}^{\infty }m_{\Omega }\left( s\right) \Delta
u\left( x,t-s\right) ds& =\omega \int_{0}^{\infty }\mu _{\Omega }\left(
s\right) \Delta \eta ^{t}\left( x,s\right) ds, \\
\left( 1-\omega \right) \int_{0}^{\infty }m_{\Omega }\left( s\right)
\partial _{n}u\left( x,t-s\right) ds& =\omega \int_{0}^{\infty }\mu _{\Omega
}\left( s\right) \partial _{n}\eta ^{t}\left( x,s\right) ds
\end{align*}%
and%
\begin{equation}
\left( 1-\nu \right) \int_{0}^{\infty }m_{\Gamma }\left( s\right) \left(
-\Delta _{\Gamma }u\left( t-s\right) +\beta u\left( t-s\right) \right)
ds=\nu \int_{0}^{\infty }\mu _{\Gamma }\left( s\right) \left( -\Delta
_{\Gamma }\eta ^{t}\left( s\right) +\beta \eta ^{t}\left( s\right) \right)
ds.  \label{eq:mu-2}
\end{equation}%
Thus, we consider the following formulation.

\noindent \textbf{Problem P}. Find a function $\left( u,\eta ^{t}\right) $
such that%
\begin{equation}
\partial _{t}u-\omega \Delta u-\omega \int_{0}^{\infty }\mu _{\Omega }\left(
s\right) \Delta \eta ^{t}\left( s\right) ds+\alpha \omega \int_{0}^{\infty
}\mu _{\Omega }\left( s\right) \eta ^{t}\left( x,s\right) ds+f\left(
u\right) =0,  \label{eqm}
\end{equation}%
in $\Omega \times \left( 0,\infty \right) ,$%
\begin{align}
& \partial _{t}u-\nu \Delta _{\Gamma }u+\omega \partial _{n}u+\omega
\int_{0}^{\infty }\mu _{\Omega }\left( s\right) \partial _{n}\eta ^{t}\left(
s\right) ds  \label{eqmbis} \\
& +\nu \int_{0}^{\infty }\mu _{\Gamma }\left( s\right) \left( -\Delta
_{\Gamma }\eta ^{t}\left( s\right) +\beta \eta ^{t}\left( s\right) \right)
ds+g\left( u\right)  \notag \\
& =0,  \notag
\end{align}%
on $\Gamma \times \left( 0,\infty \right) ,$ and%
\begin{equation}
\partial _{t}\eta ^{t}\left( s\right) +\partial _{s}\eta ^{t}\left( s\right)
=u\left( t\right) ,\text{ in }\overline{\Omega }\times \left( 0,\infty
\right) ,  \label{eqmtris}
\end{equation}%
subject to the boundary conditions%
\begin{equation}
\eta ^{t}\left( 0\right) =0\text{, in }\overline{\Omega }\times \left(
0,\infty \right)  \label{eqic0}
\end{equation}%
and initial conditions%
\begin{equation}
u\left( 0\right) =u_{0}\text{ in }\Omega ,\text{ }u\left( 0\right) =v_{0}%
\text{ on }\Gamma ,  \label{eqic1}
\end{equation}%
and%
\begin{equation}
\eta ^{0}\left( s\right) =\eta _{0}\text{ in }\Omega \text{, }\eta
^{0}\left( s\right) =\xi _{0}\text{ on }\Gamma .  \label{eqic2}
\end{equation}%
Note that we do not require that the boundary traces of $u_{0}$ and $\eta
_{0}$ equal to $v_{0}$ and $\xi _{0}$, respectively. Thus, we are solving a
much more general problem in which equation (\ref{eqm}) is interpreted as an
evolution equation in the bulk $\Omega $ properly coupled with the equation (%
\ref{eqmbis}) on the boundary $\Gamma .$ Finally, we note that $\eta
_{0},\xi _{0}$ are defined by%
\begin{eqnarray*}
\eta _{0} &=&\int_{0}^{s}u_{0}\left( x,-y\right) dy,\text{ in }\Omega \text{%
, for }s>0, \\
\xi _{0} &=&\int_{0}^{s}v_{0}\left( x,-y\right) dy,\text{ on }\Gamma \text{,
for }s>0.
\end{eqnarray*}%
However, from now on both $\eta _{0}$ and $\xi _{0}$ will be regarded as
independent of the initial data $u_{0},v_{0}.$ Indeed, below we will
consider a more general problem with respect to the original one. In order
to give a more rigorous notion of solutions for problem (\ref{eqm})-(\ref%
{eqic2}), we need to introduce some terminology and the functional setting
associated with this system.

In the sequel, we denote by $\left\Vert \cdot \right\Vert _{L^{2}\left(
\Gamma \right) }$ and $\left\Vert \cdot \right\Vert _{L^{2}\left( \Omega
\right) }$ the norms on $L^{2}\left( \Gamma \right) $ and $L^{2}\left(
\Omega \right) $, whereas the inner products in these spaces are denoted by $%
\left\langle \cdot ,\cdot \right\rangle _{L^{2}\left( \Gamma \right) }$ and $%
\left\langle \cdot ,\cdot \right\rangle _{L^{2}\left( \Omega \right) },$
respectively. Furthermore, the norms on $H^{s}\left( \Omega \right) $ and $%
H^{s}\left( \Gamma \right) ,$ for $s>0,$ will be indicated by $\left\Vert
\cdot \right\Vert _{H^{s}}$ and $\left\Vert \cdot \right\Vert _{H^{s}\left(
\Gamma \right) }$, respectively. The symbol $\left\langle \cdot ,\cdot
\right\rangle $ stands for pairing between any generic Banach spaces $V$ and
its dual $V^{\ast }$; $(u,v)^{\mathrm{tr}}$ will also simply denote the
vector-valued function $\binom{u}{v}.$ Constants below may depend on various
structural parameters such as $|\Omega |$, $|\Gamma |$, $\ell _{1},$ $\ell
_{2}$, etc, and these constants may even change from line to line.
Furthermore, we denote by $K(R)$ a generic monotonically increasing function
of $R>0,$ whose specific dependance on other parameters will be made
explicit on occurrence.

Let us now define the basic functional setup for (\ref{eqm})-(\ref{eqic2}).
From this point on, we assume that $\Omega $ is a bounded domain of $\mathbb{%
R}^{3}$ with boundary $\Gamma $ which is of class $\mathcal{C}^{2}$. To this
end, consider the space $\mathbb{X}^{2}=L^{2}\left( \overline{\Omega },d\mu
\right) ,$ where $d\mu =dx_{\mid \Omega }\oplus d\sigma ,$ such that $dx$
denotes the Lebesgue measure on $\Omega $ and $d\sigma $ denotes the natural
surface measure on $\Gamma $. It is easy to see that $\mathbb{X}%
^{2}=L^{2}\left( \Omega ,dx\right) \oplus L^{2}\left( \Gamma ,d\sigma
\right) $ may be identified under the natural norm%
\begin{equation*}
\left\Vert u\right\Vert _{\mathbb{X}^{2}}^{2}=\int\limits_{\Omega
}\left\vert u\left( x\right) \right\vert ^{2}dx+\int\limits_{\Gamma
}\left\vert u\left( x\right) \right\vert ^{2}d\sigma .
\end{equation*}%
Moreover, if we identify every $u\in C\left( \overline{\Omega }\right) $
with $U=\left( u_{\mid \Omega },u_{\mid \Gamma }\right) \in C\left( \Omega
\right) \times C\left( \Gamma \right) $, we may also define $\mathbb{X}^{2}$
to be the completion of $C\left( \overline{\Omega }\right) $ in the norm $%
\left\Vert \cdot \right\Vert _{\mathbb{X}^{2}}$. In general, any function $%
u\in \mathbb{X}^{2}$ will be of the form $u=\binom{u_{1}}{u_{2}}$ with $%
u_{1}\in L^{2}\left( \Omega ,dx\right) $ and $u_{2}\in L^{2}\left( \Gamma
,d\sigma \right) ,$ and there need not be any connection between $u_{1}$ and
$u_{2}$. From now on, the inner product in the Hilbert space $\mathbb{X}^{2}$
will be denoted by $\left\langle \cdot ,\cdot \right\rangle _{\mathbb{X}%
^{2}}.$ Next, recall that the Dirichlet trace map ${\mathrm{tr_{D}}}%
:C^{\infty }\left( \overline{\Omega }\right) \rightarrow C^{\infty }\left(
\Gamma \right) ,$ defined by ${\mathrm{tr_{D}}}\left( u\right) =u_{\mid
\Gamma }$ extends to a linear continuous operator ${\mathrm{tr_{D}}}%
:H^{r}\left( \Omega \right) \rightarrow H^{r-1/2}\left( \Gamma \right) ,$
for all $r>1/2$, which is onto for $1/2<r<3/2.$ This map also possesses a
bounded right inverse ${\mathrm{tr_{D}}}^{-1}:H^{r-1/2}\left( \Gamma \right)
\rightarrow H^{r}\left( \Omega \right) $ such that ${\mathrm{tr_{D}}}\left( {%
\mathrm{tr_{D}}}^{-1}\psi \right) =\psi ,$ for any $\psi \in H^{r-1/2}\left(
\Gamma \right) $. We can thus introduce the subspaces of $H^{r}\left( \Omega
\right) \times H^{r-1/2}\left( \Gamma \right) $ and $H^{r}\left( \Omega
\right) \times H^{r}\left( \Gamma \right) $, respectively, by%
\begin{align}
\mathbb{V}_{0}^{r}& :=\{U=\left( u,\psi \right) \in H^{r}\left( \Omega
\right) \times H^{r-1/2}\left( \Gamma \right) :{\mathrm{tr_{D}}}\left(
u\right) =\psi \},  \label{vvv} \\
\mathbb{V}^{r}& :=\{U=\left( u,\psi \right) \in \mathbb{V}_{0}^{r}:{\mathrm{%
tr_{D}}}\left( u\right) =\psi \in H^{r}\left( \Gamma \right) \},  \notag
\end{align}%
for every $r>1/2,$ and note that $\mathbb{V}_{0}^{r},$ $\mathbb{V}^{r}$ are
not product spaces. However, we have the following dense and compact
embeddings $\mathbb{V}_{0}^{r_{1}}\subset \mathbb{V}_{0}^{r_{2}},$ for any $%
r_{1}>r_{2}>1/2$ (by definition, this also true for the sequence of spaces $%
\mathbb{V}^{r_{1}}\subset \mathbb{V}^{r_{2}}$). Naturally, the norm on the
spaces $\mathbb{V}_{0}^{r},$ $\mathbb{V}^{r}$ are defined by%
\begin{equation}
\Vert U\Vert _{\mathbb{V}_{0}^{r}}^{2}:=\Vert u\Vert _{H^{r}}^{2}+\Vert \psi
\Vert _{H^{r-1/2}(\Gamma )}^{2},\text{ }\Vert U\Vert _{\mathbb{V}%
^{r}}^{2}:=\Vert u\Vert _{H^{r}}^{2}+\Vert \psi \Vert _{H^{r}(\Gamma )}^{2}.
\label{eq:Vr-norm}
\end{equation}%
In particular, the norm in the spaces $\mathbb{V}^{1},$ $\mathbb{V}_{0}^{1}$
can be defined as in terms of the following equivalent norms:%
\begin{eqnarray*}
\Vert U\Vert _{\mathbb{V}^{1}} &:&=\left( \omega \Vert \nabla u\Vert
_{L^{2}\left( \Omega \right) }^{2}+\nu \Vert \nabla _{\Gamma }\psi \Vert
_{L^{2}(\Gamma )}^{2}+\beta \nu \left\Vert \psi \right\Vert _{L^{2}\left(
\Gamma \right) }^{2}\right) ^{1/2},\text{ }\nu >0, \\
\Vert U\Vert _{\mathbb{V}_{0}^{1}} &:&=\left( \omega \Vert \nabla u\Vert
_{L^{2}\left( \Omega \right) }^{2}+\alpha \omega \left\Vert u\right\Vert
_{L^{2}\left( \Omega \right) }^{2}\right) ^{1/2}.
\end{eqnarray*}%
Now we introduce the spaces for the memory vector-valued function $\left(
\eta ,\xi \right) $. For a given nonnegative, not identically equal to
zero,\ and measurable function $\theta _{S},$ $S\in \left\{ \Omega ,\Gamma
\right\} $, defined on $\mathbb{R}_{+},$ and a real Hilbert space $W$ (with
inner product denoted by $\left\langle \cdot ,\cdot \right\rangle _{\mathrm{W%
}}$), let $L_{\theta _{S}}^{2}\left( \mathbb{R}_{+};W\right) $ be the
Hilbert space of $W$-valued functions on $\mathbb{R}_{+}$, endowed with the
following inner product%
\begin{equation}
\left\langle \phi _{1},\phi _{2}\right\rangle _{L_{\theta _{S}}^{2}\left(
\mathbb{R}_{+};W\right) }=\int_{0}^{\infty }\theta _{S}(s)\left\langle \phi
_{1}\left( s\right) ,\phi _{2}\left( s\right) \right\rangle _{\mathrm{W}}ds.
\label{sc}
\end{equation}%
Moreover, for each $r>1/2$ we define%
\begin{equation*}
L_{\theta _{\Omega }\oplus \theta _{\Gamma }}^{2}\left( \mathbb{R}_{+};%
\mathbb{V}^{r}\right) \simeq L_{\theta _{\Omega }}^{2}\left( \mathbb{R}_{+};%
\mathbb{V}_{0}^{r}\right) \oplus L_{\theta _{\Gamma }}^{2}\left( \mathbb{R}%
_{+};H^{r}\left( \Gamma \right) \right)
\end{equation*}%
as the Hilbert space of $\mathbb{V}^{r}$-valued functions $\left( \eta ,\xi
\right) ^{\mathrm{tr}}$ on $\mathbb{R}_{+}$ endowed with the inner product%
\begin{equation*}
\left\langle \binom{\eta _{1}}{\xi _{1}},\binom{\eta _{2}}{\xi _{2}}%
\right\rangle _{L_{\theta _{\Omega }\oplus \theta _{\Gamma }}^{2}\left(
\mathbb{R}_{+};\mathbb{V}^{r}\right) }=\int_{0}^{\infty }\left( \theta
_{\Omega }(s)\left\langle \eta _{1}\left( s\right) ,\eta _{2}\left( s\right)
\right\rangle _{H^{r}}+\theta _{\Gamma }(s)\left\langle \xi _{1}\left(
s\right) ,\xi _{2}\left( s\right) \right\rangle _{H^{r}\left( \Gamma \right)
}\right) ds.
\end{equation*}%
Consequently, for $r>1/2$ we set%
\begin{equation*}
\mathcal{M}_{\Omega }^{0}:=L_{\mu _{\Omega }}^{2}\left( \mathbb{R}%
_{+};L^{2}\left( \Omega \right) \right) \text{, }\mathcal{M}_{\Omega
}^{r}:=L_{\mu _{\Omega }}^{2}(\mathbb{R}_{+};\mathbb{V}_{0}^{r})\text{, }%
\mathcal{M}_{\Gamma }^{r}:=L_{\mu _{\Gamma }}^{2}(\mathbb{R}_{+};H^{r}\left(
\Gamma \right) )
\end{equation*}%
and%
\begin{equation*}
\mathcal{M}_{\Omega ,\Gamma }^{0}:=L_{\mu _{\Omega }\oplus \mu _{\Gamma
}}^{2}\left( \mathbb{R}_{+};\mathbb{X}^{2}\right) ,\text{ }\mathcal{M}%
_{\Omega ,\Gamma }^{r}:=L_{\mu _{\Omega }\oplus \mu _{\Gamma }}^{2}\left(
\mathbb{R}_{+};\mathbb{V}^{r}\right) .
\end{equation*}%
Clearly, because of the topological identification $H^{r}\left( \Omega
\right) \simeq \mathbb{V}_{0}^{r}$, one has the inclusion $\mathcal{M}%
_{\Omega ,\Gamma }^{r}\subset \mathcal{M}_{\Omega }^{r}$ for each $r>1/2$.
In the sequel, we will also consider Hilbert spaces of the form $W_{\mu
_{\Omega }}^{k,2}\left( \mathbb{R}_{+};\mathbb{V}_{0}^{r}\right) $ for $k\in
\mathbb{N}$. When it is convenient, we will also use the notation%
\begin{equation}
\mathcal{H}_{\Omega ,\Gamma }^{0,1}:=\mathbb{X}^{2}\times \mathcal{M}%
_{\Omega ,\Gamma }^{1}\text{, }\mathcal{H}_{\Omega ,\Gamma }^{s,r}:=\mathbb{V%
}^{s}\times \mathcal{M}_{\Omega ,\Gamma }^{r}\text{ for }s,r\geq 1.  \notag
\end{equation}%
For matter of convenience, we will also set the inner product in $\mathcal{M}%
_{\Omega ,\Gamma }^{1},$ as follows:%
\begin{align*}
& \left\langle \binom{\eta _{1}}{\xi _{1}},\binom{\eta _{2}}{\xi _{2}}%
\right\rangle _{L_{\theta _{\Omega }\oplus \theta _{\Gamma }}^{2}\left(
\mathbb{R}_{+};\mathbb{V}^{1}\right) } \\
& =\omega \int_{0}^{\infty }\theta _{\Omega }(s)\left( \left\langle \nabla
\eta _{1}\left( s\right) ,\nabla \eta _{2}\left( s\right) \right\rangle
_{L^{2}\left( \Omega \right) }+\alpha \left\langle \eta _{1}\left( s\right)
,\eta _{2}\left( s\right) \right\rangle _{L^{2}\left( \Omega \right)
}\right) ds \\
& +\nu \int_{0}^{\infty }\theta _{\Gamma }(s)\left( \left\langle \nabla
_{\Gamma }\xi _{1}\left( s\right) ,\nabla _{\Gamma }\xi _{2}\left( s\right)
\right\rangle _{L^{2}\left( \Gamma \right) }+\beta \left\langle \xi
_{1}\left( s\right) ,\xi _{2}\left( s\right) \right\rangle _{L^{2}\left(
\Gamma \right) }\right) ds.
\end{align*}

The following basic elliptic estimate is taken from \cite[Lemma 2.2]%
{Gal&Grasselli08}.

\begin{lemma}
\label{t:appendix-lemma-3} Consider the linear boundary value problem,
\begin{equation}
\left\{
\begin{array}{rl}
-\Delta u & =p_{1}~~\text{in}~\Omega , \\
-\Delta _{\Gamma }u+\partial _{n}u+\beta u & =p_{2}~~\text{on}~\Gamma .%
\end{array}%
\right.  \label{eq:appendix-BVP}
\end{equation}%
If $(p_{1},p_{2})^{{\mathrm{tr}}}\in H^{s}(\Omega )\times H^{s}(\Gamma )$,
for $s\geq 0$ and $s+\frac{1}{2}\not\in \mathbb{N}$, then the following
estimate holds for some constant $C>0$,
\begin{equation}
\Vert u\Vert _{H^{s+2}}+\Vert u\Vert _{H^{s+2}(\Gamma )}\leq C\left( \Vert
p_{1}\Vert _{H^{s}}+\Vert p_{2}\Vert _{H^{s}(\Gamma )}\right) .
\label{eq:H2-regularity-estimate}
\end{equation}
\end{lemma}

We also recall the following basic inequality from \cite[Lemma A.2]{Gal12-2}.

\begin{lemma}
\label{t:appendix-lemma-1} Let $s>1$ and $u\in H^{1}(\Omega )$. Then, for
every $\varepsilon >0$, there exists a positive constant $C_{\varepsilon
}\sim \varepsilon ^{-1}$ such that,
\begin{equation}
\Vert u\Vert _{L^{s}(\Gamma )}^{s}\leq \varepsilon \Vert \nabla u\Vert
_{L^{2}(\Omega )}^{2}+C_{\varepsilon }\left( \Vert u\Vert _{L^{\gamma
}(\Omega )}^{\gamma }+1\right) ,  \label{eq:appendix-lemma-1}
\end{equation}%
where $\gamma =\max \{s,2(s-1)\}$.
\end{lemma}

Next, we consider the linear (self-adjoint, positive)\ operator $\mathrm{C}%
\psi :=\mathrm{C}_{\beta }\psi =-\Delta _{\Gamma }\psi +\beta \psi $ acting
on $D\left( \mathrm{C}\right) =H^{2}\left( \Gamma \right) $. The basic
(linear) operator, associated with problem (\ref{eqm})-(\ref{eqmtris}), is
the so-called \textquotedblleft Wentzell\textquotedblright\ Laplace
operator. Recall that $\omega \in \left( 0,1\right) $. We let%
\begin{align}
\mathrm{A_{W}^{\alpha ,\beta ,\nu ,\omega }}\binom{u_{1}}{u_{2}}& :=\left(
\begin{array}{cc}
-\omega \Delta +\alpha \omega I & 0 \\
\omega \partial _{n}\left( \cdot \right) & \nu \mathrm{C}%
\end{array}%
\right) \left(
\begin{array}{c}
u_{1} \\
u_{2}%
\end{array}%
\right)  \label{A_Wentzell1} \\
& =\mathrm{A_{W}^{\alpha ,0,0,\omega }}\binom{u_{1}}{u_{2}}+\binom{0}{\nu
\mathrm{C}u_{2}},  \notag
\end{align}%
with%
\begin{equation}
D\left( \mathrm{A_{W}^{\alpha ,\beta ,\nu ,\omega }}\right) :=\left\{ U=%
\binom{u_{1}}{u_{2}}\in \mathbb{Y}:-\Delta u_{1}\in L^{2}\left( \Omega
\right) ,\text{ }\omega \partial _{n}u_{1}-\nu \mathrm{C}u_{2}\in
L^{2}\left( \Gamma \right) \right\} ,  \label{A_Wentzell2}
\end{equation}%
where $\mathbb{Y}:=\mathbb{V}_{0}^{1}$ if $\nu =0,$ and $\mathbb{Y}:=\mathbb{%
V}^{1}$ if $\nu >0.$ It is well-known that $(\mathrm{A_{W}^{\alpha ,\beta
,\nu ,\omega }},D(\mathrm{A_{W}^{\alpha ,\beta ,\nu ,\omega })})$ is
self-adjoint and nonnegative operator on $\mathbb{X}^{2}$ whenever $\alpha
,\beta ,\nu \geq 0,$ and $\mathrm{A_{W}^{\alpha ,\beta ,\nu ,\omega }}>0$ if
either $\alpha >0$ or $\beta >0$. Moreover, the resolvent operator $(I+%
\mathrm{A_{W}^{\alpha ,\beta ,\nu ,\omega }})^{-1}\in \mathcal{L}\left(
\mathbb{X}^{2}\right) $ is compact. Moreover, since $\Gamma $ is of class $%
\mathcal{C}^{2},$ then $D(\mathrm{A_{W}^{\alpha ,\beta ,\nu ,\omega }})=%
\mathbb{V}^{2}$ if $\nu >0$. Indeed, for any $\alpha ,\beta \geq 0$ with $%
\left( \alpha ,\beta \right) \neq \left( 0,0\right) ,$ the map $\Psi
:U\mapsto \mathrm{A_{W}^{\alpha ,\beta ,\nu ,\omega }}U,$ when viewed as a
map from $\mathbb{V}_{2}$ into $\mathbb{X}^{2}=L^{2}\left( \Omega \right)
\times L^{2}\left( \Gamma \right) ,$ is an isomorphism and there exists a
positive constant $C_{\ast }$, independent of $U=\left( u,\psi \right) ^{%
\mathrm{tr}}$, such that%
\begin{equation}
C_{\ast }^{-1}\left\Vert U\right\Vert _{\mathbb{V}^{2}}\leq \left\Vert \Psi
\left( U\right) \right\Vert _{\mathbb{X}^{2}}\leq C_{\ast }\left\Vert
U\right\Vert _{\mathbb{V}^{2}},  \label{regularity-oper}
\end{equation}%
for all $U\in \mathbb{V}^{2}$ (cf. Lemma \ref{t:appendix-lemma-3}). Whenever
$\nu =0$, by elliptic regularity theory and $U\in D(\mathrm{A_{W}^{\alpha
,\beta ,0,\omega }})$ one has $u\in H^{3/2}\left( \Omega \right) $ and $\psi
={\mathrm{tr_{D}}}\left( u\right) \in H^{1}\left( \Gamma \right) $, since
the Dirichlet-to-Neumann map is bounded from $H^{1}\left( \Gamma \right) $
to $L^{2}\left( \Gamma \right) $; hence $D(\mathrm{A_{W}^{\alpha ,\beta
,0,\omega }})=\mathbb{W}$, where $\mathbb{W}$ is the Hilbert space equipped
with the following (equivalent) norm%
\begin{equation*}
\left\Vert U\right\Vert _{\mathbb{W}}^{2}:=\left\Vert U\right\Vert _{\mathbb{%
V}_{0}^{3/2}}^{2}+\left\Vert \Delta u\right\Vert _{L^{2}\left( \Omega
\right) }^{2}+\left\Vert \partial _{n}u\right\Vert _{L^{2}\left( \Gamma
\right) }^{2}.
\end{equation*}%
We refer the reader to more details to e.g., \cite{CGGM10}, \cite%
{Gal&Warma10}, \cite{CFGGGOR09} and the references therein. We now have all
the necessary ingredients to introduce a rigorous formulation of problem
\textbf{P} in the next section.

\section{Variational formulation and well-posedness}

\label{vf}

We need the following hypotheses for problem \textbf{P}. For the function $%
\mu _{S},$ $S\in \left\{ \Omega ,\Gamma \right\} $, given by (\ref{eq:mu-1}%
), we consider the following assumptions (cf., e.g. \cite{CPS06}, \cite%
{GPM98} and \cite{GPM00}). Assume
\begin{align}
& \mu _{S}\in C^{1}(\mathbb{R}_{+})\cap L^{1}(\mathbb{R}_{+}),  \label{miu1}
\\
& \mu _{S}(s)\geq 0~~\text{for all}~~s\geq 0,  \label{miu2} \\
& \mu _{S}^{\prime }(s)\leq 0~~\text{for all}~~s\geq 0.  \label{miu3}
\end{align}%
These assumptions are equivalent to assuming that $m_{S}(s),$ $S\in \left\{
\Omega ,\Gamma \right\} $, is a bounded, positive, nonincreasing, convex
function of class $C^{2}$. These conditions are commonly used in the
literature (see, for example, \cite{CPS06}, \cite{GPM98} and \cite{Grasselli}%
) to establish existence and uniqueness of continuous global weak solutions
for Coleman-Gurtin type equations subject to Dirichlet boundary conditions.

As far as natural conditions for the nonlinear terms are concerned, we
assume $f$, $g\in C^{1}(\mathbb{R})$ satisfy the sign conditions%
\begin{equation}
f^{\prime }(s)\geq -M_{f},\text{ }g^{\prime }(s)\geq -M_{g},\text{ for all }%
s\in \mathbb{R}\text{,}  \label{eq:assumption-2}
\end{equation}%
for some $M_{f},M_{g}>0$ and the growth assumptions, for all $s\in \mathbb{R}
$,
\begin{equation}
|f(s)|\leq \ell _{1}(1+|s|^{r_{1}-1}),\text{ }|g(s)|\leq \ell
_{2}(1+|s|^{r_{2}-1}),  \label{eq:assumption-4}
\end{equation}%
for some positive constants $\ell _{1}$ and $\ell _{2}$, and where $%
r_{1},r_{2}\geq 2$. Let now%
\begin{equation}
\widetilde{g}\left( s\right) :=g\left( s\right) -\nu \beta s,\text{ for }%
s\in \mathbb{R}\text{.}  \label{tilda-f}
\end{equation}%
In addition, we assume there exists $\varepsilon \in (0,\omega )$ so that
the following balance condition%
\begin{equation}
\liminf_{|s|\rightarrow \infty }\frac{f(s)s+\frac{|\Gamma |}{|\Omega |}%
\widetilde{g}(s)s-\frac{C_{\Omega }^{2}|\Gamma |^{2}}{4\varepsilon |\Omega
|^{2}}|\widetilde{g}^{\prime }(s)s+\widetilde{g}(s)|^{2}}{\left\vert
s\right\vert ^{r_{1}}}>0  \label{eq:assumption-5}
\end{equation}%
holds for $r_{1}\geq \max \{r_{2},2(r_{2}-1)\}$. The number $C_{\Omega }>0$
is the best Sobolev constant in the following Sobolev-Poincar\'{e} inequality%
\begin{equation}
\left\Vert u-\left\langle u\right\rangle _{\Gamma }\right\Vert _{L^{2}\left(
\Omega \right) }\leq C_{\Omega }\left\Vert \nabla u\right\Vert
_{L^{2}(\Omega )},\text{ }\left\langle u\right\rangle _{\Gamma }:=\frac{1}{%
\left\vert \Gamma \right\vert }\int\limits_{\Gamma }\mathrm{tr_{D}}\left(
u\right) d\sigma ,  \label{eq:Poincare}
\end{equation}%
for all $u\in H^{1}\left( \Omega \right) $, see \cite[Lemma 3.1]{RBT01}.

The assumption (\ref{eq:assumption-5}) deserves some additional comments.
Suppose that that for $\left\vert y\right\vert \rightarrow \infty ,$ both
the internal and boundary functions behave accordingly to the following laws:%
\begin{equation}
\lim_{\left\vert y\right\vert \rightarrow \infty }\frac{f^{^{\prime }}\left(
y\right) }{\left\vert y\right\vert ^{r_{1}-2}}=\left( r_{1}-1\right) c_{f}%
\text{, }\lim_{\left\vert y\right\vert \rightarrow \infty }\frac{\widetilde{g%
}^{^{\prime }}\left( y\right) }{\left\vert y\right\vert ^{r_{2}-2}}=\left(
r_{2}-1\right) c_{\widetilde{g}},  \label{poly}
\end{equation}%
for some $c_{f},c_{\widetilde{g}}\in \mathbb{R}\backslash \left\{ 0\right\} $%
. In particular, it holds%
\begin{equation*}
f\left( y\right) y\sim c_{f}\left\vert y\right\vert ^{r_{1}},\text{ }%
\widetilde{g}\left( y\right) y\sim c_{\widetilde{g}}\left\vert y\right\vert
^{r_{2}}\text{ as }\left\vert y\right\vert \rightarrow \infty .
\end{equation*}%
For the case of bulk dissipation (i.e., $c_{f}>0$) and anti-dissipative
behavior at the boundary $\Gamma $ (i.e., $c_{\widetilde{g}}<0$), assumption
(\ref{eq:assumption-5}) is automatically satisfied provided that $r_{1}>\max
\{r_{2},2(r_{2}-1)\}$. Furthermore, if $2<r_{2}<2\left( r_{2}-1\right)
=r_{1} $ and%
\begin{equation}
c_{f}>\frac{1}{4\varepsilon }\left( \frac{C_{\Omega }\left\vert \Gamma
\right\vert c_{\widetilde{g}}r_{2}}{\left\vert \Omega \right\vert }\right)
^{2},  \label{poly2}
\end{equation}%
for some $\varepsilon \in (0,\omega )$, then once again (\ref%
{eq:assumption-5}) is satisfied. In the case when $f$ and $g$ are sublinear
(i.e., $r_{1}=r_{2}=2$ in (\ref{eq:assumption-4})), the condition (\ref%
{eq:assumption-5}) is also automatically satisfied provided that%
\begin{equation}
\left( c_{f}+\frac{|\Gamma |}{|\Omega |}c_{\widetilde{g}}\right) >\frac{1}{%
\varepsilon }\left( \frac{C_{\Omega }\left\vert \Gamma \right\vert c_{%
\widetilde{g}}}{\left\vert \Omega \right\vert }\right) ^{2}  \label{poly3}
\end{equation}%
for some $\varepsilon \in \left( 0,\omega \right) $. Of course, when both
the bulk and boundary nonlinearities are dissipative, i.e., there exist two
constants $C_{f}>0,C_{g}>0$ such that, additionally to (\ref{eq:assumption-4}%
),%
\begin{equation}
\left\{
\begin{array}{l}
f\left( s\right) s\geq C_{f}\left\vert s\right\vert ^{r_{1}}, \\
\widetilde{g}\left( s\right) s\geq C_{g}\left\vert s\right\vert ^{r_{2}},%
\end{array}%
\right.  \label{disf1}
\end{equation}%
for all $\left\vert s\right\vert \geq s_{0}$, for some sufficiently large $%
s_{0}>0,$ condition (\ref{eq:assumption-5}) can be dropped and is no longer
required (see \cite{Gal12-2}).

In order to introduce a rigorous formulation for problem \textbf{P}, we
define
\begin{equation}
D(\mathrm{T}):=\left\{ \Phi =\binom{\eta ^{t}}{\xi ^{t}}\in \mathcal{M}%
_{\Omega ,\Gamma }^{1}:\partial _{s}\Phi \in \mathcal{M}_{\Omega ,\Gamma
}^{1},\text{ }\Phi (0)=0\right\}  \label{eq:memory-4}
\end{equation}%
and consider the linear (unbounded) operator $\mathrm{T}:D(\mathrm{T}%
)\rightarrow \mathcal{M}_{\Omega ,\Gamma }^{1}$ by%
\begin{equation*}
\mathrm{T}\Phi =-\binom{\frac{d\eta }{ds}}{\frac{d\xi }{ds}},\text{ }\Phi =%
\binom{\eta ^{t}}{\xi ^{t}}\in D(\mathrm{T}).
\end{equation*}%
The follow result can be proven following \cite[Theorem 3.1]{Grasselli}.

\begin{proposition}
\label{t:generator-T} The operator $\mathrm{T}$ with domain $D(\mathrm{T})$
is an infinitesimal generator of a strongly continuous semigroup of
contractions on $\mathcal{M}_{\Omega ,\Gamma }^{1}$, denoted $e^{\mathrm{T}%
t} $.
\end{proposition}

As a consequence, we also have (cf., e.g. \cite[Corollary IV.2.2]{Pazy83}).

\begin{corollary}
\label{t:memory-regularity-1} Let $T>0$ and assume $U=\binom{u}{v}\in
L^{1}(0,T;\mathbb{V}^{1})$. Then, for every $\Phi _{0}\in \mathcal{M}%
_{\Omega ,\Gamma }^{1}$, the Cauchy problem for $\Phi ^{t}=\binom{\eta ^{t}}{%
\xi ^{t}},$%
\begin{equation}
\left\{
\begin{array}{ll}
\partial _{t}\Phi ^{t}=\mathrm{T}\Phi ^{t} +U(t), & \text{for}~~t>0, \\
\Phi ^{0}=\Phi _{0}, &
\end{array}%
\right.  \label{eq:memory-1}
\end{equation}%
has a unique (mild) solution $\Phi \in C([0,T];\mathcal{M}_{\Omega ,\Gamma
}^{1})$ which can be explicitly given as%
\begin{equation}
\Phi ^{t}(s)=\left\{
\begin{array}{ll}
\displaystyle\int_{0}^{s}U(t-y)dy, & \text{for}~~0<s\leq t, \\
\displaystyle\Phi _{0}(s-t)+\int_{s}^{t}U(t-y)dy, & \text{for }~~s>t,%
\end{array}%
\right.  \label{eq:representation-formula-1}
\end{equation}%
cf. also \cite[Section 3.2]{CPS06} and \cite[Section 3]{Grasselli}.
\end{corollary}

\begin{remark}
(i) Note that, from assumption (\ref{miu3}), the following inequality%
\begin{equation}
\left\langle \mathrm{T} \Phi,\Phi \right\rangle _{\mathcal{M}_{\Omega
,\Gamma }^{1}}\leq 0  \label{eq:operator-T-1}
\end{equation}%
holds for all $\Phi \in D(\mathrm{T})$.

(ii) If $\Phi _{0}\in D(\mathrm{T})$ and $\partial _{t}U\in L^{1}\left( 0,T;%
\mathbb{V}^{1}\right) $, the function $\Phi ^{t}$ given by (\ref%
{eq:representation-formula-1}) satisfies (\ref{eq:memory-1}) in the strong
sense a.e. on $\left( 0,T\right) ,$ for any $T>0.$
\end{remark}

We are now ready to introduce the rigorous (variational) formulation of
problem \textbf{P}.

\begin{definition}
\label{weak}Let $\alpha ,\beta >0$, $\omega ,\nu \in (0,1)$ and $T>0$. Given
$\binom{u_{0}}{v_{0}}\in \mathbb{X}^{2}$, $\binom{\eta _{0}}{\xi _{0}}\in
\mathcal{M}_{\Omega ,\Gamma }^{1}$, we seek to find functions $U\left(
t\right) =\binom{u\left( t\right) }{v\left( t\right) },$ $\Phi ^{t}=\binom{%
\eta ^{t}}{\xi ^{t}}$ with the following properties:%
\begin{align}
U& \in L^{\infty }\left( 0,T;\mathbb{X}^{2}\right) \cap L^{2}(0,T;\mathbb{V}%
^{1}),\text{ }\Phi \in L^{\infty }\left( 0,T;\mathcal{M}_{\Omega ,\Gamma
}^{1}\right) ,  \label{eq:problem-p-1} \\
u& \in L^{r_{1}}(\Omega \times \left( 0,T\right) ),\text{ }v\in
L^{r_{2}}(\Gamma \times (0,T)),  \label{eq:problem-p-1bis} \\
\partial _{t}U& \in L^{2}\left( 0,T;(\mathbb{V}^{1})^{\ast }\right) \oplus
\left( L^{r_{1}^{\prime }}(\Omega \times (0,T))\times L^{r_{2}^{\prime
}}(\Gamma \times (0,T))\right) ,  \label{eq:problem-p-1bbis} \\
\partial _{t}\Phi & \in L^{2}\left( 0,T;W_{\mu _{\Omega }\oplus \mu _{\Gamma
}}^{-1,2}(\mathbb{R}_{+};\mathbb{V}^{1})\right) .
\label{eq:problem-p-1bbbis}
\end{align}%
$\left( U,\Phi ^{t}\right) $ is said to be a weak solution to problem
\textbf{P} if $v\left( t\right) ={\mathrm{tr_{D}}}\left( u\left( t\right)
\right) $ and $\xi ^{t}={\mathrm{tr_{D}}}\left( \eta ^{t}\right) $ for
almost all $t\in (0,T]$, and $\left( U\left( t\right) ,\Phi ^{t}\right) $
satisfies, for almost all $t\in (0,T]$,%
\begin{equation}
\begin{array}{ll}
\left\langle \partial _{t}U(t),\Xi \right\rangle _{\mathbb{X}%
^{2}}+\left\langle \mathrm{A_{W}^{0,\beta ,\nu ,\omega }}U(t),\Xi
\right\rangle _{\mathbb{X}^{2}}+\int_{0}^{\infty }\mu _{\Omega
}(s)\left\langle \mathrm{A_{W}^{\alpha ,0,0,\omega }}\Phi ^{t}\left(
s\right) ,\Xi \right\rangle _{\mathbb{X}^{2}}ds &  \\
+\nu \int_{0}^{\infty }\mu _{\Gamma }(s)\left\langle \mathrm{C}\xi
^{t}\left( s\right) ,\varsigma _{\mid \Gamma }\right\rangle _{L^{2}\left(
\Gamma \right) }ds+\left\langle F\left( U(t)\right) ,\Xi \right\rangle _{%
\mathbb{X}^{2}}=0, &  \\
\left\langle \partial _{t}\eta ^{t},\rho \right\rangle _{\mathcal{M}_{\Omega
}^{1}}=\left\langle -\frac{d}{ds}\eta ^{t},\rho \right\rangle _{\mathcal{M}%
_{\Omega }^{1}}+\left\langle u(t),\rho \right\rangle _{\mathcal{M}_{\Omega
}^{1}}, &  \\
\left\langle \partial _{t}\xi ^{t},\rho _{\mid \Gamma }\right\rangle _{%
\mathcal{M}_{\Gamma }^{1}}=\left\langle -\frac{d}{ds}\xi ^{t},\rho _{\mid
\Gamma }\right\rangle _{\mathcal{M}_{\Gamma }^{1}}+\left\langle v(t),\rho
_{\mid \Gamma }\right\rangle _{\mathcal{M}_{\Gamma }^{1}}, &
\end{array}
\label{eq:problem-p-2}
\end{equation}%
for all $\Xi =\binom{\varsigma }{\varsigma _{\mid \Gamma }}\in \mathbb{V}%
^{1}\oplus \left( L^{r_{1}}(\Omega )\times L^{r_{2}}(\Gamma )\right) $, all $%
\Pi =\binom{\rho }{\rho _{\mid \Gamma }}\in \mathcal{M}_{\Omega ,\Gamma
}^{1} $ and%
\begin{equation}
U\left( 0\right) =U_{0}=\left( u_{0},v_{0}\right) ^{{\mathrm{tr}}},\text{ }%
\Phi ^{0}=\Phi _{0}=\left( \eta _{0},\xi _{0}\right) ^{{\mathrm{tr}}}.
\label{eq:problem-p-3}
\end{equation}%
Above, we have set $F:\mathbb{R}^{2}\rightarrow \mathbb{R}^{2},$%
\begin{equation*}
F\left( U\right) :=\binom{f\left( u\right) }{\widetilde{g}\left( v\right) },
\end{equation*}%
with $\widetilde{g}$ defined as in (\ref{tilda-f}). The function $[0,T]\ni
t\mapsto (U(t),\Phi ^{t})$ is called a global weak solution if it is a weak
solution for every $T>0$.
\end{definition}

In the sequel, if the initial datum $\left( U_{0},\Phi _{0}\right) $ is more
smooth, the following notion of strong solution will also become important.

\begin{definition}
\label{weak2}Let $\alpha ,\beta >0$, $\omega ,\nu \in (0,1)$ and $T>0$.
Given $\binom{u_{0}}{v_{0}}\in \mathbb{V}^{1}$, $\binom{\eta _{0}}{\xi _{0}}%
\in \mathcal{M}_{\Omega ,\Gamma }^{2}$, the pair of functions $U\left(
t\right) =\binom{u\left( t\right) }{v\left( t\right) },$ $\Phi ^{t}=\binom{%
\eta ^{t}}{\xi ^{t}}$ satisfying
\begin{align}
U& \in L^{\infty }\left( 0,T;\mathbb{V}^{1}\right) \cap L^{2}(0,T;\mathbb{V}%
^{2}),\text{ } \\
\Phi & \in L^{\infty }(0,T;\mathcal{M}_{\Omega ,\Gamma }^{2}),  \notag \\
\partial _{t}U& \in L^{\infty }\left( 0,T;(\mathbb{V}^{1})^{\ast }\right)
\cap L^{2}\left( 0,T;\mathbb{X}^{2}\right) ,  \notag \\
\partial _{t}\Phi & \in L^{2}\left( 0,T;L_{\mu _{\Omega }\oplus \mu _{\Gamma
}}^{2}\left( \mathbb{R}_{+};\mathbb{X}^{2}\right) \right) ,  \notag
\end{align}%
is called a strong solution to problem \textbf{P} if $v\left( t\right) ={%
\mathrm{tr_{D}}}\left( u\left( t\right) \right) $ and $\xi ^{t}={\mathrm{%
tr_{D}}}\left( \eta ^{t}\right) $ for almost all $t\in (0,T]$, and
additionally, $\left( U\left( t\right) ,\Phi ^{t}\right) $ satisfies (\ref%
{eq:problem-p-2}), a.e. for $t\in (0,T]$, for all $\Xi \in \mathbb{V}^{1}$, $%
\Pi \in \mathcal{M}_{\Omega ,\Gamma }^{1}$, and%
\begin{equation}
U\left( 0\right) =U_{0}=\left( u_{0},v_{0}\right) ^{{\mathrm{tr}}},\text{ }%
\Phi ^{0}=\Phi _{0}=\left( \eta _{0},\xi _{0}\right) ^{{\mathrm{tr}}}.
\end{equation}%
The function $[0,T]\ni t\mapsto (U(t),\Phi ^{t})$ is called a global strong
solution if it is a strong solution for every $T>0$.
\end{definition}

\begin{remark}
\label{qw}Note that a strong solution is incidently more smooth than a weak
solution in the sense of Definition \ref{weak}. Moreover, on account of
standard embedding theorems the regularity $U\in L^{\infty }\left( 0,T;%
\mathbb{V}^{1}\right) \cap L^{2}(0,T;\mathbb{V}^{2})$ implies that%
\begin{equation*}
u\in L^{\infty }\left( 0,T;L^{6}\left( \Omega \right) \right) \cap
L^{q}\left( 0,T;L^{p}\left( \Omega \right) \right)
\end{equation*}%
for any $p\in \left( 6,\infty \right) $ , $1\leq q\leq 4p/\left( p-6\right) $%
, and ${\mathrm{tr_{D}}}\left( u\right) \in L^{\infty }\left(
0,T;L^{s}\left( \Omega \right) \right) $, for any $s\in \left( 1,\infty
\right) $.
\end{remark}

Another notion of strong solution to problem \textbf{P}, although weaker
than the notion in Definition \ref{weak2}, can be introduced as follows.

\begin{definition}
\label{d:strong-solution} The pair $U=\binom{u}{v}$ and $\Phi =\binom{\eta }{%
\xi }$ is called a quasi-strong solution of problem \textbf{P} on $[0,T)$ if
$(U(t),\Phi ^{t})$ satisfies the equations (\ref{eq:problem-p-2})-(\ref%
{eq:problem-p-3}) for all $\Xi \in \mathbb{V}^{1}$, $\Pi \in \mathcal{M}%
_{\Omega ,\Gamma }^{1}$, almost everywhere on $\left( 0,T\right) $ and if it
has the regularity properties:
\begin{eqnarray}
U &\in &L^{\infty }(0,T;\mathbb{V}^{1})\cap W^{1,2}(0,T;\mathbb{V}^{1}),
\label{eq:strong-defn-1} \\
\Phi &\in &L^{\infty }(0,T;D\left( \mathrm{T}\right) ),
\label{eq:strong-defn-2} \\
\partial _{t}U &\in &L^{\infty }\left( 0,T;\mathbb{X}^{2}\right) ,
\label{eq:strong-defn-3} \\
\partial _{t}\Phi &\in &L^{\infty }\left( 0,T;\mathcal{M}_{\Omega ,\Gamma
}^{1}\right) .  \label{eq:strong-defn-4}
\end{eqnarray}%
As before, the function $[0,T]\ni t\mapsto (U(t),\Phi ^{t})$ is called a
global quasi-strong solution if it is a quasi-strong solution for every $T>0$%
.
\end{definition}

Our first result in this section is contained in the following theorem. It
allows us to obtain generalized solutions in the sense of Definition \ref%
{weak}.

\begin{theorem}
\label{t:weak-solutions}Assume (\ref{miu1})-(\ref{miu3}) and (\ref%
{eq:assumption-4})-(\ref{eq:assumption-5}) hold. For each $\alpha ,\beta >0$%
, $\omega ,\nu \in (0,1)$ and $T>0$, and for any $U_{0}=(u_{0},v_{0})^{{%
\mathrm{tr}}}\in \mathbb{X}^{2}$, $\Phi _{0}=(\eta _{0},\xi _{0})^{{\mathrm{%
tr}}}\in \mathcal{M}_{\Omega ,\Gamma }^{1},$ there exists at least one
(global) weak solution $\left( U,\Phi \right) \in C(\left[ 0,T\right] ;%
\mathcal{H}_{\Omega ,\Gamma }^{0,1})$\ to problem \textbf{P}.
\end{theorem}

\begin{proof}
The proof is divided into several steps. Much of the motivation for the
above theorem comes from \cite{Gal12-2}. Indeed, the dissipativity induced
by the balance condition (\ref{eq:assumption-5}) will be exploited to obtain
an \emph{apriori} bound. Of course, several modifications need to be made in
order to incorporate the dynamic boundary conditions with memory into the
framework.

\underline{Step 1}. (An \emph{apriori} bound) To begin, we derive an \emph{%
apriori} energy estimate for any (sufficiently) smooth solution $(U,\Phi )$
of problem \textbf{P}. Under the assumptions of the theorem, we claim that
the following estimate holds:%
\begin{align}
& \Vert U(t)\Vert _{\mathbb{X}^{2}}^{2}+\left\Vert \Phi ^{t}\right\Vert _{%
\mathcal{M}_{\Omega ,\Gamma }^{1}}^{2}-2\left\langle \mathrm{T}\Phi^t
,\Phi^t \right\rangle _{\mathcal{M}_{\Omega ,\Gamma
}^{1}}+2\int_{0}^{t}\left( \Vert U(\tau )\Vert _{\mathbb{V}^{1}}^{2}+\Vert
u(\tau )\Vert _{L^{r_{1}}(\Omega )}^{r_{1}}\right) d\tau
\label{eq:a-priori-estimate} \\
& \leq C_{T}\left( 1+\Vert U(0)\Vert _{\mathbb{X}^{2}}^{2}+\left\Vert \Phi
^{0}\right\Vert _{\mathcal{M}_{\Omega ,\Gamma }^{1}}^{2}\right) ,  \notag
\end{align}%
for all $t\in \lbrack 0,T]$, for some constant $C>0$, independent of $%
(U,\Phi )$ and $t$.

We now show (\ref{eq:a-priori-estimate}). In Definition \ref{weak} we are
allowed to take, for almost all $t\in \lbrack 0,T]$,
\begin{equation*}
\Xi =U(t)=\left( u(t),u(t)_{\mid \Gamma }\right) ^{{\mathrm{tr}}}\in \mathbb{%
V}^{1}\cap \left( L^{r_{1}}(\Omega )\times L^{r_{2}}(\Gamma )\right)
\end{equation*}%
and%
\begin{equation*}
\Pi =\Phi ^{t}=\left( \eta ^{t},\xi ^{t}\right) ^{{\mathrm{tr}}}\in \mathcal{%
M}_{\Omega ,\Gamma }^{1}.
\end{equation*}%
Then we obtain the differential identities%
\begin{equation}
\frac{1}{2}\frac{d}{dt}\Vert U\Vert _{\mathbb{X}^{2}}^{2}+\left\langle
\mathrm{A_{W}^{0,\beta ,\nu ,\omega }}U,U\right\rangle _{\mathbb{X}%
^{2}}+\left\langle \Phi ^{t},U\right\rangle _{\mathcal{M}_{\Omega ,\Gamma
}^{1}}+\left\langle F(U),U\right\rangle _{\mathbb{X}^{2}}=0,
\label{eq:a-priori-estimate-1}
\end{equation}%
where%
\begin{align}
\left\langle \Phi ^{t},U\right\rangle _{\mathcal{M}_{\Omega ,\Gamma }^{1}}&
=\omega \int_{0}^{\infty }\mu _{\Omega }(s)\left( \left\langle \nabla \eta
^{t}\left( s\right) ,\nabla u\right\rangle _{L^{2}\left( \Omega \right)
}+\alpha \left\langle \eta ^{t}\left( s\right) ,u\right\rangle _{L^{2}\left(
\Omega \right) }\right) ds  \label{identity*} \\
& +\nu \int_{0}^{\infty }\mu _{\Gamma }(s)\left( \left\langle \nabla
_{\Gamma }\xi ^{t}\left( s\right) ,\nabla _{\Gamma }u\right\rangle
_{L^{2}\left( \Gamma \right) }+\beta \left\langle \xi ^{t}\left( s\right)
,u\right\rangle _{L^{2}\left( \Gamma \right) }\right) ds  \notag \\
& =\int_{0}^{\infty }\mu _{\Omega }(s)\left\langle \mathrm{A_{W}^{\alpha
,0,0,\omega }}\Phi ^{t}\left( s\right) ,U\right\rangle _{\mathbb{X}%
^{2}}ds+\nu \int_{0}^{\infty }\mu _{\Gamma }(s)\left\langle \mathrm{C}\xi
^{t}\left( s\right) ,u\right\rangle _{L^{2}\left( \Gamma \right) }ds,  \notag
\end{align}%
and
\begin{equation}
\ \frac{1}{2}\frac{d}{dt}\Vert \Phi ^{t}\Vert _{\mathcal{M}_{\Omega ,\Gamma
}^{1}}^{2}=\left\langle \mathrm{T}\Phi ^{t},\Phi ^{t}\right\rangle _{%
\mathcal{M}_{\Omega ,\Gamma }^{1}}+\left\langle U,\Phi ^{t}\right\rangle _{%
\mathcal{M}_{\Omega ,\Gamma }^{1}},  \label{eq:a-priori-estimate-1bis}
\end{equation}%
which hold for almost all $t\in \lbrack 0,T]$. Adding these identities
together and recalling (\ref{eq:operator-T-1}), we obtain%
\begin{align}
& \frac{1}{2}\frac{d}{dt}\left( \Vert U\Vert _{\mathbb{X}^{2}}^{2}+\Vert
\Phi ^{t}\Vert _{\mathcal{M}_{\Omega ,\Gamma }^{1}}^{2}\right) -\left\langle
\mathrm{T}\Phi^t,\Phi^t \right\rangle _{\mathcal{M}_{\Omega ,\Gamma
}^{1}}+\left( \omega \Vert \nabla u\Vert _{L^{2}\left( \Omega \right)
}^{2}+\nu \Vert \nabla _{\Gamma }u\Vert _{L^{2}(\Gamma )}^{2}+\beta
\left\Vert u\right\Vert _{L^{2}\left( \Gamma \right) }^{2}\right)
\label{eq:a-priori-estimate-1tris} \\
& \leq -\left\langle f(u),u\right\rangle _{L^{2}\left( \Omega \right)
}-\left\langle \widetilde{g}\left( u\right) ,u\right\rangle _{L^{2}\left(
\Gamma \right) }.  \notag
\end{align}%
Following \cite[(2.22)]{Gal12-2} and \cite[(3.11)]{RBT01}, we estimate the
product with $F$ on the right-hand side of (\ref{eq:a-priori-estimate-1tris}%
), as follows:%
\begin{align}
\left\langle F(U),U\right\rangle _{\mathbb{X}^{2}}& =\left\langle
f(u),u\right\rangle _{L^{2}\left( \Omega \right) }+\left\langle \widetilde{g}%
\left( u\right) ,u\right\rangle _{L^{2}\left( \Gamma \right) }
\label{eq:a-priori-estimate-2} \\
& =\int_{\Omega }\left( f(u)u+\frac{|\Gamma |}{|\Omega |}\widetilde{g}%
(u)u\right) dx-\frac{|\Gamma |}{|\Omega |}\int_{\Omega }\left( \widetilde{g}%
(u)u-\frac{1}{|\Gamma |}\int_{\Gamma }\widetilde{g}(u)u\mathrm{d}\sigma
\right) dx.  \notag
\end{align}%
Exploiting Poincar\'{e} inequality (\ref{eq:Poincare}) and Young's
inequality, we see that for all $\varepsilon \in (0,\omega )$,
\begin{align}
\frac{|\Gamma |}{|\Omega |}\int_{\Omega }\left( \widetilde{g}(u)u-\frac{1}{%
|\Gamma |}\int_{\Gamma }\widetilde{g}(u)ud\sigma \right) dx & \leq C_{\Omega
}\frac{|\Gamma |}{|\Omega |}\int_{\Omega }|\nabla (\widetilde{g}(u)u)|dx
\label{eq:a-priori-estimate-4} \\
& =C_{\Omega }\frac{|\Gamma |}{|\Omega |}\int_{\Omega }|\nabla u(\widetilde{g%
}^{\prime }(u)u+\widetilde{g}(u))|dx  \notag \\
& \leq \varepsilon \Vert \nabla u\Vert _{L^{2}(\Omega )}^{2}+\frac{C_{\Omega
}^{2}|\Gamma |^{2}}{4\varepsilon |\Omega |^{2}}\int_{\Omega }|\widetilde{g}%
^{\prime }(u)u+\widetilde{g}(u)|^{2}dx.  \notag
\end{align}%
Combining (\ref{eq:a-priori-estimate-2})-(\ref{eq:a-priori-estimate-4}) and
applying assumption (\ref{eq:assumption-5}) yields%
\begin{equation}
\left\langle F(U),U\right\rangle _{\mathbb{X}^{2}}\geq \delta\Vert u\Vert
_{L^{r_{1}}(\Omega )}^{r_{1}}-\varepsilon \Vert \nabla u\Vert _{L^{2}(\Omega
)}^{2}-C_{\delta},  \label{eq:a-priori-estimate-4bis}
\end{equation}%
for some positive constants $\delta$ and $C_{\delta}$ that are independent
of $U$, $t$ and $\varepsilon $. Plugging (\ref{eq:a-priori-estimate-4bis})
into (\ref{eq:a-priori-estimate-1tris}) gives, for almost all $t\in \lbrack
0,T]$,%
\begin{align}
& \frac{1}{2}\frac{d}{dt}\left( \Vert U\Vert _{\mathbb{X}^{2}}^{2}+\Vert
\Phi ^{t}\Vert _{\mathcal{M}_{\varepsilon }^{0}}^{2}\right) -\left\langle
\mathrm{T}\Phi ^t,\Phi^t \right\rangle _{\mathcal{M}_{\Omega ,\Gamma
}^{1}}+\left( \omega -\varepsilon \right) \Vert \nabla u\Vert _{L^{2}\left(
\Omega \right) }^{2}  \label{eq:a-priori-estimate-3} \\
& +\left( \nu \Vert \nabla _{\Gamma }u\Vert _{L^{2}(\Gamma )}^{2}+\beta
\left\Vert u\right\Vert _{L^{2}\left( \Gamma \right) }^{2}\right)
+\delta\Vert u\Vert _{L^{r_{1}}(\Omega )}^{r_{1}}  \notag \\
& \leq C.  \notag
\end{align}%
Integrating (\ref{eq:a-priori-estimate-3} over the interval $\left(
0,t\right) $ yields the desired estimate (\ref{eq:a-priori-estimate}).
Additionally, from the above \emph{apriori} estimate (\ref%
{eq:a-priori-estimate}), we immediately see that
\begin{align}
U& \in L^{\infty }\left(0,T;\mathbb{X}^{2}\right)\cap L^{2}\left( 0,T;%
\mathbb{V}^{1}\right) ,  \label{eq:a-priori-bound-11} \\
\Phi & \in L^{\infty }\left(0,T;\mathcal{M}_{\Omega ,\Gamma }^{1}\right),
\label{eq:a-priori-bound-12} \\
u& \in L^{r_{1}}\left(\Omega \times (0,T)\right).
\label{eq:a-priori-bound-13}
\end{align}%
Applying Lemma \ref{t:appendix-lemma-1}, in view of of (\ref%
{eq:a-priori-bound-11}) and (\ref{eq:a-priori-bound-13}), we also get%
\begin{equation}
u\in L^{r_{2}}(\Gamma \times (0,T)).  \label{eq:a-priori-bound-15}
\end{equation}%
Thus, we indeed recover the bounds (\ref{eq:problem-p-1})-(\ref%
{eq:problem-p-1bis}) through estimate (\ref{eq:a-priori-estimate}).
Moreover, we have from (\ref{eq:a-priori-bound-13}) and (\ref%
{eq:a-priori-bound-15}) that $f\left( u\right) \in L^{r_{1}^{\prime
}}(\Omega \times \left( 0,T\right) )$, $\widetilde{g}\left( v\right) \in
L^{r_{2}^{\prime }}(\Gamma \times \left( 0,T\right) )$; hence,%
\begin{equation}
F(U)\in L^{r_{1}^{\prime }}(\Omega \times \left( 0,T\right) )\times
L^{r_{2}^{\prime }}(\Gamma \times \left( 0,T\right) ).
\label{eq:a-priori-bound-16}
\end{equation}%
Clearly, since $U\in L^{2}(0,T;\mathbb{V}^{1})$ and $\Phi \in L^{\infty
}(0,T;\mathcal{M}_{\Omega ,\Gamma }^{1})$ we also have $\mathrm{%
A_{W}^{0,\beta ,\nu ,\omega }}\Phi (s)\in L^{2}(0,T;(\mathbb{V}^{1})^{\ast
}) $ for almost all $s\in \mathbb{R}_{+},$ and $\mathrm{A_{W}^{0,\beta ,\nu
,\omega }}U\in L^{2}(0,T;(\mathbb{V}^{1})^{\ast }),$ respectively.
Therefore, after comparing terms in the first equation of (\ref%
{eq:problem-p-2}), we see that
\begin{equation}
\partial _{t}U\in L^{2}\left( 0,T;(\mathbb{V}^{1})^{\ast }\right) \oplus
\left( L^{r_{1}^{\prime }}(\Omega \times \left( 0,T\right) )\times
L^{r_{2}^{\prime }}(\Gamma \times \left( 0,T\right) )\right) .
\label{eq:a-priori-bound-17}
\end{equation}%
Hence, this justifies our choice of test function for the first of (\ref%
{eq:problem-p-2}). Concerning the second equation of (\ref{eq:problem-p-2}),
in view of (\ref{eq:a-priori-bound-11}) and the representation formula (\ref%
{eq:representation-formula-1}) we have%
\begin{equation*}
\mathrm{T}\Phi^{t}(s)=-\partial _{s}\Phi ^{t}(s)=\left\{
\begin{array}{ll}
-U(t-s) & \text{for}~0<s\leq t, \\
-\partial _{s}\Phi _{0}(s-t)+U(t-s) & \text{for}~s>t.%
\end{array}%
\right.
\end{equation*}%
Then, with a given $\Phi _{0}\in \mathcal{M}_{\Omega ,\Gamma }^{1},$ $%
\partial _{s}\Phi _{0}(\cdot )\in W_{\mu _{\Omega }\oplus \mu _{\Gamma
}}^{-1,2}\left( \mathbb{R}_{+};\mathbb{V}^{1}\right) ,$ we conclude%
\begin{equation}
\partial _{t}\Phi \in L^{2}\left(0,T;W_{\mu _{\Omega }\oplus \mu _{\Gamma
}}^{-1,2}\left( \mathbb{R}_{+};\mathbb{V}^{1}\right)\right).
\label{eq:a-priori-bound-18}
\end{equation}%
This concludes Step 1.

\underline{Step 2}. (A Galerkin basis) First, for any $\alpha ,\beta \geq $ $%
0$ we recall that ($\mathrm{A_{W}^{\alpha ,\beta ,\nu ,\omega })}^{-1}\in
\mathcal{L}\left( \mathbb{X}^{2}\right) $ is compact provided that either $%
\beta >0$ or $\alpha >0$. This means that, for $i\in \mathbb{N}$, there is a
complete system of eigenfunctions $\Psi _{i}^{\alpha ,\beta ,\nu ,\omega
}=(\vartheta _{i}^{\alpha ,\beta ,\nu ,\omega },\vartheta _{i\mid \Gamma
}^{\alpha ,\beta ,\nu ,\omega })^{\mathrm{tr}}$ of the eigenvalue problem
\begin{equation*}
\mathrm{A_{W}^{\alpha ,\beta ,\nu ,\omega }}\Psi _{i}^{\alpha ,\beta ,\nu
,\omega }=\lambda _{i}\Psi _{i}^{\alpha ,\beta ,\nu ,\omega }\text{ in }%
\mathbb{X}^{2}
\end{equation*}%
with
\begin{equation*}
\Psi _{i}^{\alpha ,\beta ,\nu ,\omega }\in D\left(\mathrm{A_{W}^{\alpha
,\beta ,\nu ,\omega }}\right)\cap \left( C^{2}({\overline{\Omega }})\times
C^{2}\left( \Gamma \right) \right) ,
\end{equation*}%
see \cite[Appendix]{Gal12-3}. The eigenvalues $\lambda _{i}=\lambda
_{i}^{\alpha ,\beta ,\nu ,\omega }\in (0,\infty )$ may be put into
increasing order and counted according to their multiplicity to form a
divergent sequence going to infinity. In addition, also due to standard
spectral theory, the related eigenfunctions form an orthogonal basis in $%
\mathbb{V}^{1}$ that is orthonormal in $\mathbb{X}^{2}$. Note that for each $%
i\in \mathbb{N}$, the pair $\left( \lambda _{i},\vartheta _{i}\right) \in
\mathbb{R}_{+}\times C^{2}\left( \overline{\Omega }\right) ,$ $\vartheta
_{i}=\vartheta _{i}^{\alpha ,\beta ,\nu ,\omega },$ is a classical solution
of the elliptic problem%
\begin{equation}
\left\{
\begin{array}{ll}
-\omega \Delta \vartheta _{i}+\alpha \omega \vartheta _{i}=\lambda
_{i}\vartheta _{i}, & \text{in }\Omega , \\
-\nu \Delta _{\Gamma }\left( \vartheta _{i\mid \Gamma }\right) +\omega
\partial _{n}\vartheta _{i}+\beta \nu \vartheta _{i\mid \Gamma }=\lambda
_{i}\vartheta _{i\mid \Gamma }, & \text{on }\Gamma .%
\end{array}%
\right.  \label{eigen-p}
\end{equation}%
It remains to select an orthonormal basis $\{\zeta _{i}\}_{i=1}^{\infty }$
of $\mathcal{M}_{\Omega ,\Gamma }^{1}=L_{\mu _{\Omega }\oplus \mu _{\Gamma
}}^{2}(\mathbb{R}_{+};\mathbb{V}^{1})$ that also belongs to $D(\mathrm{T}%
)\cap W_{\mu _{\Omega }\oplus \mu _{\Gamma }}^{1,2}(\mathbb{R}_{+};\mathbb{V}%
^{1})$. We can choose vectors $\zeta _{i}=\varkappa _{i}\Psi _{i}^{\alpha
,\beta ,\nu ,\omega }$, with eigenvectors $\Psi _{i}^{\alpha ,\beta ,\nu
,\omega }\in D(\mathrm{A_{W}^{\alpha ,\beta ,\nu ,\omega }})$ satisfying (%
\ref{eigen-p}) above, such that $\{\varkappa _{i}\}_{i=1}^{\infty }\in
C_{c}^{\infty }(\mathbb{R}_{+})$ is an orthonormal basis for $L_{\mu
_{\Omega }\oplus \mu _{\Gamma }}^{2}(\mathbb{R}_{+})$. This choice will be
crucial for the derivation of strong solutions in the section later.

Let $T>0$ be fixed. For $n\in \mathbb{N}$, set the spaces
\begin{equation*}
X_{n}=\mathrm{span}\left\{ \Psi _{1}^{\alpha ,\beta ,\nu ,\omega },\dots
,\Psi _{n}^{\alpha ,\beta ,\nu ,\omega }\right\} \subset \mathbb{X}%
^{2},~~X_{\infty }=\bigcup_{n=1}^{\infty }X_{n},
\end{equation*}%
and
\begin{equation*}
M_{n}=\mathrm{span}\left\{ \zeta _{1},\zeta _{2},\dots ,\zeta _{n}\right\}
\subset \mathcal{M}_{\Omega ,\Gamma }^{1},~~M_{\infty
}=\bigcup_{n=1}^{\infty }M_{n}.
\end{equation*}%
Obviously, $X_{\infty }$ is a dense subspace of $\mathbb{V}^{1}$. For each $%
n\in \mathbb{N}$, let $P_{n}:\mathbb{X}^{2}\rightarrow X_{n}$ denote the
orthogonal projection of $\mathbb{X}^{2}$ onto $X_{n}$ and let $Q_{n}:%
\mathcal{M}_{\Omega ,\Gamma }^{1}\rightarrow M_{n}$ denote the orthogonal
projection of $\mathcal{M}_{\Omega ,\Gamma }^{1}$ onto $M_{n}$. Thus, we
seek functions of the form
\begin{equation}
U_{n}(t)=\sum_{i=1}^{n}a_{i}(t)\Psi _{i}^{\alpha ,\beta ,\nu ,\omega }~~%
\text{and}~~\Phi _{n}^{t}(s)=\sum_{i=1}^{n}b_{i}(t)\zeta
_{i}(s)=\sum_{i=1}^{n}b_{i}(t)\varkappa _{i}\left( s\right) \Psi
_{i}^{\alpha ,\beta ,\nu ,\omega }  \label{eq:approximate-1}
\end{equation}%
that will satisfy the associated discretized problem \textbf{P}$_{n}$
described below. The functions $a_{i}$ and $b_{i}$ are assumed to be (at
least) $C^{2}(0,T)$ for $i=1,\dots ,n$. By definition, note that%
\begin{equation}
u_{n}(t)=\sum_{i=1}^{n}a_{i}(t)\vartheta _{i}^{\alpha ,\beta ,\nu ,\omega }~~%
\text{and}~~u_{n}(t)_{\mid \Gamma }=\sum_{i=1}^{n}a_{i}(t)\vartheta _{i\mid
\Gamma }^{\alpha ,\beta ,\nu ,\omega },  \label{eq:approximate-3}
\end{equation}%
also
\begin{equation}
\eta _{n}^{t}(s)=\sum_{i=1}^{n}b_{i}(t)\zeta _{i}(s)~~\text{and}~~\xi
_{n}^{t}(s)=\sum_{i=1}^{n}b_{i}(t)\zeta _{i}(s)_{\mid \Gamma }.
\label{eq:approximate-4}
\end{equation}%
As usual, to approximate the given initial data $U_{0}\in \mathbb{X}^{2}$
and $\Phi _{0}\in \mathcal{M}_{\Omega ,\Gamma }^{1}$, we take $U_{n0}\in
\mathbb{V}^{1}$ such that $U_{n0}\rightarrow U_{0}~~$(in $\mathbb{X}^{2}$),
since $\mathbb{V}^{1}$ is dense in $\mathbb{X}^{2}$, and $\Phi
_{n0}\rightarrow \Phi _{0}~~$(in $\mathcal{M}_{\Omega ,\Gamma }^{1}$).

For $T>0$ and for each integer $n\geq 1$, the weak formulation of the
approximate problem \textbf{P}$_{n}$ is the following: find $(U_{n},\Phi
_{n})$, given by (\ref{eq:approximate-1}) such that, for all ${\overline{U}}%
=(\bar{u},\bar{v})^{\mathrm{tr}}\in X_{n}$ and ${\overline{\Phi }}=(\bar{\eta%
},\bar{\xi})^{\mathrm{tr}}\in M_{n}$, the equations%
\begin{equation}
\left\langle \partial _{t}U_{n},{\overline{U}}\right\rangle _{\mathbb{X}%
^{2}}+\left\langle \mathrm{A_{W}^{0,\beta ,\nu ,\omega }}U_{n},{\overline{U}}%
\right\rangle _{\mathbb{X}^{2}}+\left\langle \Phi _{n}^{t},{\overline{U}}%
\right\rangle _{\mathcal{M}_{\Omega ,\Gamma }^{1}}+\left\langle
P_{n}F(U_{n}),{\overline{U}}\right\rangle _{\mathbb{X}^{2}}=0
\label{eq:approx-ode-1}
\end{equation}%
and
\begin{equation}
\left\langle \partial _{t}\Phi _{n}^{t},{\overline{\Phi }}\right\rangle _{%
\mathcal{M}_{\Omega ,\Gamma }^{1}}=\left\langle \mathrm{T}\Phi_{n}^{t},{%
\overline{\Phi }}\right\rangle _{\mathcal{M}_{\Omega ,\Gamma
}^{1}}+\left\langle U_{n},{\overline{\Phi }}\right\rangle _{\mathcal{M}%
_{\Omega ,\Gamma }^{1}}  \label{eq:approx-ode-2}
\end{equation}%
hold for almost all $t\in \left( 0,T\right) $, subject to the initial
conditions
\begin{equation}
\left\langle U_{n}(0),{\overline{U}}\right\rangle _{\mathbb{X}%
^{2}}=\left\langle U_{n0},{\overline{U}}\right\rangle _{\mathbb{X}^{2}}~~%
\text{and}~~\left\langle \Phi _{n}^{0},{\overline{\Phi }}\right\rangle _{%
\mathcal{M}_{\Omega ,\Gamma }^{1}}=\left\langle \Phi _{n0},{\overline{\Phi }}%
\right\rangle _{\mathcal{M}_{\Omega ,\Gamma }^{1}}.  \label{eq:approx-ode-3}
\end{equation}

To show the existence of at least one solution to (\ref{eq:approx-ode-1})-(%
\ref{eq:approx-ode-3}), we now suppose that $n$ is fixed and we take ${%
\overline{U}}=\Psi _{k}$ and ${\overline{\Phi }}=\zeta _{k}$ for some $1\leq
k\leq n$. Then substituting the discretized functions (\ref{eq:approximate-1}%
) into (\ref{eq:approx-ode-1})-(\ref{eq:approx-ode-3}), we easily arrive at
a system of ordinary differential equations in the unknowns $a_{k}=a_{k}(t)$
and $b_{k}=b_{k}(t)$ on $X_{n}$ and $M_{n},$ respectively. We need to recall
that%
\begin{equation*}
\langle P_{n}F(U_{n}),U_{k}\rangle =\langle F(U_{n}),P_{n}U_{k}\rangle
=\langle F(U_{n}),U_{k}\rangle .
\end{equation*}%
Since $f,$ $g\in C^{1}(\mathbb{R})$, we may apply Cauchy's theorem for ODEs
to find that there is $T_{n}\in (0,T)$ such that $a_{k},b_{k}\in
C^{2}(0,T_{n})$, for $1\leq k\leq n$ and both (\ref{eq:approx-ode-1}) and (%
\ref{eq:approx-ode-2}) hold in the classical sense for all $t\in \lbrack
0,T_{n}]$. This argument shows the existence of at least one local solution
to problem \textbf{P}$_{n}$ and ends Step 2.

\underline{Step 3}. (Boundedness and continuation of approximate maximal
solutions) Now we apply the (uniform) \emph{apriori} estimate (\ref%
{eq:a-priori-estimate}) which also holds for any approximate solution $%
(U_{n},\Phi _{n})$ of problem \textbf{P}$_{n}$ on the interval $[0,T_{n})$,
where $T_{n}<T$. Owing to the boundedness of the projectors $P_{n}$ and $%
Q_{n}$ on the corresponding spaces, we infer%
\begin{align}
& \Vert U_{n}(t)\Vert _{\mathbb{X}^{2}}^{2}+\left\Vert \Phi
_{n}^{t}\right\Vert _{\mathcal{M}_{\Omega ,\Gamma }^{1}}^{2}-2\left\langle
\mathrm{T}\Phi _{n}^{t},\Phi _{n}^{t}\right\rangle _{\mathcal{M}_{\Omega
,\Gamma }^{1}}+2\int_{0}^{t}\left( \Vert U_{n}(\tau )\Vert _{\mathbb{V}%
^{1}}^{2}+\Vert u_{n}(\tau )\Vert _{L^{r_{1}}(\Omega )}^{r_{1}}\right) d\tau
\label{eq:a-priori-bound-1} \\
& \leq C_{T}\left( 1+\Vert U(0)\Vert _{\mathbb{X}^{2}}^{2}+\left\Vert \Phi
^{0}\right\Vert _{\mathcal{M}_{\Omega ,\Gamma }^{1}}^{2}\right) ,  \notag
\end{align}%
for some constant $C_{T}>0$ independent of $n$ and $t$. Hence, every
approximate solution may be extended to the whole interval $[0,T]$, and
because $T>0$ is arbitrary, any approximate solution is a global one. As in
Step 1, we also obtain the uniform bounds (\ref{eq:a-priori-bound-11})-(\ref%
{eq:a-priori-bound-18}) for each approximate solution $(U_{n},\Phi _{n})$.
Thus,%
\begin{align}
U_{n}& ~\text{is uniformly bounded in}~L^{\infty }\left(0,T;\mathbb{X}%
^{2}\right),  \label{eq:uniform-bounds-1} \\
U_{n}& ~\text{is uniformly bounded in}~L^{2}\left(0,T;\mathbb{V}^{1}\right),
\label{eq:uniform-bounds-2} \\
u_{n}& ~\text{is uniformly bounded in}~L^{r_{1}}(\Omega \times \left(
0,T\right) ),  \label{eq:uniform-bounds-3} \\
u_{n}& ~\text{is uniformly bounded in}~L^{r_{2}}(\Gamma \times \left(
0,T\right) ),  \label{eq:uniform-bounds-4} \\
\Phi _{n}& ~\text{is uniformly bounded in}~L^{\infty }\left(0,T;\mathcal{M}%
_{\Omega ,\Gamma }^{1}\right),  \label{eq:uniform-bounds-5} \\
F(U_{n})& ~\text{is uniformly bounded in}~L^{r_{1}^{\prime }}(\Omega \times
\left( 0,T\right) )\times L^{r_{2}^{\prime }}(\Gamma \times \left(
0,T\right) ),  \label{eq:uniform-bounds-6} \\
\partial _{t}U_{n}& ~\text{is uniformly bounded in}~L^{2}\left( 0,T;\left(%
\mathbb{V}^{1}\right)^{\ast }\right) \oplus \left( L^{r_{1}^{\prime
}}(\Omega \times \left( 0,T\right) )\times L^{r_{2}^{\prime }}(\Gamma \times
\left( 0,T\right) )\right) ,  \label{eq:uniform-bounds-7} \\
\partial _{t}\Phi _{n}& ~\text{is uniformly bounded in}~L^{2}\left(0,T;W_{%
\mu _{\Omega }\oplus \mu _{\Gamma }}^{-1,2}\left( \mathbb{R}_{+};\mathbb{V}%
^{1}\right)\right).  \label{eq:uniform-bounds-8}
\end{align}%
This concludes Step 3.

\underline{Step 4}. (Convergence of approximate solutions) By Alaoglu's
theorem (cf. e.g. \cite[Theorem 6.64]{Renardy&Rogers04}) and the uniform
bounds (\ref{eq:uniform-bounds-1})-(\ref{eq:uniform-bounds-6}), there is a
subsequence of $(U_{n},\Phi _{n})$, generally not relabelled, and functions $%
U$ and $\Phi $, obeying (\ref{eq:a-priori-bound-11})-(\ref%
{eq:a-priori-bound-18}), such that as $n\rightarrow \infty $,
\begin{equation}
\begin{array}{ll}
U_{n}\rightharpoonup U & \text{weakly-* in }L^{\infty }\left( 0,T;\mathbb{X}%
^{2}\right) , \\
U_{n}\rightharpoonup U & \text{weakly in }L^{2}\left( 0,T;\mathbb{V}%
^{1}\right) , \\
u_{n}\rightharpoonup u & \text{weakly in }L^{r_{1}}(\Omega \times \left(
0,T\right) ), \\
u_{n}\rightharpoonup u & \text{weakly in }L^{r_{2}}(\Gamma \times \left(
0,T\right) ), \\
\Phi _{n}\rightharpoonup \Phi & \text{weakly-* in }L^{\infty }\left( 0,T;%
\mathcal{M}_{\Omega ,\Gamma }^{1}\right) .%
\end{array}
\label{eq:weak-convergence-5}
\end{equation}%
Moreover, setting $k_{S}:=(-\mu _{S}^{^{\prime }})^{1/2}\geq 0$, $S\in
\left\{ \Omega ,\Gamma \right\} $ we have%
\begin{equation}
\partial _{t}U_{n}\rightharpoonup \partial _{t}U~\text{weakly in }%
L^{2}\left( 0,T;\left( \mathbb{V}^{1}\right) ^{\ast }\right) \oplus \left(
L^{r_{1}^{\prime }}(\Omega \times \left( 0,T\right) )\times L^{r_{2}^{\prime
}}(\Gamma \times \left( 0,T\right) )\right) ,
\label{eq:weak-convergence-5bis}
\end{equation}%
\begin{equation}
\Phi _{n}\rightharpoonup \Phi \text{ weakly in }L^{2}\left( 0,T;L_{k_{\Omega
}\oplus k_{\Gamma }}^{2}\left( \mathbb{R}_{+};\mathbb{V}^{1}\right) \right) ,
\label{eq:weak-convergence-5biss}
\end{equation}%
owing to the bound on $\left\langle \mathrm{T}\Phi _{n},\Phi
_{n}\right\rangle _{\mathcal{M}_{\Omega ,\Gamma }^{1}}$ from (\ref%
{eq:a-priori-bound-1}) and%
\begin{equation}
\partial _{t}\Phi _{n}\rightarrow \partial _{t}\Phi ~\text{weakly in}%
~L^{2}\left( 0,T;W_{\mu _{\Omega }\oplus \mu _{\Gamma }}^{-1,2}\left(
\mathbb{R}_{+};\mathbb{V}^{1}\right) \right) .
\label{eq:weak-convergence-5tris}
\end{equation}%
Indeed, we observe that the last of (\ref{eq:weak-convergence-5}) and
integration by parts yield, for any $\zeta \in C_{0}^{\infty }\left(
J;C_{0}^{\infty }\left( \mathbb{R}_{+};\mathbb{V}^{1}\right) \right) ,$
\begin{equation*}
\int_{0}^{T}\left\langle \partial _{t}\Phi _{n}^{y},\zeta \right\rangle _{%
\mathcal{M}_{\Omega ,\Gamma }^{1}}dy=-\int_{0}^{T}\left\langle \Phi
_{n}^{y},\partial _{t}\zeta \right\rangle _{\mathcal{M}_{\Omega ,\Gamma
}^{1}}dy\;\rightarrow \;-\int_{0}^{T}\left\langle \Phi ^{y},\partial
_{t}\zeta \right\rangle _{\mathcal{M}_{\Omega ,\Gamma }^{1}}dy,
\end{equation*}%
and that $\Phi ^{t}\in C(0,T;W_{\mu _{\Omega }\oplus \mu _{\Gamma }}^{-1,2}(%
\mathbb{R}_{+};\mathbb{V}^{1}))$. We can exploit the second of (\ref%
{eq:weak-convergence-5}) and (\ref{eq:weak-convergence-5bis}) to deduce
\begin{equation}
U_{n}\rightarrow U~\text{strongly in}~L^{2}\left( 0,T;\mathbb{X}^{2}\right) ,
\label{eq:weak-convergence-11}
\end{equation}%
by application of the Agmon-Lions compactness criterion since $\mathbb{V}%
^{1} $ is compactly embedded in $\mathbb{X}^{2}$. This last strong
convergence property is enough to pass to the limit in the nonlinear terms
since $f$, $g\in C^{1}$ (see, e.g., \cite{Gal12-2, Gal&Warma10}). Indeed, on
account of standard arguments (cf. also \cite{CGGM10})\ we have%
\begin{equation}
P_{n}F(U_{n})\rightharpoonup F\left( U\right) ~\text{weakly in}~L^{2}\left(
0,T;\mathbb{X}^{2}\right) .  \label{eq:weak-convergence-12}
\end{equation}%
The convergence properties (\ref{eq:weak-convergence-5})-(\ref%
{eq:weak-convergence-11}) allow us to pass to the limit as $n\rightarrow
\infty $ in equation (\ref{eq:approx-ode-1}) in order to recover (\ref%
{eq:problem-p-2}), using standard density arguments. Indeed, in order to
pass to the limit in the equations for memory, we use (\ref%
{eq:weak-convergence-5biss}) and the following distributional equality%
\begin{align*}
& -\int_{0}^{T}\left\langle \Phi ^{y},\partial _{t}\zeta \right\rangle _{%
\mathcal{M}_{\Omega ,\Gamma }^{1}}dy-\int_{0}^{T}\mu _{\Omega }^{\prime
}\left( s\right) \left\langle \eta ^{y},\zeta \right\rangle _{\mathcal{M}%
_{\Omega }^{1}}dy-\int_{0}^{T}\mu _{\Gamma }^{^{\prime }}\left( s\right)
\left\langle \xi ^{y},\partial _{t}\zeta \right\rangle _{\mathcal{M}_{\Gamma
}^{1}}dy \\
& =\int_{0}^{t}\left\langle \partial _{t}\Phi ^{t}-\mathrm{T}\Phi ^{y},\zeta
\right\rangle _{\mathcal{M}_{\Omega ,\Gamma }^{1}}dy.
\end{align*}%
Thus, we also get the last two equations of (\ref{eq:problem-p-2}) by virtue
of the last of (\ref{eq:weak-convergence-5}).

\underline{Step 5}. (Continuity of the solution) According to the
description for problem \textbf{P}, see (\ref{eq:problem-p-2}), we have%
\begin{equation}
\begin{array}{ll}
\partial _{t}U\in L^{2}\left( 0,T;\left(\mathbb{V}^{1}\right)^{\ast }\right)
\oplus \left( L^{r_{1}^{\prime }}(\Omega \times \left( 0,T\right) )\times
L^{r_{2}^{\prime }}(\Gamma \times \left( 0,T\right) )\right) , &  \\
\partial _{t}\Phi \in L^{2}\left(0,T;W_{\mu _{\Omega }\oplus \mu _{\Gamma
}}^{-1,2}\left( \mathbb{R}_{+};\mathbb{V}^{1}\right)\right). &
\end{array}
\label{eq:weak-bound20}
\end{equation}%
Since the spaces $L^{2}\left( 0,T;(\mathbb{V}^{1})^{\ast }\right) ,$ $%
L^{r_{1}^{\prime }}(\Omega \times \left( 0,T\right) )\times L^{r_{2}^{\prime
}}(\Gamma \times \left( 0,T\right) )$ are the dual of $L^{2}\left( 0,T;%
\mathbb{V}^{1}\right) $ and $L^{r_{1}}(\Omega \times \left( 0,T\right)
)\times L^{r_{2}}(\Gamma \times \left( 0,T\right) )$, respectively,
recalling (\ref{eq:weak-convergence-5}), we can argue exactly as in the
proof of \cite[Proposition 2.5]{Gal12-2} to deduce that $U\in C\left( \left[
0,T\right] ;\mathbb{X}^{2}\right) $. Finally, owing to $U\in $ $L^{2}(0,T;%
\mathbb{V}^{1})$ and Corollary \ref{t:memory-regularity-1}, it follows that $%
\Phi \in C\left( \left[ 0,T\right] ;\mathcal{M}_{\Omega ,\Gamma }^{1}\right)
$. Thus, both $U\left( 0\right) $ and $\Phi \left( 0\right) $ make sense and
the equalities $U\left( 0\right) =U_{0}$ and $\Phi ^{0}=\Phi _{0}$ hold in
the usual sense due to the strong convergence of $U_{0n}\rightarrow U_{0}$
in $\mathbb{X}^{2}$, and $\Phi _{0n}\rightarrow \Phi _{0}$ in $\mathcal{M}%
_{\Omega ,\Gamma }^{1}$, respectively. The proof of the theorem is finished.
\end{proof}

When both the bulk and boundary nonlinearities are dissipative (i.e., (\ref%
{disf1}) holds in place of the balance (\ref{eq:assumption-5})), we also
have the following.

\begin{theorem}
\label{t:weak-solutions-part2}Assume (\ref{miu1})-(\ref{miu3}) and (\ref%
{eq:assumption-4}), (\ref{disf1}) hold. For each $\alpha ,\beta >0$, $\omega
,\nu \in (0,1)$ and $T>0$, and for any $U_{0}=(u_{0},v_{0})^{{\mathrm{tr}}%
}\in \mathbb{X}^{2}$, $\Phi _{0}=(\eta _{0},\xi _{0})^{{\mathrm{tr}}}\in
\mathcal{M}_{\Omega ,\Gamma }^{1},$ there exists at least one (global) weak
solution $\left( U,\Phi \right) \in C(\left[ 0,T\right] ;\mathcal{H}_{\Omega
,\Gamma }^{0,1})$ to problem \textbf{P} in the sense of Definition \ref{weak}%
.
\end{theorem}

\begin{proof}
The proof is essentially the same as the proof of Theorem \ref%
{t:weak-solutions} with the exception that one employs the estimate%
\begin{equation*}
f\left( u\right) u\geq C_{f}\left\vert u\right\vert ^{r_{1}}-C_{1},\text{ }%
\widetilde{g}\left( u\right) u\geq C_{g}\left\vert u\right\vert
^{r_{2}}-C_{2},\text{ }\forall s\in \mathbb{R},
\end{equation*}%
in place of (\ref{eq:a-priori-estimate-4bis}), owing to (\ref{disf1}). This
implies the same apriori estimate (\ref{eq:a-priori-estimate}).
\end{proof}

Finally, we also have uniqueness of the weak solution in some cases.

\begin{proposition}
\label{uniqueness}Let $\left( U_{i},\Phi _{i}\right) $ be any two weak
solutions of problem \textbf{P }in the sense of Definition \ref{weak}, for $%
i=1,2.$ Assume (\ref{eq:assumption-2}). Then the following estimate holds:%
\begin{equation}
\left\Vert U_{1}(t)-U_{2}\left( t\right) \right\Vert _{\mathbb{X}%
^{2}}+\left\Vert \Phi _{1}^{t}-\Phi _{2}^{t}\right\Vert _{\mathcal{M}%
_{\Omega ,\Gamma }^{1}}\leq \left( \left\Vert U_{1}(0)-U_{2}\left( 0\right)
\right\Vert _{\mathbb{X}^{2}}+\left\Vert \Phi _{1}^{0}-\Phi
_{2}^{0}\right\Vert _{\mathcal{M}_{\Omega ,\Gamma }^{1}}\right) e^{Ct},
\label{cont-dep}
\end{equation}%
for some constant $C>0$ independent of time, $U_{i}$ and $\Phi _{i}.$
\end{proposition}

\begin{proof}
Set ${\widetilde{U}}=U_{1}-U_{2}$, ${\widetilde{\Phi }}=\Phi _{1}-\Phi _{2}$%
. The function $({\widetilde{U}},{\widetilde{\Phi }})$ satisfies the
equations:%
\begin{align}
& \left\langle \partial _{t}\widetilde{U}(t),V\right\rangle _{\mathbb{X}%
^{2}}+\left\langle \mathrm{A_{W}^{0,\beta ,\nu ,\omega }}\widetilde{U}%
(t),V\right\rangle _{\mathbb{X}^{2}}+\left\langle
F(U_{1})-F(U_{2}),V\right\rangle _{\mathbb{X}^{2}}  \label{eq:problem-p-7} \\
& +\int_{0}^{\infty }\mu _{\Omega }\left( s\right) \left\langle \mathrm{%
A_{W}^{\alpha ,0,0,\omega }}\widetilde{\Phi }^{t}\left( s\right)
,V\right\rangle _{\mathbb{X}^{2}}ds+\nu \int_{0}^{\infty }\mu _{\Gamma
}(s)\left\langle \mathrm{C}\widetilde{\xi }^{t}\left( s\right)
,v\right\rangle _{L^{2}\left( \Gamma \right) }ds  \notag \\
& =0  \notag
\end{align}%
and%
\begin{equation}
\left\langle \partial _{t}\widetilde{\Phi }^{t}\left( s\right) -\mathrm{T}%
\widetilde{\Phi }^{t}(s)-\widetilde{U}\left( t\right) ,\Pi \right\rangle _{%
\mathcal{M}_{\Omega ,\Gamma }^{1}}=0,  \label{eq:problem-p-8}
\end{equation}%
for all $\left( V,\Pi \right) \in \left( \mathbb{V}^{1}\oplus \left(
L^{r_{1}}(\Omega )\times L^{r_{2}}(\Gamma )\right) \right) \times \mathcal{M}%
_{\Omega ,\Gamma }^{1}$, subject to the associated initial conditions%
\begin{equation*}
\widetilde{U}(0)=U_{1}\left( 0\right) -U_{2}\left( 0\right) ~\text{and}~%
\widetilde{\Phi }^{0}=\Phi _{1}^{0}-\Phi _{2}^{0}.
\end{equation*}%
Multiplication of (\ref{eq:problem-p-7}) by $V=\widetilde{U}(t)$ in $\mathbb{%
X}^{2}$ and multiplication of (\ref{eq:problem-p-8}) by $\Pi =\widetilde{%
\Phi }^{t}$ in $\mathcal{M}_{\Omega ,\Gamma }^{1}$, followed by summing the
resulting identities, leads us to the differential inequality
\begin{align}
& \frac{d}{dt}\left( \left\Vert U_{1}-U_{2}\right\Vert _{\mathbb{X}%
^{2}}^{2}+\left\Vert \Phi _{1}-\Phi _{2}\right\Vert _{\mathcal{M}_{\Omega
,\Gamma }^{1}}^{2}\right)  \label{eq:part-uniqueness-1} \\
& \leq -2\left\langle F(U_{1})-F(U_{2}),\widetilde{U}\right\rangle _{\mathbb{%
X}^{2}}  \notag \\
& =-2\left\langle f(u_{1})-f(u_{2}),u_{1}-u_{2}\right\rangle _{L^{2}\left(
\Omega \right) }-2\left\langle \widetilde{g}(u_{1})-\widetilde{g}%
(u_{2}),u_{1}-u_{2}\right\rangle _{L^{2}\left( \Gamma \right) }.  \notag
\end{align}%
Employing assumption (\ref{eq:assumption-2}) on the nonlinear terms, we
easily find that%
\begin{equation}
\frac{d}{dt}\left( \left\Vert U_{1}-U_{2}\right\Vert _{\mathbb{X}%
^{2}}^{2}+\left\Vert \Phi _{1}-\Phi _{2}\right\Vert _{\mathcal{M}_{\Omega
,\Gamma }^{1}}^{2}\right) \leq C\left\Vert U_{1}-U_{2}\right\Vert _{\mathbb{X%
}^{2}}^{2},  \label{eq:part-uniqueness-2}
\end{equation}%
for some $C=C\left( M_{f},M_{g},\beta \right) >0$. Application of the
standard Gronwall lemma to (\ref{eq:part-uniqueness-2}) yields the desired
claim (\ref{cont-dep}).
\end{proof}

In the final part of this section, we turn our attention to the existence of
global strong solutions for problem \textbf{P}. First, assuming that the
interior and boundary share the same memory kernel, we can derive the
existence of strong solutions in the case when the bulk and boundary
nonlinearities have supercritical polynomial growth of order at most $7/2$.
Let $\overline{f}$, $\overline{g}$ denote the primitives of $f$ and $%
\widetilde{g}$, respectively, such that $\overline{f}\left( 0\right) =%
\overline{g}\left( 0\right) =0.$

\begin{theorem}
\label{quasi-strong}Let (\ref{miu1})-(\ref{miu3}) be satisfied for $\mu
_{\Omega }\equiv \mu _{\Gamma }$, and assume that $f,$ $g\in C^{1}\left(
\mathbb{R}\right) $ satisfy the following assumptions:

(i) $|f^{^{\prime }}\left( s\right) |$ $\leq \ell _{1}\left( 1+\left\vert
s\right\vert ^{r_{1}}\right) ,$ for all $s\in \mathbb{R}$, for some
(arbitrary) $1\leq r_{1}<\frac{5}{2}.$

(ii) $|g^{^{\prime }}\left( s\right) |$ $\leq \ell _{2}(1+|s|^{r_{2}}),$ for
all $s\in \mathbb{R}$, for some (arbitrary) $1\leq r_{2}<\frac{5}{2}.$

(iii) (\ref{eq:assumption-2}) holds and there exist constants $C_{i}>0,$ $%
i=1,...,4,$ such that%
\begin{equation}
f\left( s\right) s\geq -C_{1}\left\vert s\right\vert ^{2}-C_{2},\text{ }%
g\left( s\right) s\geq -C_{3}\left\vert s\right\vert ^{2}-C_{4},\text{ }%
\forall s\in \mathbb{R}\text{.}  \label{dissip-a1}
\end{equation}

\noindent Given $\alpha ,\beta >0$, $\omega ,\nu \in (0,1)$, $\left(
U_{0},\Phi _{0}\right) \in \mathcal{H}_{\Omega ,\Gamma }^{1,2}$, there
exists a unique global strong solution $\left( U,\Phi \right) $ to problem
\textbf{P} in the sense of Definition \ref{weak2}.
\end{theorem}

\begin{proof}
{\underline{Step 1}} (The existence argument). By Remark \ref{qw} it
suffices to deduce additional regularity for $\left( U,\Phi \right) $. In
order to get the crucial estimate we rely once again on various dissipative
estimates. First, we notice that using the condition of (\ref{dissip-a1}),
we obtain%
\begin{equation*}
\left\langle F\left( U_{n}\right) ,U_{n}\right\rangle _{\mathbb{X}^{2}}\geq
-C_{F}\left( \left\Vert U_{n}\right\Vert _{\mathbb{X}^{2}}^{2}+1\right) ,
\end{equation*}%
for some $C_{F}>0$. Thus, arguing in the same fashion as in getting (\ref%
{eq:a-priori-estimate-1tris}), in view of Gronwall's lemma we obtain%
\begin{align}
& \Vert U_{n}(t)\Vert _{\mathbb{X}^{2}}^{2}+\left\Vert \Phi
_{n}^{t}\right\Vert _{\mathcal{M}_{\Omega ,\Gamma }^{1}}^{2}-2\left\langle
\mathrm{T}\Phi _{n}^{t},\Phi _{n}^{t}\right\rangle _{\mathcal{M}_{\Omega
,\Gamma }^{1}}+C\int_{0}^{t}\Vert U_{n}(\tau )\Vert _{\mathbb{V}%
^{1}}^{2}d\tau  \label{qest1} \\
& \leq C_{T}\left( 1+\Vert U(0)\Vert _{\mathbb{X}^{2}}^{2}+\left\Vert \Phi
^{0}\right\Vert _{\mathcal{M}_{\Omega ,\Gamma }^{1}}^{2}\right) ,  \notag
\end{align}%
where $C_{T}\sim e^{CT},$ for some $C>0$ which is independent of $T,$ $n,$ $%
t.$

Next, we derive an estimate for $U_{n}\in L^{\infty }(0,T;\mathbb{V}^{1})$
and $\Phi _{n}\in L^{\infty }(0,T;\mathcal{M}_{\Omega ,\Gamma }^{2})$. We
use again the scheme (\ref{eq:approx-ode-1})-(\ref{eq:approx-ode-3}) in
which we test equation (\ref{eq:approx-ode-1}) with the function%
\begin{equation*}
\overline{U}=Z_{n}:=\binom{z_{n}}{z_{n\mid \Gamma }},\text{ }%
z_{n}:=\sum_{i=1}^{n}a_{i}(t)\lambda _{i}\theta _{i}^{\alpha ,\beta ,\nu
,\omega }\in C^{2}\left( \left( 0,T\right) \times \overline{\Omega }\right) .
\end{equation*}%
We get%
\begin{equation}
\left\langle \partial _{t}U_{n},{Z}_{n}\right\rangle _{\mathbb{X}%
^{2}}+\left\langle \mathrm{A_{W}^{0,\beta ,\nu ,\omega }}U_{n},{Z}%
_{n}\right\rangle _{\mathbb{X}^{2}}+\left\langle \Phi _{n}^{t}\left(
s\right) ,Z_{n}\right\rangle _{\mathcal{M}_{\Omega ,\Gamma
}^{1}}+\left\langle F(U_{n}),{Z}_{n}\right\rangle _{\mathbb{X}^{2}}=0.
\label{qest3}
\end{equation}%
Moreover, testing (\ref{eq:approx-ode-2}) with%
\begin{equation*}
\overline{\Phi }=\Xi _{n}^{t}:=\binom{\varphi _{n}^{t}}{\varphi _{n\mid
\Gamma }^{t}},\text{ }\varphi _{n}^{t}:=\sum_{i=1}^{n}b_{i}(t)\varkappa
_{i}\left( s\right) \lambda _{i}\theta _{i}^{\alpha ,\beta ,\nu ,\omega
}=\sum_{i=1}^{n}b_{i}(t)\lambda _{i}\zeta _{i}\left( s\right)
\end{equation*}%
we find
\begin{equation}
\left\langle \partial _{t}\Phi _{n}^{t},\Xi _{n}^{t}\right\rangle _{\mathcal{%
M}_{\Omega ,\Gamma }^{1}}=\left\langle \mathrm{T}\Phi _{n}^{t},\Xi
_{n}^{t}\right\rangle _{\mathcal{M}_{\Omega ,\Gamma }^{1}}+\left\langle
U_{n},\Xi _{n}^{t}\right\rangle _{\mathcal{M}_{\Omega ,\Gamma }^{1}}.
\label{qest4}
\end{equation}%
Indeed, $\left( Z_{n},\Xi _{n}^{t}\right) \in X_{n}\times M_{n}$ is
admissible as a test function in (\ref{eq:approx-ode-1})-(\ref%
{eq:approx-ode-2}). Recalling (\ref{eq:approximate-1}), we further notice
that $Z_{n}=\mathrm{A_{W}^{\alpha ,\beta ,\nu ,\omega }}U_{n}$ and $\Xi
_{n}^{t}=\mathrm{A_{W}^{\alpha ,\beta ,\nu ,\omega }}\Phi _{n}^{t},$
respectively, due to the fact that the eigenpair $(\lambda _{i},\theta
_{i}^{\alpha ,\beta ,\nu ,\omega })$ solves (\ref{eigen-p}). Owing to these
identities and (\ref{identity*}), we have%
\begin{align}
\left\langle \Phi _{n}^{t}\left( s\right) ,Z_{n}\right\rangle _{\mathcal{M}%
_{\Omega ,\Gamma }^{1}}& =\int_{0}^{\infty }\mu _{\Omega }(s)\left\langle
\mathrm{A_{W}^{\alpha ,0,0,\omega }}\Phi _{n}^{t}\left( s\right)
,Z_{n}\right\rangle _{\mathbb{X}^{2}}ds+\nu \int_{0}^{\infty }\mu _{\Gamma
}(s)\left\langle \mathrm{C}\xi _{n}^{t}\left( s\right) ,z_{n}\right\rangle
_{L^{2}\left( \Gamma \right) }ds  \label{qest2} \\
& \overset{\mu _{\Omega }\equiv \mu _{\Gamma }}{=}\int_{0}^{\infty }\mu
_{\Omega }(s)\left\langle \mathrm{A_{W}^{\alpha ,\beta ,\nu ,\omega }}\Phi
_{n}^{t}\left( s\right) ,\mathrm{A_{W}^{\alpha ,\beta ,\nu ,\omega }}%
U_{n}\right\rangle _{\mathbb{X}^{2}}ds  \notag \\
& =\left\langle U_{n},\Xi _{n}^{t}\right\rangle _{\mathcal{M}_{\Omega
,\Gamma }^{1}}.  \notag
\end{align}%
Adding relations (\ref{qest3})-(\ref{qest4}) together, and using (\ref{qest2}%
) we further deduce%
\begin{align}
& \frac{1}{2}\frac{d}{dt}\left( \left\Vert U_{n}\right\Vert _{\mathbb{V}%
^{1}}^{2}+\left\Vert \Xi _{n}^{t}\right\Vert _{L_{\mu _{\Omega }}^{2}(%
\mathbb{R}_{+};\mathbb{X}^{2})}^{2}\right) -\left\langle \mathrm{T}\Phi
_{n}^{t},\Xi _{n}^{t}\right\rangle _{\mathcal{M}_{\Omega ,\Gamma
}^{1}}+\left\Vert Z_{n}\right\Vert _{\mathbb{X}^{2}}^{2}  \label{qest5} \\
& =\alpha \omega \left\langle u_{n},z_{n}\right\rangle _{L^{2}\left( \Omega
\right) }-\left\langle F(U_{n}),{Z}_{n}\right\rangle _{\mathbb{X}^{2}},
\notag
\end{align}%
and%
\begin{equation}
\left\langle \mathrm{T}\Phi _{n}^{t},\Xi _{n}^{t}\right\rangle _{\mathcal{M}%
_{\Omega ,\Gamma }^{1}}=\int_{0}^{\infty }\mu _{\Omega }(s)\left\langle
\mathrm{A_{W}^{\alpha ,\beta ,\nu ,\omega }T}\Phi _{n}^{t},\Xi
_{n}^{t}\right\rangle _{\mathbb{X}^{2}}ds=\frac{1}{2}\int_{0}^{\infty }\mu
_{\Omega }^{^{\prime }}(s)\left\Vert \Xi _{n}^{t}\left( s\right) \right\Vert
_{\mathbb{X}^{2}}^{2}ds,  \label{qest2bis}
\end{equation}%
thanks to the fact that $\mu _{\Omega }\equiv \mu _{\Gamma }$. We begin
estimating both terms on the right-hand side of (\ref{qest5}). The first one
is easy,%
\begin{equation}
\alpha \omega \left\langle u_{n},z_{n}\right\rangle _{L^{2}\left( \Omega
\right) }\leq \delta \left\Vert z_{n}\right\Vert _{L^{2}\left( \Omega
\right) }^{2}+C_{\delta }\left\Vert u_{n}\right\Vert _{L^{2}\left( \Omega
\right) }^{2},  \label{qest2tris}
\end{equation}%
for any $\delta \in (0,1]$. To bound the last term we integrate by parts in
the following way:%
\begin{align}
\left\langle F(U_{n}),{Z}_{n}\right\rangle _{\mathbb{X}^{2}}& =\int_{\Omega
}f\left( u_{n}\right) \left( -\omega \Delta u_{n}+\alpha \omega u_{n}\right)
dx+\int_{\Gamma }\widetilde{g}\left( u_{n}\right) \left( -\nu \Delta
_{\Gamma }u_{n}+\omega \partial _{n}u_{n}+\nu \beta u_{n}\right) d\sigma
\label{qest5bis} \\
& =\omega \int_{\Omega }f^{^{\prime }}\left( u_{n}\right) \left\vert \nabla
u_{n}\right\vert ^{2}dx+\nu \int_{\Gamma }\widetilde{g}^{^{\prime }}\left(
u_{n}\right) \left\vert \nabla _{\Gamma }u_{n}\right\vert ^{2}d\sigma  \notag
\\
& +\alpha \omega \int_{\Omega }f\left( u_{n}\right) u_{n}dx+\nu \beta
\int_{\Gamma }\widetilde{g}\left( u_{n}\right) u_{n}d\sigma  \notag \\
& +\omega \int_{\Gamma }\left( \widetilde{g}\left( u_{n}\right) -f\left(
u_{n}\right) \right) \partial _{n}u_{n}d\sigma .  \notag
\end{align}%
By assumptions (\ref{eq:assumption-2}) and (\ref{dissip-a1}), we can easily
find a positive constant $C$ independent of $t,T$ and $n$ such that%
\begin{equation}
\omega \int_{\Omega }f^{^{\prime }}\left( u_{n}\right) \left\vert \nabla
u_{n}\right\vert ^{2}dx+\nu \int_{\Gamma }\widetilde{g}^{^{\prime }}\left(
u_{n}\right) \left\vert \nabla _{\Gamma }u_{n}\right\vert ^{2}d\sigma \geq
-M_{f}\omega \left\Vert \nabla u_{n}\right\Vert _{L^{2}\left( \Omega \right)
}^{2}-M_{g}\nu \left\Vert \nabla _{\Gamma }u_{n}\right\Vert _{L^{2}\left(
\Gamma \right) }^{2}  \label{qest5tris}
\end{equation}%
and%
\begin{equation}
\alpha \omega \int_{\Omega }f\left( u_{n}\right) u_{n}dx+\nu \beta
\int_{\Gamma }\widetilde{g}\left( u_{n}\right) u_{n}d\sigma \geq -C\left(
\left\Vert U_{n}\right\Vert _{\mathbb{X}^{2}}^{2}+1\right) .
\label{qest5quad}
\end{equation}%
In order to estimate the last boundary integral on the right-hand side of (%
\ref{qest5bis}), we observe that due to assumptions (i)-(ii) it suffices to
estimate boundary integrals of the form%
\begin{equation*}
I:=\int_{\Gamma }u_{n}^{r+1}\partial _{n}u_{n}d\sigma ,\text{ for some }%
r<5/2.
\end{equation*}%
Indeed, due to classical trace regularity and embedding results, for every $%
\delta \in (0,1]$ we have%
\begin{equation}
I\leq \left\Vert \partial _{n}u_{n}\right\Vert _{H^{1/2}\left( \Gamma
\right) }\left\Vert u_{n}^{r+1}\right\Vert _{H^{-1/2}\left( \Gamma \right)
}\leq \delta \left\Vert u_{n}\right\Vert _{H^{2}\left( \Omega \right)
}^{2}+C_{\delta }\left\Vert u_{n}^{r+1}\right\Vert _{H^{-1/2}\left( \Gamma
\right) }^{2}.  \label{qest5q}
\end{equation}%
It remains to estimate the last term in (\ref{qest5q}). To this end, we
employ the basic Sobolev embeddings $H^{1/2}\left( \Gamma \right) \subset
L^{4}\left( \Gamma \right) $ and $H^{1}\left( \Gamma \right) \subset
L^{s}\left( \Gamma \right) ,$ for any $s\in (\frac{4}{3},\infty )$,
respectively. Owing to elementary Holder inequalities, we deduce that%
\begin{align}
\left\Vert u_{n}^{r+1}\right\Vert _{H^{-1/2}\left( \Gamma \right) }^{2}&
=\sup_{\psi \in H^{1/2}\left( \Gamma \right) :\left\Vert \psi \right\Vert
_{H^{1/2}\left( \Gamma \right) }=1}\left\vert \left\langle u_{n}^{r+1},\psi
\right\rangle \right\vert ^{2}  \label{qest5qqq} \\
& \leq \left\Vert u_{n}\right\Vert _{L^{s}\left( \Gamma \right)
}^{2}\left\Vert u_{n}\right\Vert _{L^{\overline{s}r}\left( \Gamma \right)
}^{2r}  \notag \\
& \leq C\left\Vert u_{n}\right\Vert _{H^{1}\left( \Gamma \right)
}^{2}\left\Vert u_{n}\right\Vert _{L^{\overline{s}r}\left( \Gamma \right)
}^{2r},  \notag
\end{align}%
for some positive constant $C$ independent of $u,n,t,T$, for sufficiently
large $s\in (\frac{4}{3},\infty )$, where $\overline{s}:=4s/\left(
3s-4\right) >4/3$. Exploiting now the interpolation inequality%
\begin{equation*}
\left\Vert u\right\Vert _{L^{\overline{s}r}\left( \Gamma \right) }\leq
C\left\Vert u\right\Vert _{H^{2}\left( \Gamma \right) }^{1/\left( 2r\right)
}\left\Vert u\right\Vert _{L^{2}\left( \Gamma \right) }^{1-1/\left(
2r\right) },
\end{equation*}%
provided that $r=1+2/\overline{s}<5/2$, we further infer from (\ref{qest5qqq}%
) that%
\begin{align}
\left\Vert u_{n}^{r+1}\right\Vert _{H^{-1/2}\left( \Gamma \right) }^{2}&
\leq C\left\Vert u_{n}\right\Vert _{H^{1}\left( \Gamma \right)
}^{2}\left\Vert u_{n}\right\Vert _{H^{2}\left( \Gamma \right) }\left\Vert
u_{n}\right\Vert _{L^{2}\left( \Gamma \right) }^{2r-1}  \label{qest5qq} \\
& \leq \eta \left\Vert u_{n}\right\Vert _{H^{2}\left( \Gamma \right)
}^{2}+C_{\eta }\left\Vert u_{n}\right\Vert _{H^{1}\left( \Gamma \right)
}^{2}\left( \left\Vert u_{n}\right\Vert _{H^{1}\left( \Gamma \right)
}^{2}\left\Vert u_{n}\right\Vert _{L^{2}\left( \Gamma \right) }^{2\left(
2r-1\right) }\right) ,  \notag
\end{align}%
for any $\eta \in (0,1]$. Inserting (\ref{qest5qq}) into (\ref{qest5q}) and
choosing a sufficiently small $\eta =\delta /C_{\delta }$, by virtue of (\ref%
{regularity-oper}), we easily deduce%
\begin{equation}
I\leq \delta \left\Vert Z_{n}\right\Vert _{\mathbb{X}^{2}}^{2}+C_{\delta
}\left\Vert u_{n}\right\Vert _{H^{1}\left( \Gamma \right) }^{2}\left(
\left\Vert u_{n}\right\Vert _{H^{1}\left( \Gamma \right) }^{2}\left\Vert
u_{n}\right\Vert _{L^{2}\left( \Gamma \right) }^{2\left( 2r-1\right)
}\right) .  \label{qest5last}
\end{equation}%
Thus, setting
\begin{align*}
\Xi \left( t\right) & :=\left\Vert U_{n}\left( t\right) \right\Vert _{%
\mathbb{V}^{1}}^{2}+\left\Vert \Xi _{n}^{t}\right\Vert _{L_{\mu _{\Omega
}}^{2}(\mathbb{R}_{+};\mathbb{X}^{2})}^{2},\text{ } \\
\Lambda \left( t\right) & :=C_{\delta }\left( 1+\left\Vert u_{n}\right\Vert
_{H^{1}\left( \Gamma \right) }^{2}\left\Vert u_{n}\right\Vert _{L^{2}\left(
\Gamma \right) }^{2\left( 2r-1\right) }\right) ,
\end{align*}%
it follows from (\ref{qest5}), (\ref{qest2tris})-(\ref{qest5quad}) and (\ref%
{qest5last}) that%
\begin{equation}
\frac{d}{dt}\Xi \left( t\right) -2\left\langle \mathrm{T}\Phi _{n}^{t},\Xi
_{n}^{t}\right\rangle _{\mathcal{M}_{\Omega ,\Gamma }^{1}}+\left( 2-\delta
\right) \left\Vert Z_{n}\right\Vert _{\mathbb{X}^{2}}^{2}\leq \Xi \left(
t\right) \Lambda \left( t\right) ,  \label{qest3quad}
\end{equation}%
for a sufficiently small $\delta \in (0,1]$. Gronwall's inequality together
with (\ref{qest1}) yields%
\begin{align}
& \left\Vert U_{n}\left( t\right) \right\Vert _{\mathbb{V}%
^{1}}^{2}+\left\Vert \Xi _{n}^{t}\right\Vert _{L_{\mu _{\Omega }}^{2}(%
\mathbb{R}_{+};\mathbb{X}^{2})}^{2}+\int_{0}^{t}\left( \Vert Z_{n}(\tau
)\Vert _{\mathbb{X}^{2}}^{2}-2\left\langle \mathrm{T}\Phi _{n}^{\tau },\Xi
_{n}^{\tau }\right\rangle _{\mathcal{M}_{\Omega ,\Gamma }^{1}}\right) d\tau
\label{qest6} \\
& \leq C_{T}\left( \left\Vert U\left( 0\right) \right\Vert _{\mathbb{V}%
^{1}}^{2}+\left\Vert \Xi ^{0}\right\Vert _{L_{\mu _{\Omega }}^{2}(\mathbb{R}%
_{+};\mathbb{X}^{2})}^{2}\right) ,  \notag
\end{align}%
owing to the boundedness of the (orthogonal) projectors $P_{n}:\mathbb{X}%
^{2}\rightarrow X_{n}$ and $Q_{n}:\mathcal{M}_{\Omega ,\Gamma
}^{1}\rightarrow M_{n}$, and the fact that $\Lambda \in L^{1}\left(
0,T\right) ,$ for any $T>0.$

From (\ref{qest6}), recalling (\ref{regularity-oper}) we obtain the
following uniform (in $n$) bounds for each approximate solution $(U_{n},\Phi
_{n})$:%
\begin{align}
U_{n}& ~\text{is uniformly bounded in}~L^{\infty }\left( 0,T;\mathbb{V}%
^{1}\right) ,  \label{qest7-1} \\
U_{n}& ~\text{is uniformly bounded in}~L^{2}\left( 0,T;\mathbb{V}^{2}\right)
,  \label{qest7-2} \\
\Phi _{n}& ~\text{is uniformly bounded in}~L^{\infty }\left( 0,T;\mathcal{M}%
_{\Omega ,\Gamma }^{2}\right) ,  \label{qest7-3} \\
\Phi _{n}& ~\text{is uniformly bounded in}~L^{2}\left( 0,T;L_{k_{\Omega
}}^{2}\left( \mathbb{R}_{+};\mathbb{V}^{2}\right) \right) .  \label{qest7-4}
\end{align}%
Observe now that by (\ref{eq:approx-ode-1})-(\ref{eq:approx-ode-2}), we also
have%
\begin{align}
\left\langle \partial _{t}U_{n},{\overline{U}}\right\rangle _{\mathbb{X}%
^{2}}& =\left\langle \partial _{t}U_{n},P_{n}{\overline{U}}\right\rangle _{%
\mathbb{X}^{2}}  \label{qest8} \\
& =-\left\langle \mathrm{A_{W}^{0,\beta ,\nu ,\omega }}U_{n},P_{n}{\overline{%
U}}\right\rangle _{\mathbb{X}^{2}}-\left\langle \Phi _{n}^{t},P_{n}{%
\overline{U}}\right\rangle _{\mathcal{M}_{\Omega ,\Gamma }^{1}}-\left\langle
F(U_{n}),P_{n}{\overline{U}}\right\rangle _{\mathbb{X}^{2}}  \notag
\end{align}%
and%
\begin{align}
\left\langle \partial _{t}\Phi _{n}^{t},{\overline{\Phi }}\right\rangle _{%
\mathcal{M}_{\Omega ,\Gamma }^{1}}& =\left\langle \partial _{t}\Phi
_{n}^{t},Q_{n}{\overline{\Phi }}\right\rangle _{\mathcal{M}_{\Omega ,\Gamma
}^{1}}  \label{qest9} \\
& =\left\langle \mathrm{T}\Phi _{n}^{t},Q_{n}{\overline{\Phi }}\right\rangle
_{\mathcal{M}_{\Omega ,\Gamma }^{1}}+\left\langle U_{n},Q_{n}{\overline{\Phi
}}\right\rangle _{\mathcal{M}_{\Omega ,\Gamma }^{1}},  \notag
\end{align}%
respectively. Thus, from the uniform bounds (\ref{qest7-1})-(\ref{qest7-4}),
we deduce by comparison in equations (\ref{qest8})-(\ref{qest9}) that%
\begin{align}
& \partial _{t}U_{n}\text{ is uniformly bounded in }L^{\infty }\left(
0,T;\left( \mathbb{V}^{1}\right) ^{\ast }\right) \cap L^{2}\left( 0,T;%
\mathbb{X}^{2}\right) ,  \label{qest10-1} \\
& \partial _{t}\Phi _{n}^{t}\text{ is uniformly bounded in }L^{2}\left(
0,T;L_{\mu _{\Omega }}^{2}\left( \mathbb{R}_{+};\mathbb{X}^{2}\right)
\right) \cap L^{\infty }\left( 0,T;L_{\mu _{\Omega }}^{2}\left( \mathbb{R}%
_{+};\left( \mathbb{V}^{1}\right) ^{\ast }\right) \right) .  \label{qest10-2}
\end{align}%
We are now ready to pass to the limit as $n$ goes to infinity. On account of
the above uniform inequalities, we can find $U$ and $\Phi $ such that, up to
subsequences,%
\begin{align}
U_{n}& \rightarrow U~\text{weakly * in }L^{\infty }\left( 0,T;\mathbb{V}%
^{1}\right) ,  \label{qest11-1} \\
U_{n}& \rightarrow U\text{ weakly in }L^{2}\left( 0,T;\mathbb{V}^{2}\right) ,
\label{qest11-2} \\
\Phi _{n}& \rightarrow \Phi \text{ weakly * in }L^{\infty }\left( 0,T;%
\mathcal{M}_{\Omega ,\Gamma }^{2}\right) ,  \label{qest11-3} \\
\Phi _{n}& \rightarrow \Phi \text{ weakly in }L^{2}\left( 0,T;L_{k_{\Omega
}}^{2}\left( \mathbb{R}_{+};\mathbb{V}^{2}\right) \right) , \\
\partial _{t}U_{n}& \rightarrow \partial _{t}U\text{ in }L_{w^{\ast
}}^{\infty }\left( 0,T;\left( \mathbb{V}^{1}\right) ^{\ast }\right) \cap
L_{w}^{2}\left( 0,T;\mathbb{X}^{2}\right) ,  \label{qest11-5} \\
\partial _{t}\Phi _{n}^{t}& \rightarrow \partial _{t}\Phi ^{t}\text{ in }%
L_{w}^{2}\left( 0,T;L_{\mu _{\Omega }}^{2}\left( \mathbb{R}_{+};\mathbb{X}%
^{2}\right) \right) .  \label{qest11-6}
\end{align}%
Due to (\ref{qest11-1}) and (\ref{qest11-5}) and the classical Agmon-Lions
compactness theorem, we also have%
\begin{equation}
U_{n}\rightarrow U~\text{strongly in }C(\left[ 0,T\right] ;\mathbb{X}^{2}).
\label{qest12}
\end{equation}%
Thanks to (\ref{qest11-1})-(\ref{qest11-6}) and (\ref{qest12}), we can
easily control the nonlinear terms in (\ref{eq:approx-ode-1})-(\ref%
{eq:approx-ode-2}). By means of the above convergence properties, we can
pass to the limit in these equations and show that $\left( U,\Phi \right) $
solves (\ref{eq:problem-p-2}) in the sense of Definition \ref{weak2}.

Finally, uniqueness follows from Proposition \ref{uniqueness} owing to
assumption (\ref{eq:assumption-2}). The proof of the theorem is finished.
\end{proof}

\begin{remark}
\label{no-cancel}Observe that the assumption $\mu _{\Omega }\equiv \mu
_{\Gamma }$ in Theorem \ref{quasi-strong} is crucial for the identity (\ref%
{qest2}) to hold. Without it, cancellation in (\ref{qest5}) does not
generally occur and (\ref{qest2bis}) does not hold.
\end{remark}

We now let%
\begin{equation*}
h_{f}\left( s\right) =\int_{0}^{s}f^{^{\prime }}\left( \tau \right) \tau
d\tau \text{ and }h_{g}\left( s\right) =\int_{0}^{s}\widetilde{g}^{^{\prime
}}\left( \tau \right) \tau d\tau .
\end{equation*}%
The next result states that there exist strong solutions, albeit in a much
weaker sense than in Theorem \ref{quasi-strong}, even when the interior and
boundary memory kernels $\mu _{S}\left( \cdot \right) :\mathbb{R}%
_{+}\rightarrow \mathbb{R}_{+}$\ do \emph{not} coincide but both decay
exponentially fast as $s$ goes to infinity.

\begin{theorem}
\label{t:strong-solutions}Let (\ref{miu1})-(\ref{miu3}) be satisfied and
assume that $f,$ $g\in C^{1}\left( \mathbb{R}\right) $ satisfy the following
conditions:

(i) $|f^{^{\prime }}\left( s\right) |$ $\leq \ell _{1}\left( 1+\left\vert
s\right\vert ^{2}\right) ,$ for all $s\in \mathbb{R}$.

(ii) $|g^{^{\prime }}\left( s\right) |$ $\leq \ell _{2}(1+|s|^{r_{2}}),$ for
all $s\in \mathbb{R}$, for some (arbitrary) $r_{2}>2.$

(iii) (\ref{eq:assumption-2}) holds and there exist $C_{i}>0,$ $i=1,\dots
,8, $ such that%
\begin{equation}
\left\{
\begin{array}{ll}
f\left( s\right) s\geq -C_{1}\left\vert s\right\vert ^{2}-C_{2},\text{ }%
g\left( s\right) s\geq -C_{3}\left\vert s\right\vert ^{2}-C_{4}, & \forall
s\in \mathbb{R} \\
h_{f}\left( s\right) \geq -C_{5}\left\vert s\right\vert ^{2}-C_{6},\text{ }%
h_{g}\left( s\right) \geq -C_{7}\left\vert s\right\vert ^{2}-C_{8}, &
\forall s\in \mathbb{R}\text{.}%
\end{array}%
\right.  \label{dissip-3}
\end{equation}%
In addition, assume there exist constants $\delta _{S}>0$ such that%
\begin{equation}
\mu _{S}^{^{\prime }}\left( s\right) +\delta _{S}\mu _{S}\left( s\right)
\leq 0\text{, for all }s\in \mathbb{R}_{+}\text{, }S\in \left\{ \Omega
,\Gamma \right\} .  \label{fading}
\end{equation}%
\noindent Given $\alpha ,\beta >0$, $\omega ,\nu \in (0,1)$, $\left(
U_{0},\Phi _{0}\right) \in \mathbb{V}^{2}\times \left( \mathcal{M}_{\Omega
,\Gamma }^{2}\cap D\left( \mathrm{T}\right) \right) $, there exists a unique
global quasi-strong solution $\left( U,\Phi \right) $\ to problem \textbf{P}
in the sense of Definition \ref{d:strong-solution}.
\end{theorem}

\begin{proof}
It suffices to provide bounds for $(U,\Phi ^{t})$ in the (more regular)
spaces in (\ref{eq:strong-defn-1})-(\ref{eq:strong-defn-4}). With reference
to problem \textbf{P}$_{n},$ we consider the approximate problem of finding $%
\left( U_{n},\Phi _{n}\right) $ of the form (\ref{eq:approximate-1}) such
that, $\left( U_{n},\Phi _{n}\right) $ already satisfies (\ref%
{eq:approx-ode-1})-(\ref{eq:approx-ode-2}), and%
\begin{align}
& \left\langle \partial _{tt}U_{n},{\overline{U}}\right\rangle _{\mathbb{X}%
^{2}}+\left\langle \mathrm{A_{W}^{0,\beta ,\nu ,\omega }}\partial _{t}U_{n},{%
\overline{U}}\right\rangle _{\mathbb{X}^{2}}+\left\langle \partial _{t}\Phi
_{n}^{t},{\overline{U}}\right\rangle _{\mathcal{M}_{\Omega ,\Gamma }^{1}}
\label{ap1} \\
& =-\left\langle f^{^{\prime }}\left( u_{n}\right) \partial _{t}u_{n},\bar{u}%
\right\rangle _{L^{2}\left( \Omega \right) }-\left\langle \widetilde{g}%
^{^{\prime }}\left( u_{n}\right) \partial _{t}u_{n},\bar{v}\right\rangle
_{L^{2}\left( \Gamma \right) }  \notag
\end{align}%
and
\begin{equation}
\left\langle \partial _{tt}\Phi _{n}^{t},{\overline{\Phi }}\right\rangle _{%
\mathcal{M}_{\Omega ,\Gamma }^{1}}=\left\langle \mathrm{T}\partial _{t}\Phi
_{n}^{t},{\overline{\Phi }}\right\rangle _{\mathcal{M}_{\Omega ,\Gamma
}^{1}}+\left\langle \partial _{t}U_{n},{\overline{\Phi }}\right\rangle _{%
\mathcal{M}_{\Omega ,\Gamma }^{1}}  \label{ap2}
\end{equation}%
hold for almost all $t\in \left( 0,T\right) $, for all ${\overline{U}}=(\bar{%
u},\bar{v})^{\mathrm{tr}}\in X_{n}$ and ${\overline{\Phi }}=(\bar{\eta},\bar{%
\xi})^{\mathrm{tr}}\in M_{n}$; moreover, the function $\left( U_{n},\Phi
_{n}\right) $ fulfils the conditions $U_{n}\left( 0\right) =P_{n}U_{0},$ $%
\Phi _{n}^{0}=Q_{n}\Phi ^{0}$ and%
\begin{equation}
\partial _{t}U_{n}\left( 0\right) =P_{n}\widehat{U}_{0},\text{ }\partial
_{t}\Phi _{n}^{0}=Q_{n}\widehat{\Phi }^{0},  \label{ap3}
\end{equation}%
where we have set%
\begin{align*}
\widehat{U}_{0}& :=-\mathrm{A_{W}^{0,\beta ,\nu ,\omega }}%
U_{0}-\int_{0}^{\infty }\mu _{\Omega }(s)\mathrm{A_{W}^{\alpha ,0,0,\omega }}%
\Phi _{0}(s)ds-\nu \int_{0}^{\infty }\mu _{\Gamma }(s)\binom{0}{\mathrm{C}%
\xi _{0}(s)}ds-F(U_{0})\text{, } \\
\widehat{\Phi }^{0}& :=\mathrm{T}\Phi _{0}(s)+U_{0}.
\end{align*}%
Note that, if $U_{0}\in \mathbb{V}^{2}$ and $\Phi ^{0}\in D\left( \mathrm{T}%
\right) \cap \mathcal{M}_{\Omega ,\Gamma }^{2}$, then $(\widehat{U}_{0},%
\widehat{\Phi }^{0})\in \mathbb{X}^{2}\times \mathcal{M}_{\Omega ,\Gamma
}^{1}=\mathcal{H}_{\Omega ,\Gamma }^{0,1}$, owing to the continuous
embeddings $H^{2}\left( \Omega \right) \subset L^{\infty }\left( \Omega
\right) ,$ $H^{2}\left( \Gamma \right) \subset L^{\infty }\left( \Gamma
\right) $. In particular, owing to the boundedness of the projectors $P_{n}$
and $Q_{n}$ on the corresponding subspaces, we have%
\begin{equation}
\left\Vert \left( \partial _{t}U_{n}\left( 0\right) ,\partial _{t}\Phi
_{n}^{0}\right) \right\Vert _{\mathcal{H}_{\Omega ,\Gamma }^{0,1}}\leq
K\left( R\right) ,  \label{apic}
\end{equation}%
for all $\left( U_{0},\Phi ^{0}\right) \in \mathbb{V}^{2}\times \left(
D\left( \mathrm{T}\right) \cap \mathcal{M}_{\Omega ,\Gamma }^{2}\right) $
such that $\left\Vert \left( U_{0},\Phi ^{0}\right) \right\Vert _{\mathcal{H}%
_{\Omega ,\Gamma }^{2,2}}\leq R,$ for some positive monotone nondecreasing
function $K$. Indeed, according to assumptions (\ref{miu1})-(\ref{miu3}), we
can infer that%
\begin{equation}
0\leq \int_{0}^{\infty }\mu _{S}(s)ds=\mu _{S}^{0}<\infty ,\text{ for each }%
S\in \left\{ \Omega ,\Gamma \right\} ,  \label{qest15}
\end{equation}%
such that repeated application of Jensen's inequality yields%
\begin{align*}
\left\Vert \int_{0}^{\infty }\mu _{\Omega }(s)\mathrm{A_{W}^{\alpha
,0,0,\omega }}\Phi _{0}(s)ds\right\Vert _{\mathbb{X}^{2}}^{2}& \leq \mu
_{\Omega }^{0}\int_{0}^{\infty }\mu _{\Omega }(s)\left\Vert \mathrm{%
A_{W}^{\alpha ,0,0,\omega }}\Phi _{0}(s)\right\Vert _{\mathbb{X}^{2}}^{2}ds
\\
& \leq C\mu _{\Omega }^{0}\int_{0}^{\infty }\mu _{\Omega }(s)\left\Vert \Phi
_{0}(s)\right\Vert _{H^{2}}^{2}ds
\end{align*}%
and%
\begin{align*}
\left\Vert \int_{0}^{\infty }\mu _{\Gamma }(s)\mathrm{C}\xi _{0}(s){\mathrm{d%
}}s\right\Vert _{L^{2}\left( \Gamma \right) }^{2}& \leq \mu _{\Gamma
}^{0}\int_{0}^{\infty }\mu _{\Gamma }(s)\left\Vert \mathrm{C}\xi
_{0}(s)\right\Vert _{L^{2}\left( \Gamma \right) }^{2}ds \\
& \leq C\mu _{\Gamma }^{0}\int_{0}^{\infty }\mu _{\Gamma }(s)\left\Vert \Phi
_{0}(s)\right\Vert _{H^{2}\left( \Gamma \right) }^{2}ds.
\end{align*}

Our starting point is the validity of the energy estimate (\ref{qest1})
which holds on account of the first assumption of (\ref{dissip-3}). Next we
proceed to take $\overline{U}=\partial _{t}U_{n}(t)$ in (\ref{ap1}) and $%
\overline{\Phi }=\partial _{t}\Phi _{n}^{t}\left( s\right) $ in (\ref{ap2}),
respectively, by noting that this choice $\left( \overline{U},\overline{\Phi
}\right) $ is an admissible test function. Summing the resulting identities
and using (\ref{eq:assumption-2}), we obtain%
\begin{align}
& \frac{1}{2}\frac{d}{dt}\left\{ \left\Vert \partial _{t}U_{n}\right\Vert _{%
\mathbb{X}^{2}}^{2}+\left\Vert \partial _{t}\Phi _{n}^{t}\right\Vert _{%
\mathcal{M}_{\Omega ,\Gamma }^{1}}^{2}\right\} -\left\langle \mathrm{T}%
\partial _{t}\Phi _{n}^{t},\partial _{t}\Phi _{n}^{t}\right\rangle _{%
\mathcal{M}_{\Omega ,\Gamma }^{1}}  \label{qest16} \\
& +\left( \omega \left\Vert \nabla \partial _{t}u_{n}\right\Vert
_{L^{2}(\Omega )}^{2}+\nu \left\Vert \nabla _{\Gamma }\partial
_{t}u_{n}\right\Vert _{L^{2}(\Gamma )}^{2}+\beta \left\Vert \partial
_{t}u_{n}\right\Vert _{L^{2}(\Gamma )}^{2}\right)  \notag \\
& =-\left\langle f^{^{\prime }}\left( u_{n}\right) \partial
_{t}u_{n},\partial _{t}u_{n}\right\rangle _{L^{2}\left( \Omega \right)
}-\left\langle \widetilde{g}^{^{\prime }}\left( u_{n}\right) \partial
_{t}u_{n},\partial _{t}u_{n}\right\rangle _{L^{2}\left( \Gamma \right) }
\notag \\
& \leq \max \left( M_{f},M_{g}\right) \left\Vert \partial
_{t}U_{n}\right\Vert _{\mathbb{X}^{2}}^{2}.  \notag
\end{align}%
Thus, integrating (\ref{qest16}) with respect to $\tau \in (0,t),$ by
application of Growall's inequality, we have the estimate%
\begin{equation}
\left\Vert \left(\partial _{t}U_{n}\left( t\right) ,\partial _{t}\Phi
_{n}^{t}\right)\right\Vert _{\mathcal{H}_{\Omega ,\Gamma
}^{0,1}}^{2}+\int_{0}^{t} \left( 2\left\Vert \partial _{t}U_{n}(\tau
)\right\Vert _{\mathbb{V}^{1}}^{2}+\left\Vert \partial _{t}\Phi _{n}^{\tau
}\right\Vert _{L_{k_{\Omega }\oplus k_{\Gamma }}^{2}\left( \mathbb{R}_{+};%
\mathbb{V}^{1}\right) }^{2} \right) d\tau \leq K_{T}\left( R\right) ,
\label{qest17}
\end{equation}%
for all $t\geq 0$ and all $R>0$ such that $\left\Vert \left( U_{0},\Phi
^{0}\right) \right\Vert _{\mathcal{H}_{\Omega ,\Gamma }^{2,2}}\leq R$.
Thanks to (\ref{qest17}), we deduce the uniform bounds
\begin{align}
\partial _{t}U_{n}& \in L^{\infty }\left( 0,T;\mathbb{X}^{2}\right) \cap
L^{2}\left( 0,T;\mathbb{V}^{1}\right) ,  \label{stronger-U} \\
\partial _{t}\Phi _{n}& \in L^{\infty }\left( 0,T;\mathcal{M}_{\Omega
,\Gamma }^{1}\right) \cap L^{2}\left(0,T;L_{k_{\Omega }\oplus k_{\Gamma
}}^{2}\left( \mathbb{R}_{+};\mathbb{V}^{1}\right)\right),
\label{stronger-Phi}
\end{align}%
which establishes (\ref{eq:strong-defn-3})-(\ref{eq:strong-defn-4}) for the
approximate solution $\left( U_{n},\Phi _{n}\right) $.

We now establish a bound for $U_{n}$ in $L^{\infty }\left( 0,T;\mathbb{V}%
^{1}\right) $ in a different way from the proof of Theorem \ref{quasi-strong}%
. For this estimate, the uniform regularity in (\ref{stronger-U})-(\ref%
{stronger-Phi}) is crucial. To this end, we proceed to take $\overline{U}%
=U_{n}(t)$ in (\ref{ap1}) in order to derive%
\begin{align}
& \frac{d}{dt}\left( \left\Vert U_{n}\right\Vert _{\mathbb{V}%
^{1}}^{2}+\left\langle \partial _{t}U_{n},U_{n}\right\rangle _{\mathbb{X}%
^{2}}+2\int_{\Omega }h_{f}\left( u_{n}\right) dx+2\int_{\Gamma }h_{g}\left(
u_{n}\right) d\sigma \right)  \label{qest18} \\
& =2\left\Vert \partial _{t}{U}_{n}\right\Vert _{\mathbb{X}%
^{2}}^{2}-2\left\langle \partial _{t}\Phi _{n}^{t},{U}_{n}\right\rangle _{%
\mathcal{M}_{\Omega ,\Gamma }^{1}}.  \notag
\end{align}%
Moreover, using (\ref{stronger-U}) and owing to the Cauchy-Schwarz and Young
inequalities and the second of (\ref{dissip-3}), the following basic
inequality holds:%
\begin{align}
& C_{\ast }\left\Vert U_{n}\right\Vert _{\mathbb{V}^{1}}^{2}-K_{T}\left(
R\right)  \label{qest19tris} \\
& \leq \left\Vert U_{n}\right\Vert _{\mathbb{V}^{1}}^{2}+\left\langle
\partial _{t}U_{n},U_{n}\right\rangle _{\mathbb{X}^{2}}+2\int_{\Omega
}h_{f}\left( u_{n}\right) dx+2\int_{\Gamma }h_{g}\left( u_{n}\right) d\sigma
\notag \\
& \leq C\left\Vert U_{n}\right\Vert _{\mathbb{V}^{1}}^{2}+K_{T}\left(
R\right) ,  \notag
\end{align}%
for some constants $C_{\ast },C>0$ and some function $K_{T}>0,$ all
independent of $n$ and $t$. Finally, for any $\eta >0$ we estimate%
\begin{align}
-\left\langle \partial _{t}\Phi _{n}^{t},{U}_{n}\right\rangle _{\mathcal{M}%
_{\Omega ,\Gamma }^{1}}& \leq \eta \left\Vert U_{n}\right\Vert _{\mathbb{V}%
^{1}}^{2}+C_{\eta }\int_{0}^{\infty }\mu _{\Omega }(s)\left\Vert \partial
_{t}\eta _{n}(s)\right\Vert _{H^{1}}^{2}ds+C_{\eta }\int_{0}^{\infty }\mu
_{\Gamma }(s)\left\Vert \partial _{t}\xi _{n}(s)\right\Vert _{H^{1}\left(
\Gamma \right) }^{2}ds  \label{qest20} \\
& \leq \eta \left\Vert U_{n}\right\Vert _{\mathbb{V}^{1}}^{2}-C_{\eta
}\delta _{\Omega }^{-1}\int_{0}^{\infty }\mu _{\Omega }^{^{\prime
}}(s)\left\Vert \partial _{t}\eta _{n}(s)\right\Vert _{H^{1}}^{2}ds-C_{\eta
}\delta _{\Gamma }^{-1}\int_{0}^{\infty }\mu _{\Gamma }^{^{\prime
}}(s)\left\Vert \partial _{t}\xi _{n}(s)\right\Vert _{H^{1}\left( \Gamma
\right) }^{2}ds,  \notag
\end{align}%
where in the last line we have employed assumption (\ref{fading}). Thus,
from (\ref{qest18}) we obtain the inequality%
\begin{align}
& \frac{d}{dt}\left( \left\Vert U_{n}\right\Vert _{\mathbb{V}%
^{1}}^{2}+\left\langle \partial _{t}U_{n},U_{n}\right\rangle _{\mathbb{X}%
^{2}}+2\int_{\Omega }h_{f}\left( u_{n}\right) dx+2\int_{\Gamma }h_{g}\left(
u_{n}\right) d\sigma \right)  \label{qest21} \\
& \leq C_{\eta }\left\Vert U_{n}\left( t\right) \right\Vert _{\mathbb{V}%
^{1}}^{2}+\Lambda _{2}\left( t\right) ,  \notag
\end{align}%
for a.e. $t\in \left( 0,T\right) ,$ where we have set%
\begin{equation*}
\Lambda _{2}\left( t\right) :=2\left\Vert \partial _{t}{U}_{n}\right\Vert _{%
\mathbb{X}^{2}}^{2}-2\left\langle \partial _{t}\Phi _{n}^{t},{U}%
_{n}\right\rangle _{\mathcal{M}_{\Omega ,\Gamma }^{1}}.
\end{equation*}%
We now observe that $\Lambda _{2}\in L^{1}\left( 0,T\right) $ on account of (%
\ref{qest1}), (\ref{stronger-U})-(\ref{stronger-Phi}) and (\ref{qest19tris}%
)-(\ref{qest20}), because $\partial _{t}U_{n}\left( 0\right) \in \mathbb{X}%
^{2}$ by (\ref{ap3}). Thus, observing (\ref{qest19tris}), the application of
Gronwall's inequality to (\ref{qest21}) yields the desired uniform bound%
\begin{equation}
U_{n}\in L^{\infty }\left( 0,T;\mathbb{V}^{1}\right) .  \label{qest22}
\end{equation}%
Finally, by comparison in equation (\ref{qest9}), by virtue of the uniform
bounds (\ref{qest22}) and (\ref{stronger-Phi}) we also deduce%
\begin{equation}
\mathrm{T}\Phi _{n}^{t}\in L^{\infty }\left( 0,T;\mathcal{M}_{\Omega ,\Gamma
}^{1}\right)  \label{qest23}
\end{equation}%
uniformly with respect to all $n\geq 1$. In particular, it holds $\Phi
_{n}^{t}\in L^{\infty }\left( 0,T;D\left( \mathrm{T}\right) \right) $
uniformly. Finally, by (\ref{qest22}) and assumptions (i)-(ii), we also have%
\begin{equation*}
F\left( U_{n}\right) \in L^{\infty }\left( 0,T;\mathbb{X}^{2}\right) .
\end{equation*}%
We can pass to the limit as $n\rightarrow \infty $ in (\ref{stronger-U})-(%
\ref{stronger-Phi}), (\ref{qest22}) and (\ref{qest23}) to find a limit point
$\left( U,\Phi \right) $ with the properties stated in (\ref%
{eq:strong-defn-1})-(\ref{eq:strong-defn-4}). Passage to the limit in
equations (\ref{eq:approx-ode-1})-(\ref{eq:approx-ode-2}) and in particular,
in the nonlinear terms is done in the same fashion as in the proof of
Theorem \ref{quasi-strong}. Indeed, exploiting (\ref{qest22}) and (\ref%
{stronger-U}) we still have the validity of (\ref{qest12}) and, hence, the
limit solution $\left( U,\Phi \right) $ solves (\ref{eq:problem-p-2}) in the
sense of Definition \ref{d:strong-solution}.

Uniqueness follows from Proposition \ref{uniqueness} owing to assumption (%
\ref{eq:assumption-2}). The proof of theorem is now complete.
\end{proof}

Finally, we may conclude with the following.

\begin{theorem}
Let the assumptions of Theorem \ref{quasi-strong} be satisfied. Let $\left(
U,\Phi \right) $ be a unique strong solution corresponding to a given
initial datum $\left( U_{0},\Phi _{0}\right) \in \mathcal{H}_{\Omega ,\Gamma
}^{2,2}.$ Then, this solution also satisfies%
\begin{equation}
U\in L^{\infty }\left( 0,T;\mathbb{V}^{2}\right) ,\text{ }\partial _{t}U\in
L^{\infty }\left( 0,T;\mathbb{X}^{2}\right) ,\text{ }\partial _{t}\Phi \in
L^{\infty }\left( 0,T;\mathcal{M}_{\Omega ,\Gamma }^{1}\right) .
\label{dissip-a1bis}
\end{equation}%
Moreover, the equations in (\ref{eq:problem-p-2}) are satisfied in the
strong sense, that is, almost everywhere on $\left( 0,T\right) $ for $\Xi
\in \mathbb{X}^{2}$ and $\Pi \in \mathcal{M}_{\Omega ,\Gamma }^{0}$.
\end{theorem}

\begin{proof}
First, we note that $\left( U_{0},\Phi _{0}\right) \in \mathcal{H}_{\Omega
,\Gamma }^{2,2}$ is equivalent to $\left( U_{0},\Phi _{0}\right) \in
\mathcal{H}_{\Omega ,\Gamma }^{1,2}$ and $U_{0}\in \mathbb{V}^{2}$. Thus, it
suffices to prove that the additional regularity (\ref{dissip-a1bis}) is
enjoyed by any strong solution of Theorem \ref{quasi-strong}. We recall that
the limit point $\left( U,\Phi \right) $ induced by (\ref{qest11-1})-(\ref%
{qest11-6}) solves (\ref{eq:problem-p-2}), and it also satisfies%
\begin{equation*}
\left\Vert \partial _{t}u\right\Vert _{L^{\infty }\left( 0,T;L^{2}\left(
\Omega \right) \right) }+\left\Vert \partial _{t}u\right\Vert _{L^{\infty
}\left( 0,T;L^{2}\left( \Gamma \right) \right) }\leq C_{T},
\end{equation*}%
thanks to (\ref{qest17})-(\ref{stronger-Phi}) and the fact that $U_{0}\in
\mathbb{V}^{2}$. It follows that $\left( U,\Phi \right) $ also satisfies the
following elliptic system%
\begin{equation}
\left\{
\begin{array}{ll}
-\omega \Delta u+\alpha \omega u+f\left( u\right) =H_{\Omega }, & \text{a.e.
in }\Omega \times \left( 0,T\right) , \\
-\nu \Delta _{\Gamma }u+\omega \partial _{n}u+\nu \beta +\widetilde{g}\left(
u\right) =H_{\Gamma }, & \text{a.e. in }\Gamma \times \left( 0,T\right) ,%
\end{array}%
\right.  \label{qest13}
\end{equation}%
where we have set%
\begin{equation*}
\binom{H_{\Omega }}{H_{\Gamma }}:=\binom{\omega \int_{0}^{\infty }\mu
_{\Omega }(s)\left( -\Delta \eta ^{t}\left( s\right) +\alpha \eta ^{t}\left(
s\right) \right) ds-\partial _{t}u}{\nu \int_{0}^{\infty }\mu _{\Gamma }(s)%
\mathrm{C}\xi ^{t}\left( s\right) ds-\partial _{t}u}.
\end{equation*}%
We now observe that in view of (\ref{qest11-3}), we have $\Phi ^{t}=\left(
\eta ^{t},\xi ^{t}\right) ^{\mathrm{tr}}\in L^{\infty }\left( 0,T;L_{\mu
_{\Omega }}^{2}\left( \mathbb{R}_{+};\mathbb{V}^{2}\right) \right) ,$ and
therefore,%
\begin{equation}
\left\Vert H_{\Omega }\right\Vert _{L^{\infty }\left( 0,T;L^{2}\left( \Omega
\right) \right) }+\left\Vert H_{\Gamma }\right\Vert _{L^{\infty }\left(
0,T;L^{2}\left( \Gamma \right) \right) }\leq C_{T},  \label{qest13q}
\end{equation}%
owing to (\ref{stronger-U}) (see the scheme developed in Theorem \ref%
{t:strong-solutions}, cf. (\ref{ap1})-(\ref{ap3}) and (\ref{qest16})-(\ref%
{qest17})). Thus, we can apply a regularity result for the system (\ref%
{qest13}) from \cite[Lemma A.2]{MZ} to further deduce%
\begin{equation}
u\in L^{\infty }\left( 0,T;L^{\infty }\left( \Omega \right) \right) ,\text{
\textrm{tr}}_{\mathrm{D}}\left( u\right) \in L^{\infty }\left( 0,T;L^{\infty
}\left( \Gamma \right) \right) .  \label{qest13bis}
\end{equation}%
On the other hand, to prove the first bound on $U$ from (\ref{dissip-a1bis}%
), we need to apply Lemma \ref{t:appendix-lemma-3} with the obvious choices:%
\begin{equation*}
p_{1}:=H_{\Omega }-f\left( u\right) ,\text{ }p_{2}:=-\widetilde{g}\left(
u\right) +H_{\Gamma }.
\end{equation*}%
Owing to (\ref{qest13bis}) and once again to (\ref{qest13q}), it is not
difficult to realize that $p_{1}\in L^{\infty }\left( 0,T;L^{2}\left( \Omega
\right) \right) ,$ $p_{2}\in L^{\infty }\left( 0,T;L^{2}\left( \Gamma
\right) \right) ,$ which is enough to deduce%
\begin{equation}
u\in L^{\infty }\left( 0,T;H^{2}\left( \Omega \right) \right) ,\text{
\textrm{tr}}_{\mathrm{D}}\left( u\right) \in L^{\infty }\left(
0,T;H^{2}\left( \Gamma \right) \right) .  \label{qest13tris}
\end{equation}%
The final bound in (\ref{add-reg}) is an immediate consequence of (\ref%
{stronger-Phi}). The proof is finished.
\end{proof}

\begin{remark}
In Theorem \ref{t:strong-solutions}, since the initial datum $\left(
U_{0},\Phi _{0}\right) $\ belongs to $\mathbb{V}^{2}\times \left( \mathcal{M}%
_{\Omega ,\Gamma }^{2}\cap D\left( \mathrm{T}\right) \right) $ it would be
desirable to prove that%
\begin{equation}
U\in L^{\infty }\left( 0,T;\mathbb{V}^{2}\right) ,\text{ }\Phi \in L^{\infty
}\left( 0,T;\mathcal{M}_{\Omega ,\Gamma }^{2}\right) ,  \label{add-reg}
\end{equation}%
as well. Unfortunately, we cannot deduce (\ref{add-reg}) as in the proof of
Theorem \ref{quasi-strong}, because generally, $\mu _{\Omega }\neq \mu
_{\Gamma };$ see Remark \ref{no-cancel}.
\end{remark}

\end{document}